\documentclass[11pt, reqno]{amsart}
\usepackage{amsmath,amsfonts,amsthm,amsopn,amssymb}
\usepackage{amsmath}
\usepackage{verbatim,wasysym,mathrsfs}
\usepackage[left=2.6cm,right=2.6cm,top=3.1cm,bottom=3.1cm]{geometry}
\usepackage{cite,marginnote}
\usepackage{color,enumitem,graphicx}
\usepackage[colorlinks=true,urlcolor=blue, citecolor=red,linkcolor=blue,
linktocpage,pdfpagelabels, bookmarksnumbered,bookmarksopen]{hyperref}
\usepackage[english]{babel}
\usepackage[T1]{fontenc}
\usepackage{lmodern}
\usepackage[hyperpageref]{backref}

\numberwithin{equation}{section}

\makeindex
\newtheorem{theorem}{Theorem}[section]
\newtheorem{proposition}[theorem]{Proposition}
\newtheorem{lemma}[theorem]{Lemma}
\newtheorem{remark}[theorem]{Remark}
\newtheorem{example}[theorem]{Example}
\newtheorem{corollary}[theorem]{Corollary}
\newtheorem{definition}[theorem]{Definition}

\newcommand{\RN}{\mathbb R^N}

\newcommand{\iy}{\infty}

\newcommand{\s}{\section}

\newcommand{\DD}{\Delta}

\newcommand{\na}{\nabla}

\newcommand{\la}{\lambda}

\newcommand{\pa}{\partial}

\newcommand{\R}{\mathbb R}

\newcommand{\re}[1]{(\ref{#1})}

\newcommand{\C}{\mathbb C}

\newcommand{\rg}{\rightarrow}

\newcommand{\lan}{\langle}
\newcommand{\ran}{\rangle}
\newcommand{\e}{\varepsilon}
\newcommand{\vp}{\varphi}

\newcommand{\lab}{\label}
\newcommand{\bt}{\begin{theorem}}
\newcommand{\et}{\end{theorem}}
\newcommand{\bl}{\begin{lemma}}
\newcommand{\el}{\end{lemma}}
\newcommand{\bd}{\begin{definition}}
\newcommand{\ed}{\end{definition}}
\newcommand{\bc}{\begin{corollary}}
\newcommand{\ec}{\end{corollary}}
\newcommand{\bp}{\begin{proof}}
\newcommand{\ep}{\end{proof}}
\newcommand{\bx}{\begin{example}}
\newcommand{\ex}{\end{example}}
\newcommand{\bi}{\begin{exercise}}
\newcommand{\ei}{\end{exercise}}
\newcommand{\bo}{\begin{proposition}}
\newcommand{\eo}{\end{proposition}}
\newcommand{\br}{\begin{remark}}
\newcommand{\er}{\end{remark}}
\newcommand{\be}{\begin{equation}}
\newcommand{\ee}{\end{equation}}
\newcommand{\ba}{\begin{align}}
\newcommand{\ea}{\end{align}}
\newcommand{\bn}{\begin{enumerate}}
\newcommand{\en}{\end{enumerate}}
\newcommand{\bg}{\begin{align*}}
\newcommand{\bcs}{\begin{cases}}
\newcommand{\ecs}{\end{cases}}

\def\C{\mathbb C}
\def\R{\mathbb R}

\def\N{\mathbb N}

\def\Proof{\noindent{\bf Proof}\quad}
\def\qed{\hfill$\square$\smallskip}
\makeatletter
\def\@makefnmark{}
\makeatother

\newcommand{\pl}{\partial}

\newcommand{\bean}{\begin{eqnarray*}}
\newcommand{\eean}{\end{eqnarray*}}

\renewcommand{\triangle}{\Delta}
\renewcommand{\epsilon}{\varepsilon}
\title[The logarithmic Choquard equation]
{Orbital stability of positive standing waves to the logarithmic Choquard equations}

\author[Z. S. Liu]{Zhisu Liu}
\author[S. Sun]{Shu-Ming Sun}
\author[W. N. Zhang]{Weinian Zhang}

\address[Zhisu Liu]{\newline\indent
School of Mathematics and Physics, China University of Geosciences,
\newline\indent
Wuhan, Hubei, 430074, P.R. China\\
\newline\indent
Institute for Advanced Marine Research, China University of Geosciences,
\newline\indent
Guangzhou 511462, P. R. China}
\email{\href{liuzhisu@cug.edu.cn}{liuzhisu@cug.edu.cn}}
\address[Shu-Ming Sun]{\newline\indent
Department of Mathematics, Virginia Tech, Blacksburg, VA 24061,
\newline\indent
U. S. A.}
\email{\href{sun@math.vt.edu}{sun@math.vt.edu}}
\address[Weinian Zhang]{\newline\indent School of Mathematics,
 Sichuan University,
\newline\indent
Chengdu, Sichuan, 610065, P.R. China}
\email{\href{matzwn@126.com}{matzwn@126.com}}

\thanks{Z. Liu is supported by the NSFC (No.12571172) and Guangdong Basic and Applied Basic Research
Foundation (No.2024A1515012704).
W. Zhang is supported by the National Key R\&D Program of China (2022YFA1005900)
and the NSFC (Nos. 12571180 and 12226325).
}


\begin{document}
\allowdisplaybreaks

\begin{abstract}
The paper concerns the orbital stability of positive standing waves for the logarithmic Choquard equation with a
small local nonlinear perturbation. First, the uniqueness and nondegeneracy of the positive ground-state solution are established
via bifurcation and compactness arguments. Then, by exploiting the spectral properties of the linearized operator
around the positive ground-state solution in high-energy spaces, the orbital stability of the corresponding positive
standing wave is proved.
The main difficulty in establishing orbital stability stems from the fact that the logarithmic convolution kernel is both
sign-changing and strongly singular, while the norm used to prove global existence is
not translation invariant. Consequently, existing approaches to orbital stability for positive ground-state solutions of three- and higher-dimensional Choquard equations with Riesz-type kernels cannot be applied to the two-dimensional logarithmic Choquard equation considered here.
\end{abstract}
\maketitle

	\maketitle

\parskip 0.2cm

%%

%%

%%%%%%%%%%%%%%%%%%%%%%%%%%%%%%%%%%%%%%%%%%%%%%%%%%%%%%%%%%%%%%%%%%%%%%%

\section{Introduction}
\label{sec1}

%%
%\footnote{[ZWN241022]
%It is better to state those equations, formulas and definition of $\ast$ here for beginning.
%Never state those math formulas and notions too much in the Abstract.
%In the Abstract we need to explain in human language.
%}
%%

%%
%\footnote{[ZWN241022] Where did you define te concepts
%``orbital stability'' ``positive standing wave''?
%}
%%

%%%=====================================

This paper discusses the long-time behavior of solutions for an initial value problem of logarithmic Choquard equation
\begin{equation}\lab{eqn:q1}
\left\{
\begin{aligned}
    &i\pl_t\psi=-\DD \psi-\bigg(\ln\frac{1}{|x|}\ast|\psi|^2\bigg)\psi-\nu|\psi|^{p-1}\psi, & \quad t>0,\,\,\,x\in\R^2, \\
    &\psi(x,0)=\psi_0(x), & \quad x\in\R^2,
  \end{aligned}
\right.
\end{equation}
where $\ast$ denotes the convolution and $\nu$ is a fixed real constant. This problem arises from  two-dimensional
Schr\"{o}dinger-Newton system
\begin{equation}\lab{eqn:q2}
 \left\{
\begin{aligned}
    &i\pl_t\psi=-\DD \psi-\phi \psi-\nu|\psi|^{p-1}\psi, & \quad t>0,\,\,\,x\in\R^2, \\
    &-\DD\phi=2\pi|\psi|^2,\\
    &\psi(x,0)=\psi_0(x), & \quad x\in\R^2,
  \end{aligned}
\right.
\end{equation}
 where $\psi:[0,\infty)\times\R^2\rightarrow\C$, and $\phi: [0,\infty)\times\R^2\rightarrow\R$ is the Newtonian
potential given by
$$
\phi=\ln\frac{1}{|x|}\ast|\psi|^2.
$$
%%%
%[[[{\color{red}[ZWN250611] s?
%How many "dimensions"?ÕâÀïÒÑŸ­ÐÞžÄ¡£
%}
%]]]
%%%

The Schr\"odinger-Newton system is a nonlinear Schr\"odinger equation with a Newtonian gravitational potential,
where the potential arises from the interaction of a particle with its own gravitational field,
leading to a fundamental modification of quantum mechanics.
The system can also be viewed in connection with self-gravitating boson stars,
as well as  a model for explaining quantum wave function collapse. Moreover, a two-dimensional Schr\"odinger-Newton system \re{eqn:q2}, i.e., the logarithmic Choquard equation,  may model quantum bosons with logarithmic interaction.
When the spatial dimension is three, the system corresponding to \re{eqn:q2} was used by Choquard to describe an
electron trapped in its own hole within a certain approximation of the Hartree-Fock theory for a one-component plasma, as discussed by Lieb in \cite{lieb}. Moreover, such systems frequently appear in models of quantum mechanics and semiconductor theory (see \cite{Lions87,Lieb77}), and are also considered, for example, as models for the interaction of a charged particle with an electrostatic field (see \cite{Benci98}).

%%%%%%%%%%%%%%%%%%%%%%%%%%%%%%%%%%%%%%%%%%%%%%%%

In this paper, we focus only on  mathematical properties of solutions for (\ref{eqn:q1}).  Similar to the study of nonlinear Schr\"odinger equations, the first fundamental question is whether (\ref{eqn:q1}) admits \emph{standing-wave} solutions,
namely,
solutions of the form $\psi(t,x)=e^{i\omega t}u(x)$, where
 $\omega\in\R^+$ and $u:\R\rightarrow \R^+$ is a positive function.
Substituting this ansatz into (\ref{eqn:q1}) yields the equation
\begin{equation}\label{eqn:2sn}
-\DD u+\omega u =\bigg(\ln\frac{1}{|x|}\ast|u|^2\bigg)u+\nu|u|^{p-1}u\quad \mbox{in}\,\, \R^2.
\end{equation}
The existence of localized solutions (also called bound states) of  (\ref{eqn:2sn})
was investigated numerically
in \cite{Harrison03, Choquard08}.  Formally, \re{eqn:2sn} has a variational
structure with the Euler-Lagrange equation related to the energy functional
\begin{align}
S_{\omega,\nu}(u)=\frac{1}{2} \int_{\mathbb{R}^{2}}\left(|\nabla u|^{2}+\omega |u|^{2}\right) {\rm d} x
-\frac{1}{4} \int_{\mathbb{R}^{4}}
\ln\left(\frac{1}{|x-y|}\right) |u(x)|^2 |u(y)|^2 {\rm d} x {\rm d} y-\frac{\nu}{p+1}\int_{\R^2}|u|^{p+1}{\rm d} x, \label{1.3.1}
\end{align}
where the nonlocal term (second term in \eqref{1.3.1}) is not defined on the classical Sobolev space $H^1(\R^2)$.
Stubbe \cite{Stubbe08} introduced a variational framework to (\ref{eqn:2sn}) by using
a weighted Sobolev space
\begin{align}
X:= \left\{u\in H^1(\R^2):\int_{\R^2}\ln(1+|x|)|u(x)|^{2}{\rm d}x<+\infty\right\}\label{1.3.2}
\end{align}
 with a norm
$$
\|u\|_{X}:=\left[\int_{\mathbb{R}^{2}}\left(|\nabla u|^{2}+\omega|u|^{2}\right)
{\rm d}x+\int_{\mathbb{R}^{2}} \ln (1+|x|)|u(x)|^{2} {\rm d}x\right]^{\frac12},
$$
which guarantees that the associated energy functional is well-defined
and continuously differentiable in $X$.

To study  (\ref{eqn:2sn}) in $X$, we note that the norm $\|\cdot\|_X$ is not translation invariant, whereas the functional $S_{\omega,\nu}$ is. Moreover, the quadratic part of the functional $S_{\omega,\nu}$ after linearization is never coercive in $X$ for any $\omega>0$ due to the presence of the nonlocal term.
%This makes problems tough in verifying the compactness via the concentration-compactness principle.
By using a symmetric bilinear form
$$
\int_{\mathbb{R}^{2}} \int_{\mathbb{R}^{2}} \ln\left(1+|x-y|\right) u(x) v(y) {\rm d} x {\rm d} y,
$$
Cingolani and Weth \cite{Cingolani16} overcame these difficulties
and obtained a sequence of high-energy solutions to (\ref{eqn:2sn})
with a general local nonlinearity, as well as the uniqueness of positive solution to \re{eqn:2sn} with $\nu=0$ using moving-plane method.
Based on the uniqueness results in \cite{Choquard08,Cingolani16},
the non-degeneracy property
of the positive solution to \re{eqn:2sn} with $\nu=0$ was proved  in \cite{Bonheure17} using the multipole expansion of the logarithm kernel. Moreover, many attempts have been made
to address the difficulties arising from the sign-changing nature of certain interactions, such as Coulomb or logarithmic interactions, in more general nonlinear equations or systems
(see  \cite{Luca18,Albuquerque21,Bucur22,Cassani21,Cassani25,Chen20,Cingolani19,Azzollini21,
Liu22,Cassani23,Pucci24,Guo23,Dolbeault21,Shen24,Cabrera25}
 and the references therein).

After establishing the existence of localized (or solitary-wave) solutions of \eqref{eqn:2sn}, the next important question is the stability of these solutions with respect to the time-dependent equation \eqref{eqn:q1}. Due to certain invariances of \eqref{eqn:q1}, the stability is referred
%%
%\footnote{[ZWN260604] referred}
%%
to as orbital stability, i.e., if the initial datum $u_0$ of the associated Cauchy problem (\ref{eqn:q1})
is close to a solitary-wave solution of (\ref{eqn:2sn}), then the corresponding evolution $u(x,t)$ of (\ref{eqn:q1}) remains close
to the orbit generated by some invariants (such as translation, rotation, etc.) associated with the solitary-wave solution. In particular, to study
the dynamical behavior
%%
%\footnote{[ZWN260604]
%``the dynamical behavior'' should be ``dynamical behaviors''?}
%%
of
solutions to the nonlinear Schr\"{o}dinger equation,  namely \eqref{eqn:q1} without the nonlocal term,
it is shown in \cite{Berestycki81,Lions,Wein0} that
the unique positive radially symmetric standing-wave solution $u_\omega e^{i\omega t}$ of
the initial value problem
\begin{equation}\label{eqn:2sn-}
\left\{
\begin{aligned}
    &i\psi_t+\triangle {\psi}+|\psi|^{p-1}\psi=0,&\quad x\in\R^N,\\
    &\psi(x,0)=\psi_0(x), & \quad x\in\R^N,
  \end{aligned}
\right.
\end{equation}
given in \cite{Kwong89},
where $1<p<2^*-1$ and $2^*=\infty$ for $N=1,2$ or $2^*=\frac{2N}{N-2}$ for $N\geq3$,
is either stable if $p<1+4/N$
or unstable if $p\geq 1+4/N$.
Stability proof is based on variational and compactness arguments, while
instability proof is based on variational characterizations related to the virial identity.
Another approach was developed in \cite{Wein1,Wein2,Grillakis87} to show that the
standing-wave solution $u_\omega e^{i\omega t}$
is either stable if $\partial_{\omega}\|u_\omega\|_2^2>0$ or unstable if $\partial_{\omega}\|u_\omega\|_2^2<0$. If $u_1\in H^1(\R^N)$ is
the unique positive solution of the equation $-\triangle {u}+u=|u|^{p-1}u$,
then the function
%%
%\footnote{[ZWN260604] `the function' is added}
%%
\begin{equation}\label{eqn:binahuan}
u_\omega(x)=\omega^{\frac{1}{p-1}}u_1 \left (\sqrt{\omega}x\right ),
\end{equation}
for $\omega > 0$ is also a solution and represents a kind of scaling symmetric invariant.
After calculating $\partial_{\omega}\|u_\omega\|_2^2$, it is
obtained that
\begin{equation}\label{eqn:dengjia}
\aligned
&p<1+4/N \quad\text{if\,and\,only\,if}\quad \partial_{\omega}\|u_\omega\|_2^2>0, \,\text{and}\\
&p>1+4/N \quad\text{if\,and\,only\,if}\quad \partial_{\omega}\|u_\omega\|_2^2<0.
\endaligned
\end{equation}

For the three-dimensional version of Choquard equation (\ref{eqn:q1}) with nonlocal nonlinear term only, Cazenave and Lions \cite{Lions} considered the following problem
\begin{equation}\lab{eqn:q1+3}
\left\{
\begin{aligned}
    &i\pl_t\psi=-\DD \psi-\bigg(\frac{1}{|x|}\ast|\psi|^2\bigg)\psi, & \quad t>0,\,\,\,x\in\R^3, \\
    &\psi(x,0)=\psi_0(x), & \quad x\in\R^3,
  \end{aligned}
\right.
\end{equation}
and applied constrained variational and compactness arguments to prove the orbital
stability of standing-wave solutions $u_\omega e^{i\omega t}$ for any $\omega>0$. In particular, in their proof,
the uniqueness of positive standing-wave solutions $u_\omega e^{i\omega t}$ and
the scaling symmetry of (\ref{eqn:q1+3})
play essential roles to obtain the classification of positive
solutions for some minimization problem with prescribed $L^2$-norm.
With the help of such classification,
the orbital stability of standing-wave solutions for any fixed frequency $\omega > 0$ can be proved.
In addition,
%%
%\footnote{[ZWN260604] typo is corrected}
%%
d'Avenia and Squassina \cite{Marco1} studied the solution dynamics
of \re{eqn:q1+3} with external potential and established suitable
modulational stability estimates in the spirit of those
obtained by Weinstein in \cite{Wein1,Wein2}.
We remark that there are also some stability results (see \cite{Bellazzini11,Bellazzini13,Kikuchi07,Colin19} and references therein)
regarding the standing-wave solutions of \re{eqn:q1+3} with additional local nonlinearity.

%%%%%%%%%%%%%%%%%%%%%%%%%%%%%%%%%%%%%%%%%%%%%%%%%%%%%%%%%%%%%%%%

In this paper, we mainly investigate the orbital stability
%%
%\footnote{[ZWN260604]{\color{red} This notion is not matched with the title of the paper and Prof. Sun's idea in the message.
%I suggest}:
%In this paper, we study the orbital stability of standing-wave solutions for  \re{eqn:q1} with both local and nonlocal terms and $\nu\geq0$ small.
%In particular, to the best of our knowledge, there are no theoretical results on the orbital stability of standing-wave solutions for  \re{eqn:q1} even in the case $\nu = 0$.
%To study the orbital stability, we first establish results on the uniqueness and nondegeneracy
%of ground-state solutions to (\ref{eqn:2sn}) when the positive parameter $\nu$ is small.
%Indeed, contrary to ...
%}
%%
of standing-wave solutions for  \re{eqn:q1}, which contains both local and nonlocal terms for small $\nu\geq0$. In particular, to the best of our knowledge, no theoretical results on the orbital stability of standing-wave solutions to \re{eqn:q1} are available, even in the case $\nu = 0$.
To study the orbital stability, we first need to establish results on the uniqueness and nondegeneracy
of ground-state solutions to (\ref{eqn:2sn}) for small positive $\nu$.
In contrast to the existence theory, very little is known about these issues for logarithmic Choquard equations. Despite recent progress in several special
cases \cite{Choquard08,Bonheure17,Cingolani16}, many fundamental questions are still open. Our result is stated as follows.
%%
%\footnote{[ZWN260604] I simplified the statement of Theorem 1.1. Please check}
%%
\bt\lab{Th1}
For $p\in(1,+\infty)$, there exists $\nu^*>0$ such that for any $\nu\in(0,\nu^*)$, {equation (\ref{eqn:2sn})} has
 a unique positive radially symmetric ground-state solution $Q_{\omega,\nu}\in X$, and
 the linearized operator of  (\ref{eqn:2sn}) at $Q_{\omega,\nu}$,  defined by
$$L_{+,\nu}\vp:=-\DD\vp+\omega\vp- \left (\ln\frac{1}{|x|}\ast Q_{\omega,\nu}^2\right )\vp
-2\left (\ln\frac{1}{|x|}\ast Q_{\omega,\nu}\vp\right )Q_{\omega,\nu}-p\nu Q_{\omega,\nu}^{p-1}\vp,\,\,\vp\in \tilde{X}\, ,
$$
is nondegenerate in the sense that
$$
\mbox{ker} \, L_{+,\nu}=span\left \{\partial_{x_1}Q_{\omega,\nu},\partial_{x_2}Q_{\omega,\nu}\right \},
$$
where
$\mbox{ker}\,  L_{+,\nu}$ represents the kernel of $L_{+,\nu}$
%, the linearized operator of
% (\ref{eqn:2sn}) at $Q_{\omega,\nu}$ defined by
%$$L_{+,\nu}\vp:=-\DD\vp+\omega\vp- \left (\ln\frac{1}{|x|}\ast Q_{\omega,\nu}^2\right )\vp
%-2\left (\ln\frac{1}{|x|}\ast Q_{\omega,\nu}\vp\right )Q_{\omega,\nu}-p\nu Q_{\omega,\nu}^{p-1}\vp,\,\,\vp\in \tilde{X}\, ,
%$$
and
$\tilde{X}$ is the space defined by (see \cite{Bonheure17})
\begin{equation}\label{eqn:X+}
\aligned
\tilde{X}:=\bigg\{\vp\in X:\,\, &\text{there\, exists}\,f\in L^2(\R^2)\,\text{such\, that}\\
& \int_{\R^2}\vp L_{+,\nu}\psi{\rm d}x=\int_{\R^2}f\psi{\rm d}x,~ \forall\, \psi\in C_c^\iy(\R^2)\bigg\}.
 \endaligned
\end{equation}
In addition, the operator $L_{+,\nu}$ has exactly one negative eigenvalue and hence has
Morse index one.
\et
\begin{remark}
Note that the Morse
index of $L_{+,\nu}$ is the number of negative eigenvalues (counted with multiplicity) of the
self-adjoint operator $L_{+,\nu}$ acting on $X$.
\end{remark}

The ground-state solution $Q_{\omega,\nu}$ of (\ref{eqn:2sn}) was obtained in \cite{Du17},
and the critical value of the functional $S_{\omega,\nu}$ in \eqref{1.3.1} at $Q_{\omega,\nu}$
can be estimated uniformly in $\nu$ via a mountain-pass characterization.
It can be shown that $Q_{\omega,\nu}$ is in $X$ uniformly for small $\nu>0$.
A variant of the moving plane method (see \cite{Cingolani16}) implies that every positive solution
in $X$ is radially symmetric up to translation and strictly decreasing with respect to the distance from the center of symmetry. Together with Radial Lemma in \cite{Berestycki81} and a comparison principle, this yields that
 $Q_{\omega,\nu}$ decays exponentially to zero at infinity. Hence, $Q_{\omega,\nu}$
converges strongly in $X$ to $Q_\omega$ as $\nu\rightarrow0^+$, and the implicit function
theorem can be applied to prove the uniqueness
of the ground state $Q_{\omega,\nu}$ for small $\nu>0$, using the fact that $Q_\omega$ is the unique positive solution of equation (\ref{eqn:2sn}) with $\nu=0$.
The proof of nondegeneracy is again based on the strong convergence of $Q_{\omega,\nu}$ in $X$.
By analyzing eigenvalue problem for the linear operator $L_{+,\nu}$,
it can be shown that the Morse index of $L_{+,\nu}$ equals $1$ and $\mbox{ker}\,  L_{+,\nu}=span\{\zeta\cdot \nabla Q_{\omega,\nu}:
 \,\zeta\in\R^2\}$, using a compactness perturbation argument from the case $\nu=0$.
 A similar analysis has been carried out for
other problems (see \cite{Lenzmann09,Fall14,Xiang16,Bonheure18} and the references therein).
\begin{remark}
Theorem \ref{Th1} can be seen as a generalization
of the uniqueness and nondegeneracy results for (\ref{eqn:2sn}) with $\nu=0$ obtained
in \cite{Choquard08,Bonheure17,Cingolani16} to the case of small $\nu>0$.
\end{remark}

%%%%%%%%%%%%%%%%%%%%%%%%%%%%%%%%%%%%%%%%%%%%%%%%%%%%%%%%%%%%%%%%%%%%

Theorem \ref{Th1}  plays an  important role in studying the orbital stability of ground-state solutions
for the time-dependent Choquard equation (\ref{eqn:q1}), serving as a spectral assumption for the linearized operator.
In general, to study stability,  the global well-posedness of the initial value problem is
a necessary condition for any qualitative analysis of long-time dynamics.
The global well-posedness result of  (\ref{eqn:q1}) was established in \cite{Masaki11},
where the nonlocal nonlinearity is decomposed into $\|\psi\|_2^2(\ln\langle x\rangle)$ and $\int_{\R^2}
\left(\ln\frac{|x-y|}{\langle x\rangle}\right)|\psi(y)|^2{\rm d}y$,
that is,
$$
\bigg(\ln\frac{1}{|x|}\ast|\psi|^2\bigg)\psi=-
\|\psi\|_2^2(\ln\langle x\rangle)\psi-\psi\int_{\R^2}
\left(\ln\frac{|x-y|}{\langle x\rangle}\right)|\psi(y)|^2{\rm d}y,
$$
and $\langle x\rangle=(1+|x|^2)^{1/2}$.
Then, the Cauchy problem (\ref{eqn:q1}) can be rewritten as
\begin{equation}\lab{eqn:q2+}
\left\{
\begin{aligned}
    &i\pl_t\psi+\DD \psi-m\ln\langle x\rangle\psi
    =-\int_{\R^2}\bigg(\ln\frac{|x-y|}{\langle x\rangle}\bigg)|\psi|^2{\rm d}y\psi+\nu|\psi|^{p-1}\psi, &\quad t>0,\,\,\,x\in\R^2, \\
    &\psi(x,0)=\psi_0(x), & \quad x\in\R^2,
  \end{aligned}
\right.
\end{equation}
where $m:=\|\psi_0\|_2^2$.
Since (\ref{eqn:q2+}) has a linear part
$
i\pl_t\psi+\DD \psi-m\ln\langle x\rangle\psi,
$
a natural choice of the function space \cite{Masaki11} is %the following
%%
%\footnote{[ZWN260604] I cancel
%`` space,''
%}
%%
$$
\aligned
&\mathcal{H}:=\{u\in H^1(\R^2,\mathbb{C}):\,\int_{\R^2}\ln\langle x\rangle |u|^2{\rm d}x<+\infty\},\quad
\|u\|_{\mathcal{H}}:=\|u\|_{H^1(\R^2,\mathbb{C})}+\left \|\sqrt{\ln\langle \cdot\rangle}u \right \|_{L^2(\R^2,\mathbb{C})}.
\endaligned
$$
This means that  the initial condition must have some kind of decay at infinity. It is straightforward to check  that the space $\mathcal{H}$ is equivalent to $X(\R^2,\mathbb{C})$, which is similar to $X = X(\R^2,\mathbb{R})$ defined in \eqref{1.3.2}.
It has been shown in \cite{Masaki11} that  (\ref{eqn:q1})
is globally well-posed in $X(\R^2,\mathbb{C})$ either for $\nu=0$, or for $p\in[2,3)$ and $\nu>0$, while some local well-posedness results can be found in \cite{Masaki10,Sun24}.
Note that the solution of (\ref{eqn:q1}) perserves $\|u(t)\|_2^2$ and the energy
\begin{align}
\mathcal{E}_\nu(u(t)):=\frac{1}{2}\|\nabla u(t)\|_2^2
-\frac{1}{4}\int_{\R^4}\ln\left(\frac{1}{|x-y|}\right)|u(x,t)|^2|u(y,t)|^2{\rm d}x{\rm d}y-\frac{\nu}{p+1}\int_{\R^2}|u|^{p+1}{\rm d} x.\label{1.9.1}
\end{align}

\begin{definition} (see \cite{Cazenave03})
%%
%\footnote{[ZWN260604] Do we need a bracket $()$ for \cite{Cazenave03}?
%}
%%
A standing-wave solution $e^{i\omega t}u_\omega$ of equation (\ref{eqn:q1}) is said to be {\it orbitally stable}  if for any $\epsilon>0$,
there exists $\delta>0$ such that for each $u_0\in X(\R^2,\mathbb{C})$ satisfying that
$\|u_0-u_\omega\|_{X(\R^2,\mathbb{C})}<\delta$,
the solution $u(t)$ of (\ref{eqn:q1}) with $u(0)=u_0$ exists for any $t \geq 0$ and satisfies
$$
\sup\limits_{t\geq0}\inf\limits_{\theta\in\R,y\in\R^2}\|u(t)-e^{i\theta}u_\omega(\cdot-y)\|_{X(\R^2,\mathbb{C})}<\epsilon.
$$
Otherwise, $e^{i\omega t}u_\omega$ is said to be {\it unstable}.
\end{definition}

Comparing with (\ref{eqn:2sn-}) and (\ref{eqn:q1+3}), equation (\ref{eqn:2sn}) does not have
the similar scaling symmetric invariant as (\ref{eqn:binahuan}) even for $\nu=0$.
Specifically, if $Q_\omega$ is the unique positive solution of \re{eqn:2sn} with $\nu=0$, then for any $a,b\in\R$, the function
%%
%\footnote{[ZWN260604] I added ``the function''}
%%
 $Q(x):=\omega^{a}Q_\omega (\omega^{b}x)$ solves
$$
-\omega^{-2b}\DD Q+\omega Q =\omega^{2(b-a)}\left[\ln\frac{1}{\omega^b}\|Q\|_2^2Q+\bigg(\ln\frac{1}{|x|}\ast|Q|^2\bigg)Q\right]\quad \mbox{in}\,\, \R^2,
$$
which can be simplified as
$$
-\DD Q+Q =\frac{1}{\omega}\left[\frac{1}{2}(\ln\omega ) \|Q\|_2^2Q+\bigg(\ln\frac{1}{|x|}\ast|Q|^2\bigg)Q\right]\quad
\mbox{in}\,\, \R^2,
$$
if $a=b=-\frac12$. Note that here the solution $Q$ still depends on $\omega$. The above analysis shows that it is difficult to determine the
sign of $\partial_{\omega}\|u_\omega\|_2^2$ in a way similar to (\ref{eqn:dengjia}),
even when \re{eqn:2sn} with $\nu=0$ is considered. This difficulty indicates that the abstract theorem established in \cite{Grillakis87} cannot be directly applied to obtain the orbital stability of the solutions obtained in Theorem \ref{Th1} for the problem \eqref{eqn:q1}. Moreover,
the constrained variational approach in \cite{Lions} also seems to be inapplicable  to the stability problem for (\ref{eqn:q1}) even in the case $\nu=0$.

Inspired by the ideas in \cite{Wein1,Wein2,Grillakis87},
we are nevertheless able to establish the spectral properties of the linearized operator in the
high-energy space $X$ and prove the orbital stability of standing-wave solutions.
\bt\lab{Th3}
Given  $\omega\in(0,+\infty)$ and $p\in[2,3)$, there exists $\nu_*>0$ such that for any $\nu\in(0,\nu_*)$
the ground-state standing-wave solution $e^{i\omega t}Q_{\omega,\nu}$
of  (\ref{eqn:q1}) obtained in Theorem \ref{Th1} is orbitally stable in $X(\R^2,\mathbb{C})$.
\et

The proof of Theorem \ref{Th3} adapts the general framework developed in \cite{Wein1,Wein2,Grillakis87},
%%
%\footnote{\color{red}[ZWN260604] This notion may cause the editor/reviewer to think our paper does not have any original creation.
%While we say this, do we talk about our innovation as:
%Although The proof of Theorem \ref{Th3} [[[essentially follows the framework
 %in \cite{Wein1,Wein2,Grillakis87}, we encounter a substantial difficulty that ???????.
%We use ????? to overcome the difficulty}
%%
which
reduces the characterization of the orbital stability of the standing-wave solutions to establishing the coercivity of a bilinear form associated with
the second variational derivative of the Hamiltonian on a constrained function space.
By the nondegeneracy of $Q_{\omega,\nu}$ in Theorem \ref{Th1},
certain spectral properties of the linearized operator associated with
the second variational derivative of the Lagrangian
in the weighted Sobolev space $X$ are obtained. These properties show that
the unstable directions are ruled out by constraints arising from two invariants and
the kernel of a skew-symmetric operator.
The orbital stability of the positive solution $Q_{\omega,\nu}$ of (\ref{eqn:2sn}) is then proved using conserved quantities together with a contradiction argument.
This approach requires new ideas to overcome the difficulties arising from the absence of a scaling symmetry similar to (\ref{eqn:binahuan}) and from
the weighted
Sobolev space $X(\R^2,\mathbb{C})$, which is not translation invariant and makes the quadratic part of the functional $S_{\omega, \nu}$ in \eqref{1.3.1} non-coercive.
In particular, to address these issues, it is important to use an a-priori boundedness constraint together with the translation invariance of the norm $\|\cdot\|_{H^1(\R,\mathbb{C})}$ and energy estimates for the Newtonian
potential, and then to estimate the differences of energy functional $\mathcal{E}_\nu$  between arbitrary solutions of \eqref{eqn:q1} and the standing-wave solution $Q_{\omega,\nu}$ (see
 Lemma \ref{lem-xiajie}).
We emphasize that the result presented in this paper appears to be the first stability result for  the problem (\ref{eqn:q1}).

\begin{remark}
When $\nu=0$, Theorem \ref{Th3} can be viewed as a two-dimensional generalization of the orbital stability of ground-state standing-wave solutions for  the three-dimensional Choquard equations (\ref{eqn:q1+3}) established in \cite{Lions}.
\end{remark}

This paper is organized as follows. In Section 2, some preliminary notations and results of
logarithmic Choquard equation in the weighted
Sobolev space $X$ are introduced.
Section 3 is devoted to the uniqueness and nondegeneracy of positive ground-state
solution $Q_{\omega,\nu}$ of (\ref{eqn:2sn}) with small $\nu>0$.
% We start by a minimum problem of
%energy $\mathcal{E}$ constrained at some prescribed $L^2$-norm manifold,
%which is of use in proving non-negativity of the linearized operator $L_+$ defined in \eqref{eqn:L+}.
%Subsequently, we establish spectral
%properties of the linearized operator evaluated at positive solution $Q_\omega$ and at last
%we state the proof of Theorem \ref{Th1}.
In Section 4,  the orbital stability of ground state
$Q_{\omega,\nu}$ for (\ref{eqn:q1}) with small  $\nu\geq 0$ is proved and the proof of Theorem \ref{Th3} is given. An existence and continuity of an inverse operator together with an integral inequality is proved in Appendix.

%%%
%\footnote{[ZWN241111] Can we state the following notations in next section. Never write the section of Introduction too long.
%}
%%

%%%%%%%%%%%%%%%%%%%%%%%%%%%%%%%%%%%%%%%%%%%%%%%%%%%%%%%%%%%%%%%%%%%%%%%%%%%%%%%%

\s{Notations and function spaces}

In this section, some notations and function spaces which will be used frequently in this paper will be presented.

Firstly, for  any $z\in\mathbb{C}$, $\bar{z}, \mathfrak{Re}(z), \mathfrak{Im}(z)$ are denoted as
the complex conjugate, real part, and imaginary part of $z$, respectively.
It is easy to check that $\mathfrak{Re}(\bar{z}w)=\mathfrak{Re}(z\bar{w})$ and $\mathfrak{Im}(\bar{z}w)=-\mathfrak{Im}(z\bar{w})$
for all $z,w\in\mathbb{C}$. Also, $c $ and $C$ are used to denote generic positive constants {throughout the paper.}
$S_{\omega}$ and $\mathcal{E}$ are denoted as $S_{\omega,0}$
and $\mathcal{E}_0$, respectively, where $S_{\omega,\nu}$
and $\mathcal{E}_\nu$ are defined in \eqref{1.3.1} and \eqref{1.9.1}, respectively.

For function spaces, given $1\leq q \leq +\infty$, let $\|\cdot\|_q$ denote the usual
norm of the real (and complex) Lebesgue integrable
space $L^q(\R^2,\R)$ (and $L^q(\R^2,\C)$).
Let $(u,v)_{L^2}$ denote the inner product
of both $L^2(\R^2,\mathbb{C})$ and $L^2(\R^2)$.
Moreover, $H^1(\R^2,\R)$ and $H^1(\R,\mathbb{C})$ are  the real and complex Sobolev spaces, respectively, endowed with the inner product
$$
(u,v)_{H^1}:=\int_{\R^2}\left ( \nabla u\cdot \nabla \bar{v}+\omega u\bar{v}\right ) {\rm d}x,\,\,\mbox{for}\,\,\, u,v\in H^1(\R,\mathbb{C})
$$
with the norm $\|u\|^2_{H^1(\R,\mathbb{C})}:=\int_{\R^2}\left ( |\nabla u|^2+\omega |u|^2\right ) {\rm d}x$.
For the sake of simplicity,  $\|u\|^2_{H^1}$ is used for $\|u\|^2_{H^1(\R,\mathbb{C})}$ if $u$ is real.
%%%%%%%

In addition, denote $X=X(\R^2,\R)$ and $X(\R^2,\mathbb{C})$ as the real and complex Hilbert spaces, endowed with the norm
$$
\|u\|_{X(\R^2,\mathbb{C})}:=\left(\|u\|_{H^1(\R,\mathbb{C})}^2+\|u\|_w^2\right)^{\frac{1}{2}}
$$
for $u \in X(\R^2,\mathbb{C})$, where
$$
\aligned
%&\|u\|^2_{H^1(\R,\mathbb{C})}:=\int_{\R^2}|\nabla u|^2+\omega|u|^2{\rm d}x,\\
&\|u\|_w:=\left(\int_{\R^2}\ln(1+|x|)|u|^2{\rm d}x\right)^{1/2}\, .
\endaligned
$$
If $u\in X(\R^2,\mathbb{R})$, denote  $\|u\|_{X}$ for $\|u\|_{X(\R^2,\mathbb{R})}= \|u\|_{X(\R^2,\mathbb{C})}$.
Moreover, $X^\ast$ is the dual space of $X$ with the norm
$$
\|\nu\|_{X^\ast}:=\sup\left\{\left|\int_{\R^2}\vp(x)\nu\right|: \vp\in X,\,\,\|\vp\|_{X}\le1 \right\}\qquad \mbox{for}\quad \nu\in X^\ast.
$$

Next, we state some basic properties of the Newton kernel related to problem \re{eqn:2sn}.
From the notations introduced in \cite{Cingolani16},  define symmetric bilinear forms
$$
\aligned
&(u,v)\mapsto B_1(u,v)=\int_{\R^4}\ln(1+|x-y|)u(x)v(y){\rm d}x{\rm d}y,\\
&(u,v)\mapsto B_2(u,v)=\int_{\R^4}\ln\left(1+\frac{1}{|x-y|}\right)u(x)v(y){\rm d}x{\rm d}y,\\
&(u,v)\mapsto B_0(u,v)=B_1(u,v)-B_2(u,v)=\int_{\R^4}\ln(|x-y|)u(x)v(y){\rm d}x{\rm d}y,
\endaligned
$$
where $u,v:\R^2\rightarrow\R$ are measurable functions in the sense of Lebesgue measure space and $B_1, B_2$ have positive kernels.
For  $u\in X$,  the associated functionals are
$$
\aligned
&V_1(u)=B_1\left (u^2,u^2\right )=\int_{\R^4}\ln(1+|x-y|)u^2(x)u^2(y){\rm d}x{\rm d}y,\\
&V_2(u)= B_2\left (u^2,u^2\right )=\int_{\R^4}\ln\left(1+\frac{1}{|x-y|}\right)u^2(x)u^2(y){\rm d}x{\rm d}y,\\
&V_0(u)=B_0\left (u^2,u^2\right )=V_1(u)-V_2(u)=\int_{\R^4}\ln(|x-y|)u^2(x)u^2(y)\textcolor[rgb]{0.00,0.07,1.00}{{\rm d}x{\rm d}y.}
\endaligned
$$
By inequalities
$$
\ln(1+|x-y|)\leq \ln(1+|x|+|y|)\leq \ln(1+|x|)+\ln(1+|y|)\quad \text{for}\,\,\textcolor[rgb]{0.00,0.07,1.00}{x,y\in\R^2,}
$$
it is obtained that
\begin{equation}\label{eqn:es1}
\aligned
B_1(uv,wz)&\leq \int_{\R^4}[\ln(1+|x|)+\ln(1+|y|)]|u(x)v(x)||w(y)z(y)|{\rm d}x{\rm d}y\\
&  {\leq\|u\|_w\|v\|_w \|w\|_2\|z\|_2+\|u\|_2\|v\|_2\|w\|_w\|z\|_w} \quad \mbox{for any} \   u,v,w,z\in X.
\endaligned
\end{equation}
Moreover,  Hardy-Littewood-Sobolev inequality  \cite{Lieb01} gives that
\begin{equation}\label{eqn:es2}
|B_2(u,v)|\leq \int_{\R^4}\frac{u(x)v(y)}{|x-y|}{\rm d}x{\rm d}y\leq C\|u\|_{4/3}\|v\|_{4/3} \quad \mbox{for} \quad u,v\in X,
\end{equation}
or
\begin{equation}\label{eqn:es3}
|V_2(u)| = |B_2(u^2,u^2)| \leq C\|u\|^4_{8/3} \quad \mbox{for} \quad \textcolor[rgb]{0.00,0.07,1.00}{u\in X,}
\end{equation}
which yields that $V_2$ is  finite in $L^{8/3}(\R^2)$.

The following lemmas can be found in \cite{Cingolani16}.
\begin{lemma}\label{Lem:cw1}
Let $\{u_n\}$ be a sequence in $L^2(\R^2)$ such that $u_n\rightarrow u\in L^2(\R^2)\setminus\{
0\}$ pointwise almost everywhere in $\R^2$, and let $\{v_n\}$ be a bounded sequence
in $L^2(\R^2)$ such that $\sup_{n\in\N}B_1\left (u_n^2,v_n^2\right )<\infty$.
Then there exist $n_0\in\N$ and $C>0$ such that $\|v_n\|_w<C$ for $n\geq n_0$. Furthermore, if $u_n, v_n$ satisfy additional conditions,
$$
B_1\left (u_n^2,v_n^2\right )\rightarrow0\quad \text{and}\quad \|v_n\|_2\rightarrow0\quad \text{as}\,\,\textcolor[rgb]{0.00,0.07,1.00}{n\rightarrow\infty,}
$$
then $\|v_n\|_w\rightarrow0$ as $n\rightarrow\infty$.
\end{lemma}

\begin{lemma}\label{Lem:cw2}
 If $\{u_n\},\{v_n\},\{w_n\}$ are bounded sequences in $X$
such that $u_n\rightharpoonup u$ weakly in $X$, then
$B_1(v_n w_n,z(u_n-u))\rightarrow 0$ for every $z \in X$.
\end{lemma}
\begin{lemma}\label{Lem:compact0} The following facts hold on $X$.
\begin{itemize}
	
	\item [$(i)$] The space $X$ is compactly embedded in $L^{q}(\R^2)$ for all $q\in[2,+\infty)$.
	
	\item [$(ii)$] The functionals $V_0,V_1,V_2$ are of $C^1$ on $X$. Moreover, $V'_i(u)v=4B_i(u^2,uv), i =0, 1, 2,$ if
$u,v\in X$.
	
	\item [$(iii)$]  $V_2$ is continuously differentiable on $L^{\frac{8}{3}}(\R^2)$.

	\item [$(iv)$] $V_1$ is weakly lower semicontinuous on $H^1(\R^2)$.
	\item [$(v)$] If $\{\phi_n\}\subset X$ is  bounded, there exists $\phi\in X$ such that, up to a subsequence,
$\phi_n\rightharpoonup\phi$ weakly in $X$ and the convergence is uniform on any compact set.
\end{itemize}
\end{lemma}

%%%%%%%%%%%%%%%%%%%%%%%%%%%%%%%%%%%%%%%%%%%%%%%%%%%%%%%%%%%%%%%%%%%%%%%%%%%%%%%%%%%%%%%

\s{Uniqueness and nondegeneracy of ground-state solutions}
%Let us define the energy functional $S_{\omega,\nu}(u)$ associated with the perturbed
%equation (\ref{eqn:2sn-small}) as follows
%$$
%S_{\omega,\nu}(u):=\frac{1}{2} \int_{\mathbb{R}^{2}}\left(|\nabla u|^{2}+\omega u^{2}\right) {\rm d} x
%-\frac{1}{4} \int_{\mathbb{R}^{4}}
%\ln\left(\frac{1}{|x-y|}\right) u^{2}(x) u^{2}(y) {\rm d} x {\rm d} y-\frac{\nu}{p+1}\int_{\R^2}|u|^{p+1}{\rm d} x,
%$$
%where $u\in X$ and $\nu\in(0,+\infty)$.

This section gives the properties of positive ground-state solution $ Q_{\omega , \nu } \in X$ for \eqref{eqn:2sn}.

It is known that a ground-state solution of \eqref{eqn:2sn} is a solution $u_0(x)$ of \eqref{eqn:2sn} which has a minimum energy, i.e.,
$$
c_{g}:=\inf  \left \{S_{\omega,\nu}(u):\,\,u\in X\setminus\{0\}\,\text{is\,a\,solution\,of\,(\ref{eqn:2sn})}\right \}= S_{\omega, \nu }( u_0) \, ,
$$
where $S_{\omega,\nu}(u)$ is defined in \eqref{1.3.1}. It has been shown in \cite{Choquard08} that if $\nu = 0$, the positive solution of \eqref{eqn:2sn} is unique. Then, we have the following lemma.
\begin{lemma}\label{lemma3.1}
For $p>1$ and $0 \leq  \nu\leq 1$ fixed,  (\ref{eqn:2sn})
has at least one positive ground-state solution $Q_{\omega,\nu}\in X$ satisfying $S_{\omega, \nu }( Q_{\omega,\nu} ) = c_g$. Moreover, when $\nu \rightarrow 0^+$, any positive ground-state solution $Q_{\omega,\nu}\rightarrow Q_{\omega, 0} = Q_{\omega}\in X$, where $ Q_{\omega}$ is the unique positive solution of \eqref{eqn:2sn}  with $\nu = 0$.
\end{lemma}

\begin{proof}
 It was shown in \cite{Du17} that for $p>1$
and $\nu>0$ fixed,  (\ref{eqn:2sn})
has a positive ground-state solution $Q_{\omega,\nu}\in X$ which has the energy level
$S_{\omega, \nu }( Q_{\omega,\nu}) = c_{g}$. Moreover, (\ref{eqn:2sn}) also has a positive mountain pass solution in $X$.
For $0 \leq \nu\leq 1$,  it was obtained in \cite{Du17} that when $p\geq3$, the ground-state solution is a mountain pass solution as well, which,  by virtue of the mountain pass geometry of functional $S_{\omega,\nu}$, yields that there are two constants $c_1,c_2>0$ independent of $\nu$ such that $c_1\leq c_{g}\leq c_2$. For $1<p<3$, by the fact that $c_g$ is less than mountain pass solution energy and using the corresponding Pohozaev identity in \cite{Du17}, it can be deduced again that  there are two constants $c_1, c_2 > 0$ satisfying  $c_1\leq c_{g}\leq c_2$.
%An alternative argument will be given in Lemma 3.1.
Then, applying a similar argument  to that in \cite[Proposition 3.3]{Du17},
it is obtained that $\{Q_{\omega,\nu}\}_{\nu\in(0,1]}$ is uniformly bounded in $H^1(\R^2)$.
{In addition, it has been shown in \cite{Cingolani16} by using the moving planes method that
all positive solutions in $X$ of equation (\ref{eqn:2sn}) are radially symmetric and strictly decreasing with respect to the distance from
the center of symmetry. Therefore, $Q_{\omega,\nu}$ is positive and radially symmetric.
Without loss of generality (or up to a subsequence), by the uniform boundness of $\{Q_{\omega,\nu}\}_{\nu\in(0,1]}$  in $X$,
there exists $u_\omega\in X$ such that
$Q_{\omega,\nu}\rightharpoonup u_\omega$ weakly in $H^1(\R^2)$ and $Q_{\omega,\nu}\rightarrow u_\omega$ in $L^q(\R^2)$ for $q>2$
as $\nu\rightarrow0^+$.}

Let us show $u_\omega\not\equiv0$. Suppose by contradiction that this is false.
%Suppose that, by contradiction, $u_\omega\equiv0$.
Since $Q_{\omega,\nu}\in X$ is a positive ground-state critical point of functional $S_{\omega,\nu}$,
we have  $S'_{\omega,\nu}(Q_{\omega,\nu})=0$. Then, by the  Hardy-Littleword-Sobolev inequality we get
\begin{equation}\label{bound-0}
\aligned
\|Q_{\omega,\nu}\|_{H^1} ^2+ & \int_{\R^4}\ln(1+|x-y|)Q_{\omega,\nu}^2(y)Q_{\omega,\nu}^2(x){\rm d}x{\rm d}y\\
=&\int_{\R^4}\ln\left(1+\frac{1}{|x-y|}\right)Q_{\omega,\nu}^2(y)Q_{\omega,\nu}^2(x){\rm d}x{\rm d}y
+\nu\int_{\R^2}|Q_{\omega,\nu}|^{p+1}{\rm d} x\\
\leq& C\|Q_{\omega,\nu}\|_{\frac{8}{3}}^4+\int_{\R^2}|Q_{\omega,\nu}|^{p+1}{\rm d} x,
\endaligned
\end{equation}
which yields that
$\|Q_{\omega,\nu}\|_{H^1}^2+B_1\left (Q_{\omega,\nu}^2,Q_{\omega,\nu}^2\right )\rightarrow0$ as $\nu\rightarrow0^+$.
It further implies that
$S_{\omega,\nu}(Q_{\omega,\nu})\rightarrow0$ as $\nu\rightarrow0^+$, a contradiction due to $c_g\geq c_1 > 0$. Consequently, $u_\omega\not\equiv0$.
It follows once again from (\ref{bound-0}) that $B_1(Q_{\omega,\nu}^2,Q_{\omega,\nu}^2)<+\infty$ uniformly in $\nu$.
Therefore, Lemma \ref{Lem:cw1} implies that $\{Q_{\omega,\nu}\}_{\nu\in(0,1]}$ is uniformly bounded in $X$ as well.

%It has been shown in \cite{Cingolani16}  using the moving planes method that
%all positive solutions in $X$ of equation (\ref{eqn:2sn}) are radial symmetric and strictly decreasing in terms of the distance from
%the center of symmetry. Therefore, $Q_{\omega,\nu}$ is positive and radial symmetric.
Without loss of generality (or up to a subsequence),
%by the uniform boundness of $\{Q_{\omega,\nu}\}_{\nu\in(0,1]}$  in $X$,
%there exists $u_\omega\in X$ such that
$Q_{\omega,\nu}\rightharpoonup u_\omega$ weakly in $X$
%and $Q_{\omega,\nu}\rightarrow u_\omega$ in $L^q(\R^2)$ for $q>2$
as $\nu\rightarrow0^+$. Thus, $u_\omega$ is a weak solution of  (\ref{eqn:2sn}) with $\nu=0$.
By \cite[Lemma A.IV]{Berestycki83}, there exists $C>0$ independent of $\nu\in(0,1]$ such that
\begin{align*}
|Q_{\omega,\nu}(x)|&\leq C|x|^{-1}\|Q_{\omega,\nu}\|\leq C|x|^{-1},
\end{align*}
which\ implies\ that $
\lim_{|x|\to\infty}|Q_{\omega,\nu}(x)|= 0\text{\  uniformly for }\nu\in(0,1]. $ By a similar argument to that in \cite[Lemma 4.4]{Liu23} and a comparison principle,
there exist constants $R,C > 0$
independent of $\nu$ such that
\begin{equation}\label{expo}
Q_{\omega,\nu}(x)\leq C\exp\left(-|x|/2\right)\quad\text{for }|x|\geq R.
\end{equation}
In addition, by applying the regularity estimates for elliptic equations,
one can show that  $\|Q_{\omega,\nu}\|_{W^{2,p}(\R^2)}\leq C$ for $p\in(2,+\infty)$, where $C$ is independent of $\nu$.
Sobolev compact embedding theorem in a bounded region
yields that for some $\tau\in(0,1)$
\begin{equation}\label{cov-X}
Q_{\omega,\nu}\rightarrow u_\omega\,\quad\text{in}\,\,C_{loc}^{1,\tau}(\R^2),
\quad\mbox{which, by (\ref{expo}), gives} \quad Q_{\omega,\nu}\rightarrow u_\omega \quad\text{in}\,\, X
\end{equation}
as $\nu\rightarrow0^+$, and then
$c_g=S_{\omega,\nu}(Q_{\omega,\nu})\rightarrow S_{\omega}(u_{\omega})\geq c_1$ as $\nu\rightarrow0^+$.
Thus, by uniqueness of positive solution of equation (\ref{eqn:2sn}) with $\nu=0$ \cite{Choquard08, Cingolani19}, it is obtained that  $u_{\omega}=Q_\omega$ and $Q_{\omega,\nu}\rightarrow Q_\omega$ in $ X$ as $\nu \rightarrow 0^+$, where the convergence may occur for only a subsequence.  However, above argument holds for any sequence of $\nu \rightarrow 0$ which has a subsequence such that  $Q_{\omega,\nu}\rightarrow Q_\omega$ in $ X$. Consequently, $Q_{\omega,\nu}\rightarrow Q_\omega$ in $ X$ for any $\nu \rightarrow 0^+$.
\end{proof}

%%%%%%%%%%%%%%%%%%%%%%%%%%%%%%%%%%%%%%%%%%%%

Now, we apply the Implicit Function Theorem
to prove uniqueness of ground-state solution of problem (\ref{eqn:2sn}) for $\nu>0$ small.
\begin{proposition}\label{prop3.2}
For $0 \leq \nu\leq 1$ small, the positive ground-state solution of (\ref{eqn:2sn}) is unique.
\end{proposition}
\begin{proof}
Since it has {been proved in \cite{Cingolani16}} that positive ground-state solutions of (\ref{eqn:2sn})  are radially symmetric, we only need to focus on radially symmetric solutions.
First, define $F:\mathbb{R}^+\times X_{rad}\to L_{rad}^2(\R^2)$ as
\begin{equation*}
F(\nu,u):=-\DD u+\omega u-\bigg(\ln\frac{1}{|x|}\ast|u|^2\bigg)u-\nu |u|^{p-1}u.
\end{equation*}
Then, for every $v\in X_{rad}$,
\begin{equation*}
\left ( F(\nu,u), v\right )_{L^2} =\int_{\mathbb{R}^2}\left(\nabla u\nabla v+\omega uv
-\bigg(\ln\frac{1}{|x|}\ast|u|^2\bigg)uv-\nu|u|^{p-1}uv\right){\rm d}x.
\end{equation*}
Clearly $F(0,Q_{\omega})=0$ and $F$ is continuously differentiable in a neighborhood of $(0,Q_{\omega})$.
Moreover, the Fr\'{e}chet derivative of $F$ with respect to $Q_{\omega}$,
denoted by $D_{u}F(\nu,Q_{\omega}):\tilde{X}_{rad}\to L_{rad}^2(\R^2)$,
satisfies
\begin{equation*}
D_{u}F(\nu,u)v=-\DD v+\omega v-\bigg(\ln\frac{1}{|x|}\ast |u|^2\bigg)v-2\bigg(\ln\frac{1}{|x|}\ast uv\bigg)u-\nu p|u|^{p-2}uv,
\,\ \forall v\in \tilde{X}_{rad},
\end{equation*}
where $\tilde{X}_{rad}$ is the space of all radially symmetric functions in $\tilde{X}$.

Next, define an operator $\mathcal{L}:\,\tilde{X}_{rad}\to L_{rad}^2(\R^2)$ as
\begin{equation*}
\mathcal{L}(v):=D_uF(0,Q_{\omega})v=-\DD v+\omega v-\bigg(\ln\frac{1}{|x|}\ast |Q_{\omega}|^2\bigg)v
-2\bigg(\ln\frac{1}{|x|}\ast Q_{\omega}v\bigg)Q_{\omega}.
\end{equation*}
It is known from \cite{Bonheure17} that the kernel of $\mathcal{L}$ on $\tilde{X}_{rad}$ is trivial (it is noted that the kernel of $\mathcal{L}$ on $\tilde{X}$ is two-dimensional and spanned
by the partial derivatives of $Q_{\omega}$).

\noindent{\it Claim:  $\mathcal{L}:\tilde{X}_{rad}\to L_{rad}^2(\R^2)$ is one-to-one and  the inverse
$\mathcal{L}^{-1}:L_{rad}^2(\R^2)\to \tilde{X}_{rad}$ is continuous. }

\noindent The proof of the claim will be given in Appendix. Thus,
by the Implicit Function Theorem, there exist $\nu_0>0$ and a neighborhood $\Phi$ of $Q_\omega$  in $\tilde{X}_{rad}$
%%
%\footnote{[ZWN260604] revised. Please check}
%%
such that, for every $\nu\in [0,\nu_0)$, the equation $F(\nu,u)=0$ has
a unique solution $\bar{u}_{\omega,\nu}\in \Phi$ and the curve $g:[0,\nu_0)\rightarrow
\tilde{X}_{rad}$ with $g(\nu ) =  \bar{u}_{\omega,\nu}$ is
of class $C^1$.

Assume that
there exist sequences $\{u_{\omega,\nu_n}\}$ and $\{v_{\omega,\nu_n}\}$ of positive
ground-state solutions such that
$u_{\omega,\nu_n}\ne v_{\omega,\nu_n}$ are radially symmetric and ${\nu_n}\to 0$. By Lemma \ref{lemma3.1}, they both converge to $Q_\omega$ in $X$,
that is,
\begin{equation}\label{cov-X+}
u_{\omega,\nu_n}\to Q_\omega,\quad\text{and}\quad v_{\omega,\nu_n}\to Q_\omega,\ \mbox{strongly in}\ X.
\end{equation}
Therefore, for $n$ large enough, the equation  $F(\nu_n,u)=0$ has two solutions in $\Phi$
such that $\nu_n<\nu_0$, a contradiction. Therefore, there exists $\nu^*>0$ such that for any $\nu\in[0,\nu^*)$
equation (\ref{eqn:2sn}) has only one positive ground-state solution.
\end{proof}

%%%%%%%%%%%%%%%%%%%%%%%%%%%%%%%%%%

We are ready to  prove the nondegeneracy of the unique positive radially symmetric
ground-state solution $Q_{\omega,\nu}$ of (\ref{eqn:2sn}) when $\nu>0$ is small.
In order to obtain the result,  consider the following minimization problem with a constraint.
For any $u\in X$, let
$$
\mathcal{E}(u):=\frac{1}{2}\int_{\R^2}|\na u|^2{\rm d}x-\frac{1}{4}\int_{\R^4}\ln\frac{1}{|x-y|}u^2(x)u^2(y){\rm d}x{\rm d}y,
\quad \mathcal{F}(u):=\frac{1}{2}\int_{\R^2}|u|^2{\rm d}x,
$$
and define
\be\lab{mini}
c_{min}:=\inf_{u\in\mathcal{M}}\mathcal{E}(u),\quad \mbox{where}\quad \mathcal{M}:=\left\{u\in X:\mathcal{F}(u)=\mathcal{F}(Q_\omega)\right\},
\ee
where $Q_\omega$ is the unique radially symmetric positive solution of equation \re{eqn:2sn} with $\nu=0$.

\begin{lemma}\lab{ground} The following conclusions hold:
\begin{itemize}
\item [{\rm(i)}] $-\infty<c_{min}:=\inf_{u\in\mathcal{M}}\mathcal{E}(u)<+\infty$.
\item [{\rm(ii)}] $c_{min}$ can be achieved by $Q_\omega$.
\end{itemize}
\end{lemma}

\bp
For any $u\in \mathcal{M}$, by the Hardy-Littleword-Sobolev inequality in  \cite{Lieb01}
and the Gagliardo-Nirenberg inequality in \cite{Wein0}, it is obtained that
\begin{equation}\label{eqn:s1}
\aligned
&\int_{\R^4}\ln\left(1+\frac{1}{|x-y|}\right)u^2(x)u^2(y){\rm d}x{\rm d}y
\leq\int_{\R^4}\frac{u^2(x)u^2(y)}{|x-y|}{\rm d}x{\rm d}y\leq C\|u\|^4_{8/3}\leq C\|\na u\|_2\|u\|_2^{3}\, ,
\endaligned
\end{equation}
which yields that
\begin{equation}\label{eqn:s2}
\inf_{u\in\mathcal{M}}\mathcal{E}(u)\ge\inf_{u\in\mathcal{M}}
\left(\frac{1}{2}\|\na u\|_2^2-C\|\na u\|_2\|u\|_2^{3}
\right)\ge\min_{t\in[0,\iy)}\left(\frac{1}{2}t^2-C\|Q_\omega\|_2^3t\right)>-\iy.
\end{equation}
On the other hand, by the Pohozaev identity in \cite{Du17}
$$
\omega\int_{\R^2}Q_\omega^2{\rm d}x+\frac{1}{4}\|Q_\omega\|_2^4
-\int_{\R^4}\ln\frac{1}{|x-y|}Q_\omega^2(x)Q_\omega^2(y){\rm d}x{\rm d}y=0
$$
and the fact that $Q_\omega$ is the unique positive solution of \re{eqn:2sn} with $\nu=0$ which gives a variational identity, it is then deduced that
\begin{equation}\label{eqn:s2+}
\|\na Q_\omega\|_2^2=\frac{1}{4}\|Q_\omega\|_2^4.
\end{equation}
Using the Pohozaev identity again, we obtain
$$
c_{min}\le\mathcal{E}(Q_\omega)=\left(\frac{1}{2}\|Q_\omega\|_2^4-\frac{1}{4}\int_{\R^4}\ln\frac{1}{|x-y|}Q_\omega^2(x)Q_\omega^2(y){\rm d}x{\rm d}y\right)
{<\frac{1}{2}\| Q_\omega\|_2^4+\frac{\omega}{4}\| Q_\omega\|_2^2.}
$$
This completes the proof of  {\rm(i)}.

For {\rm(ii)}, taking any minimization sequence $\{u_n\}$ in $X$ of $c_{min}$ such that $\mathcal{E}(u_n)\rg c_{min}$
as $n\rg\iy$ and $\|u_n\|_2=\|Q_\omega\|_2$,
as arguing in \re{eqn:s1} and \re{eqn:s2}, we see that $\{\|\na u_n\|_2^2\}$ is bounded and then $\{u_n\}$ is
bounded in $H^1(\R^2)$. {Similarly to the arguments in \cite[Proposition 3.1]{Cingolani16},
one uses Lions' Lemma to obtain the existence of a sequence  $\{ x_n\} \subset \R^2$ such that $\tilde{u}_n:=u_n(x-x_n)$ converges
weakly in $H^1(\R^2)$ to some function $u_0\in H^1(\R^2)\setminus\{0\}$ as $n\rightarrow\infty$. Without loss of generality,
we still denote $\tilde{u}_n$ by $u_n$.}
Similar to \re{eqn:s1},  the Hardy-Littleword-Sobolev inequality and the Gagliardo-Nirenberg inequality yield
$$
\aligned
\mathcal{E}(u_n)&=\frac{1}{2}\int_{\R^2}|\na u_n|^2{\rm d}x+\frac{1}{4}B_1\left (u_n^2,u_n^2\right )-\frac{1}{4}B_2\left (u_n^2,u_n^2\right )\\
&\geq \frac{1}{2}\int_{\R^2}|\na u_n|^2{\rm d}x+\frac{1}{4}B_1\left (u_n^2,u_n^2\right )
-C\|\na u_n\|_2\|u_n\|_2^{3},
\endaligned
$$
which implies that $\left |B_1\left (u_n^2,u_n^2\right )\right |\leq C$ uniformly for $n$. Lemma \ref{Lem:cw1} yields $\{\|u_n\|_w^2\}$ is bounded.
Hence, the sequence $\{u_n\}$ is bounded in $X$. It follows from \cite[Lemma 2.8]{Cingolani-Jeanjean} and Lemma \ref{Lem:compact0} that, up to a subsequence, $u_n\rightarrow u_0$ in $X$ with $\| u_0\|_2 = \|Q_\omega \|_2$.
Thus, $c_{min}$
can be achieved
by {$u_0\in \mathcal{M}$}. Furthermore, since $u_0 $ is a minimizer of $\mathcal{E}(u)$ with a constraint $ u \in \mathcal {M}$, there exists a Lagrange multiplier $\lambda_0\in\R$ such that
\begin{equation}\label{eqn:s3}
-\DD u_0-(\ln\frac{1}{|x|}\ast|u_0|^2)u_0=-\lambda_0u_0.
\end{equation}
Hence,  the Pohozaev identity corresponding to (\ref{eqn:s3}) (see \cite{Du17})
$$
\la_0\int_{\R^2}u_0^2{\rm d}x+\frac{1}{4}\|u_0\|_2^4
-\int_{\R^4}\ln\frac{1}{|x-y|}u_0^2(x)u_0^2(y){\rm d}x{\rm d}y=0
\quad\mbox{implies }\quad
\|\na u_0\|_2^2=\frac{1}{4}\|u_0\|_2^4.$$
Thus,  \eqref{eqn:s2+} yields that $\|\na u_0\|_2^2=\|\na Q_\omega\|_2^2$.
By the fact that $c_{min}=\mathcal{E}(u_0)\le\mathcal{E}(Q_\omega)$, it is deduced that
$$
\int_{\R^4}\ln\frac{1}{|x-y|}Q_\omega^2(x)Q_\omega^2(y){\rm d}x{\rm d}y\leq\int_{\R^4}\ln\frac{1}{|x-y|}u_0^2(x)u_0^2(y){\rm d}x{\rm d}y.
$$
From the above facts together with \re{eqn:2sn}, it is obtained that $\lambda_0\geq\omega$.

Since $u_0$ is a minimizer of \eqref{mini} and must keep one sign, without loss of generality,
assume $u_0\geq0$. By maximum principle, $u_0$ is a positive solution of (\ref{eqn:s3})
and thus is unique and spherically symmetric for given $\lambda_0$ (see \cite{Cingolani16}).
{Next, we need to rule out the case $\lambda_0>\omega$, which then gives $\lambda_0=\omega$.
Assume $\lambda_0>\omega$. It is straightforward to see that  (\ref{eqn:s3}) yields
\begin{equation}\label{eqn:s3-1}
S_{\lambda_0}(u_0)=\max_{t\in(0,+\infty)}g(t)
:=\max_{t\in(0,+\infty)}S_{\lambda_0}(tu_0) = \max_{t\in(0,+\infty)}((1/2)t^2 - (1/4)t^4 ) C(u_0)
\end{equation}
 where $C(u_0 ) = \int_{\R^2}|\na u_0 |^2{\rm d}x+\lambda_0\int_{\R^2}|u_0 |^2{\rm d}x $ and
$$
S_{\lambda_0}(u)=\frac{1}{2}\int_{\R^2}|\na u|^2{\rm d}x+\frac{\lambda_0}{2}\int_{\R^2}|u|^2{\rm d}x
-\frac{1}{4}\int_{\R^4}\ln\frac{1}{|x-y|}u^2(x)u^2(y){\rm d}x{\rm d}y,\quad u\in X.
$$
%It is easy to get $g'(t)<0$ for $t>1$ and $g(t)\rightarrow-\infty$ as $t\rightarrow+\infty$.
Moreover, define $h(t)=S_{\omega}(tu_0)$ for $t\in(0,+\infty)$ with
$$
S_{\omega}(u)=\frac{1}{2}\int_{\R^2}|\na u|^2{\rm d}x+\frac{\omega}{2}\int_{\R^2}|u|^2{\rm d}x
-\frac{1}{4}\int_{\R^4}\ln\frac{1}{|x-y|}u^2(x)u^2(y){\rm d}x{\rm d}y,\quad u\in X.
$$
Then, by  (\ref{eqn:s3}) and $\omega<\lambda_0$, it is obtained that  $h (0) = 0$, $h(t)>0$ for $0<t<<1$,
and $h(t)\rightarrow-\infty$ as $t\rightarrow+\infty$, which implies that there must be $t^*\in(0,+\infty)$ such that
\begin{equation}\label{eqn:s3-2}
h(t^*)=\max_{t\in(0,+\infty)}h(t)=\max_{t\in(0,+\infty)}S_{\omega}(tu_0)
\end{equation}
and
$h'(t^*)=0$. Since
\begin{equation}
h'(t^*)=(t^*)^3\left[\frac{\int_{\R^2}|\na u_0|^2{\rm d}x+\omega\int_{\R^2}|u_0|^2{\rm d}x}{(t^*)^2}
-\int_{\R^4}\ln\frac{1}{|x-y|}u_0^2(x)u_0^2(y){\rm d}x{\rm d}y\right]=0, \label {10000}
\end{equation}
 (\ref{eqn:s3}) yields that $t^*\in(0,1)$. It is known from \cite{Stubbe08} that $Q_\omega$ is a
ground-state critical point of $S_\omega$. It can also be deduced from \cite[Theorem 1.1]{Cingolani16}  that $S_\omega (Q_\omega ) \leq \sup_{ t \in \mathbb{R} } S_\omega (t u)  = \sup_{ t \in [0, \infty) } S_\omega (t u) $ for any $u \in X$.
Hence,   (\ref{eqn:s3-2}) gives
\begin{equation}\label{eqn:s3-3}
S_\omega(Q_\omega)\leq h(t^*)=S_{\omega}(t^*u_0)=\max_{t\in(0,+\infty)}S_{\omega}(tu_0).
\end{equation}
Then, \eqref{10000} implies that
\begin{equation}\label{eqn:s3-4}
S_{\omega}(t^*u_0)=\frac{(t^*)^2}4\left(\int_{\R^2}|\na u_0|^2{\rm d}x+\omega\int_{\R^2}|u_0|^2{\rm d}x\right)\, .
\end{equation}
Moreover,
\begin{equation}\label{eqn:s3-5}
\aligned
S_{\omega}(Q_\omega)&={S_{\omega}(Q_\omega)-\frac14S'_{\omega}(Q_\omega)Q_\omega}=\frac{1}4\left(\int_{\R^2}|\na Q_\omega|^2{\rm d}x+\omega\int_{\R^2}|Q_\omega|^2{\rm d}x\right).
\endaligned
\end{equation}
Since $\|\na u_0\|_2^2=\|\na Q_\omega\|_2^2$ and $\|u_0\|_2^2=\|Q_\omega\|_2^2$,
(\ref{eqn:s3-3})-(\ref{eqn:s3-5}) yield $t^*\geq1$, which is a contradiction.
Therefore, $\lambda_0=\omega$.
By the uniqueness, $u_0=Q_\omega$ and
 (ii) is proved.}
\ep

We define a linearized operator at $Q_\omega$ as
\begin{eqnarray}\label{eqn:L+}
L_{+}\vp:=-\DD\vp+\omega\vp-\left (\ln\frac{1}{|x|}\ast Q_{\omega}^2\right )\vp
-2\left (\ln\frac{1}{|x|}\ast Q_{\omega}\vp\right )Q_{\omega},\,\,\vp\in \tilde{X}.\,
\end{eqnarray}
%and call Morse index of $L_{+}$ as the number of negative eigenvalues of $L_{+}$.
The following lemma holds.
\bl\label{X-compact}
The Morse index of operator $L_+$ is $1$ and the eigenvalue problem $L_+f=\lambda f$ possesses a sequence of eigenvalues $\{\lambda_k\}$ approaching to
infinity as $k\rightarrow\infty$ with finite multiplicity for each eigenvalue $\lambda_k$.
\el

\Proof
%%
%\footnote{[ZWN] The following sentence is simplified. Please check}
%%
\textcolor[rgb]{0.00,0.00,1.00}{By Lemma \ref{ground}, $Q_\omega$ is a minimizer of $S_\omega (u) $ on $\mathcal{M}$.}
Given a test function $\phi$ satisfying $(\phi,Q_{\omega})=0$, there is a function $\psi (x) $ such that $\| Q_\omega + \epsilon \phi + \epsilon^2 \psi \|_2 = \| Q_\omega \|_2$ (for example, $\psi = a_0 e^{-\|x\|^2}$ with an appropriate real constant $a_0$). If $\tilde \phi = \phi + \epsilon \psi$,
$$
\aligned
S_{\omega}(Q_\omega+\epsilon \tilde \phi)=&\frac{1}{2}\int_{\R^2}|\na (Q_\omega+\epsilon \tilde \phi )|^2{\rm d}x+\frac{\omega}{2}\int_{\R^2}|Q_\omega
+\epsilon \tilde \phi |^2{\rm d}x\\
&-\frac{1}{4}\int_{\R^4}\ln\frac{1}{|x-y|}[Q_\omega+\epsilon \tilde \phi ]^2(x)[Q_\omega+\epsilon \tilde \phi ]^2(y){\rm d}x{\rm d}y\\
=&S _{\omega}(Q_\omega)+\frac{\epsilon^2}{2}\lan L_+\tilde \phi , \tilde \phi\ran_{L^2}+o(\epsilon^2) =
S _{\omega}(Q_\omega)+\frac{\epsilon^2}{2}\lan L_+\phi,\phi\ran_{L^2}+o(\epsilon^2).
\endaligned
$$
Since $Q_\omega$ is a minimizer of $S_\omega$ on $\mathcal{M}$, the second variation $\lan L_+\phi,\phi\ran\geq0$. By  the minimax realization of the eigenvalues
of self-adioint operators, it is  inferred that $L_+$ has at most one negative eigenvalue, that is, the Morse index of $L_+$ is less than $1$.
Furthermore,
$$
\lan L_+Q_\omega,Q_\omega\ran_{L^2} =-2\int_{\R^2}\left ( |\nabla Q_{\omega}|^2+\omega Q^2_{\omega}\right ) {\rm d}x<0
$$
implies, again by the minimax realization of the eigenvalues, that the Morse index of $L_+$ is $1$.

Next, for a parameter $K>0$, define an operator $T_K: L^2(\R^2)\mapsto \tilde{X}$ as
%%
%\footnote{[ZWN260604] The above is revised.}
%%
$$
T_Kf:=\left(-\DD+K+\omega-\left (\ln\frac{1}{|x|}\ast Q_{\omega}^2\right )-2\left (\ln\frac{1}{|x|}\ast Q_{\omega}\cdot\right )Q_{\omega}\right)^{-1}f,\quad f\in L^2(\R^2).
$$
Then, the general elliptic operator theory (see e.g. \cite{Gilbarg98})
%%
%\footnote{[ZWN260604] citation}
%%
together with Lemmas \ref{Lem:cw1}-\ref{Lem:compact0}
gives that $T_K$ is compact in {$L^2 (\mathbb{R}^2)$}  for large $K>0$.
Therefore, the eigenvalue problem
associated with the inverse
$$
T_K^{-1}=-\DD+K+\omega-\left (\ln\frac{1}{|x|}\ast Q_{\omega}^2\right )-2(\ln\frac{1}{|x|}\ast Q_{\omega}\cdot)Q_{\omega}
$$
has a sequence of positive eigenvalues $\lambda_k^K$ approaching to
infinity as $k\rightarrow\infty$ with finite multiplicity for each eigenvalue $\lambda_k^K$.
Furthermore, the eigenvalue problem
associated with
$L_+$ has a sequence of eigenvalues $\lambda_k^K-K$ approaching to
infinity as $k\rightarrow\infty$ with finite multiplicity for each eigenvalue $\lambda_k^K-K$.
The proof is completed.
\qed

{\begin{lemma}\label{tuihua1}
There exist $\nu^*_1>0 $ and $c>0$ such that
$$
\inf_{\phi\in\mathcal{N}}\frac{\langle L_{+,\nu}\phi,\phi\rangle}{\|\phi\|_X^2}> c ~~~\mbox{for any}~ \nu\in[0,\nu^*_1),
$$
where $\mathcal{N}:=\{\phi\in \tilde{X}\setminus\{0\}:\,(\phi,Q_{\omega,\nu})_{L^2} =(\phi,\partial_{x_j}Q_{\omega,\nu})_{L^2} =0,\,\, j=1,2\}$
and
$$
\aligned
\langle L_{+,\nu}\phi,\phi\rangle_{L^2}=&\|\phi\|^2_{H^1}-\int_{\R^2}\left (\ln\frac{1}{|x|}\ast Q_{\omega,\nu}^2\right )\phi^2{\rm d}x\\
&-2\int_{\R^2}  \left (\ln\frac{1}{|x|}\ast Q_{\omega,\nu}\phi \right )Q_{\omega,\nu}\phi{\rm d}x-p\nu\int_{\R^2} Q_{\omega,\nu}^{p-1}\phi^2{\rm d}x.
\endaligned
$$
\end{lemma}
\Proof
The lemma is proved by a contradiction if otherwise.  Thus, let sequences $\{\nu_n\}\subset\R$
 and $\{\phi_n\}\subset\mathcal{N}$ satisfying that  $\nu_n\rightarrow0^+$, and
  \begin{equation*}
    \lim_{n\rightarrow\infty}\langle L_{+,\nu_n}\phi_n,\phi_n\rangle_{L^2} =b_*\leq0,\quad
    \left\|\phi_n\right\|_{X}^2=1, \quad(\phi_n,Q_{\omega,\nu_n})_{L^2} =0,\quad (\phi_n,\partial_{x_j}Q_{\omega,\nu_n})_{L^2} =0,\,\,j=1,2.
  \end{equation*}
Since $\left\|\phi_n\right\|_{X}^2=1$, it is obtained that $\{\phi_n\}$ is bounded in $\tilde{X}$ and
there exists a $\phi_0\in \tilde{X}$ such that $\phi_n\rightharpoonup \phi_0$ in $\tilde{X}$ and
$\phi_n\rightarrow  \phi_0, a.e.$ with its derivatives.
Since (\ref{cov-X}) implies that $Q_{\omega,\nu_n}\rightarrow Q_{\omega}$ in $X$, it is deduced that
$Q_{\omega,\nu_n}\phi_n\rightarrow Q_{\omega}\phi_0$ in $X$.
Thus, it follows from (\ref{eqn:es1})-(\ref{eqn:es2}), Lemma \ref{Lem:cw2} and Lemma \ref{Lem:compact0}(i) that, up to a subsequence,
\begin{equation}\label{cov-Xguji1}
\aligned
  &\lim_{n\rightarrow\infty}B_2\left (Q^2_{\omega,\nu_n},\phi_n^2\right )=B_2\left (Q^2_{\omega},\phi_0^2\right )\\
  &\lim_{n\rightarrow\infty}B_2\left (Q_{\omega,\nu_n}\phi_n,Q_{\omega,\nu_n}\phi_n\right )=B_2(Q_{\omega}\phi_0,Q_{\omega}\phi_0)
\endaligned
\end{equation}
and
\begin{equation}\label{cov-Xguji2}
\aligned
 &\lim_{n\rightarrow\infty}\left|B_1(Q_{\omega,\nu_n}\phi_n,Q_{\omega,\nu_n}\phi_n)
  -B_1(Q_{\omega}\phi_0,Q_{\omega}\phi_0)\right|\\
  &=\lim_{n\rightarrow\infty}\left|B_1(Q_{\omega,\nu_n}\phi_n-Q_{\omega}\phi_0,Q_{\omega,\nu_n}\phi_n)
  +B_1(Q_{\omega}\phi_0,Q_{\omega,\nu_n}\phi_n-Q_{\omega}\phi_0)\right|=0.
\endaligned
\end{equation}
Then,  Fatou's lemma
and (\ref{cov-Xguji1})-(\ref{cov-Xguji2}) yield
\begin{equation}\label{eqn:++}
\aligned
  0\geq b_*\geq& \liminf_{n\rightarrow\infty}\langle L_{+,\nu_n}\phi_n,\phi_n\rangle_{L^2} \\
=&\liminf_{n\rightarrow\infty}\bigg(\left\|\phi_n\right\|_{H^1}^2+B_1\left (Q^2_{\omega,\nu_n},\phi_n^2\right )-
  B_2(Q^2_{\omega},\phi_0^2)\\
  &\qquad\qquad +2B_1(Q_{\omega}\phi_0,Q_{\omega}\phi_0)
  -2B_2(Q_{\omega}\phi_0,Q_{\omega}\phi_0)\bigg)\\
\geq&\left\|\phi_0\right\|_{H^1}^2+B_0\left (Q^2_{\omega},\phi_0^2\right )+2B_0\left (Q_{\omega}\phi_0,Q_{\omega}\phi_0\right ),
\endaligned
\end{equation}
which also implies that $\langle L_{+}\phi_0,\phi_0\rangle_{L^2} \leq 0$. We claim that $\phi_0\not\equiv0$. In fact,
if $\phi_0\equiv0$, then it follows from (\ref{eqn:++}) that
\begin{equation}
\lim_{n\rightarrow\infty}\left(\left\|\phi_n\right\|_{H^1}^2+B_1\left (Q^2_{\omega,\nu_n},\phi_n^2\right )\right)=0, \label{eqn:+++}
\end{equation}
which gives by Lemma \ref{Lem:cw1} that $\phi_n\rightarrow0$ in $\tilde{X}$, a contradication with $\|\phi_n\|_X=1$.
 Moreover, $(\phi_n,Q_{\omega,\nu_n})_{L^2} =0$
yields that $(\phi_0,Q_{\omega})_{L^2} =0$, and $(\phi_n,\partial_{x_i}Q_{\omega,\nu_n})_{L^2} =0$
yields that $(\phi_0,\partial_{x_i}Q_{\omega})_{L^2} =0$. As a direct consequence of Lemma \ref{X-compact}
and nondegeneracy of $L_+$ at $Q_\omega$, it is immediately found  that there exists $\tilde{c}>0$ such that
$\langle L_{+}\phi_0,\phi_0\rangle_{L^2} \geq \tilde{c}\|\phi_0\|_X^2 > 0$,  a contradiction with $\langle L_{+}\phi_0,\phi_0\rangle_{L^2} \leq 0$.
The proof of this lemma is completed.
\qed
}

In the following, the Morse index and nondegeneracy of $L_{+, \nu} $ at $Q_{\omega,\nu}$ for $\nu>0$ small enough are obtained.
\begin{proposition}  \label{prop3.6} If $\nu>0$ is small enough, then the Morse index of the linear operator $L_{+,\nu}$ at $Q_{\omega,\nu}$ is 1 and  the operator $L_{+,\nu}$  is nondegenerate.  That is, all the elements in
the kernal  of $L_{+,\nu}$ are given by $v=\xi\cdot\nabla Q_{\omega,\nu}$, where $\xi\in \R^2$.
\end{proposition}
\begin{proof}
Define the eigenvalue problem
for the operator $L_{+,\nu}$ as $L_{+,\nu}\phi=\sigma \phi$ for $\phi\in \tilde{X}$, i.e.,
\begin{equation}\label{cov-Xguji3}
-\triangle \phi+\omega\phi-\left(\ln\frac{1}{|x|}\ast Q_{\omega,\nu}^2\right)\phi
-2\left(\ln\frac{1}{|x|}\ast Q_{\omega,\nu}\phi\right)Q_{\omega,\nu}
-p\nu Q_{\omega,\nu}^{p-1}\phi=\sigma \phi.
\end{equation}
%Denote $W:\R^2\rightarrow\R$ by
%$$
%W(x):=-\left(\ln\frac{1}{|x|}\ast Q_{\omega,\nu}^2\right)
%-2\left(\ln\frac{1}{|x|}\ast Q_{\omega,\nu}\right)Q_{\omega,\nu}-p\nu Q_{\omega,\nu}^{p-1}.
%$$
By a similar argument to that used for Lemma \ref{X-compact}, it is deduced that
there exist a nondecreasing sequence of eigenvalues $\sigma_{\nu,i}\in\R$
 satisfying $\sigma_{\nu,i}\rightarrow+\infty$ as $i\rightarrow+\infty$
and a corresponding orthonormal basis of eigenfunctions $\phi_{\nu,i}\subset \tilde{X}$, namely,
$L_{+,\nu}\phi_{\nu,i}=\sigma_{\nu,i} \phi_{\nu,i}$. Let us normalize the sequence of eigenfunctions with
$\|\phi_{\nu,i}\|_2=1$. From earlier analysis,
$0$ is an eigenfunction of $L_{+,\nu}$ and
\begin{equation}\label{cov-Xguji3+}
\{\partial_{x_1}Q_{\omega,\nu},\partial_{x_2}Q_{\omega,\nu}\}\subset \mbox{ker}\,  L_{+,\nu},
\end{equation}
which implies that the first eigenvalue $\sigma_{\nu,1}$ is negative and simple,
and the first eigenfunction $\phi_{\nu,1}>0$.

First, show that there exists $\nu_2^*>0$ such that
the Morse index of $L_{+,\nu}$ is $1$ for $\nu\in(0,\nu_2^*)$.
If otherwise, then there exists a sequence $\{\nu_n\}\subset\R$ satisfying $\nu_n\rightarrow0^+$
with corresponding another eigenfunction $\psi_{\nu_n}\in \tilde{X}$ associated
with some eigenvalue $\sigma_{\nu_n}<0$. Without loss of generality, let  $\psi_{\nu_n}=\phi_{\nu_n,2}$ and
$\sigma_{\nu_n}=\sigma_{\nu_n,2}$. By $(Q_{\omega,\nu_n},\partial_{x_i}Q_{\omega,\nu_n})_{L^2}=(\phi_{\nu_n,2},\partial_{x_i}Q_{\omega,\nu_n})_{L^2}=0$, if
$$
w_{\nu_n}:=\phi_{\nu_n,2}-\frac{(\phi_{\nu_n,2},Q_{\omega,\nu_n})_{L^2}}{\|Q_{\omega,\nu_n}\|_2^2}Q_{\omega,\nu_n},
$$
then $(w_{\nu_n},Q_{\omega,\nu_n})_{L^2}=(w_{\nu_n},\partial_{x_i}Q_{\omega,\nu_n})_{L^2}=0,i=1,2$.
Hence,
\begin{equation}\label{cov-Xguji4}
|\langle L_{+,\nu_n}\phi_{\nu_n,2},\phi_{\nu_n,2}\rangle_{L^2}|=-\sigma_{\nu_n,2} \|\phi_{\nu_n,2}\|_2^2\leq C,
\end{equation}
which implies by (\ref{expo}) that $C$ is independent of $\nu_n$. By (\ref{expo}) again, a straightforward calculation shows that
\begin{equation}
|B_0(Q_{\omega,\nu_n}\phi_{\nu_n,2},Q_{\omega,\nu_n}\phi_{\nu_n,2})|\leq C,\quad{and}\quad
|B_2\left (Q_{\omega,\nu_n}^2,\phi_{\nu_n,2}^2\right )|\leq C. \nonumber
\end{equation}
%where we have the fact that $\left\|\int_{\R^2}\frac{1}{|\cdot-y|}Q^2_{\omega,\nu_n}(y){\rm d}y\right\|_\infty\leq C$.
Then, (\ref{cov-Xguji4}) and Lemma \ref{Lem:cw1} give that $\|\phi_{\nu_n,2}\|_{X}\leq C$ \textcolor[rgb]{0.00,0.07,1.00}{uniformly with respect to}
%%
%\footnote{[ZWN] uniformly for or with respect to??}
%%
%for
$\nu_n$, which implies that there exists $\phi_{2}\in \tilde{X}$ such that, up to a subsequence, $\phi_{\nu_n,2}\rightharpoonup\phi_{2}$ in $\tilde{X}$ as $\nu_n\rightarrow0^+$.
For any fixed $\psi\in C_0^\infty(\R^2)$, $\langle L_{+,\nu_n}\phi_{\nu_n,2},\psi\rangle_{L^2}
=\sigma_{\nu_n,2}(\phi_{\nu_n,2},\psi)_{L^2}$
yields that there exists $\sigma_{2}\in(-\infty,0]$ such that, again up to subsequence,
\begin{equation}
\langle L_{+}\phi_{2},\psi\rangle_{L^2} =\lim\limits_{\nu_n\rightarrow0^+}\langle L_{+,\nu_n}\phi_{\nu_n,2},\psi\rangle_{L^2}
=\lim\limits_{\nu_n\rightarrow0^+}\sigma_{\nu_n,2}(\phi_{\nu_n,2},\psi)_{L^2}=\sigma_{2}(\phi_2,\psi)_{L^2}, \label{3.23}
\end{equation}
where $L_{+}$ is given in \eqref{eqn:L+}.
%Using the elliptic regularity estimate,
%one has $\|\phi_{\nu_n,2}\|_{W^{2,p}(\R^2)}\leq C$ for %$p\in(2,+\infty)$ with $C$ independent of $\nu_n$.
%Sobolev's embedding theory yields that for some $\tau\in(0,1)$ such that
%$\phi_{\nu_n,2}\rightarrow\phi_{2}$ in $C_{loc}^{1,\tau}(\R^2)$ as $\nu_n\rightarrow0^+$.
Since $\phi_{\nu_n,2}$ does not keep the same sign, it can be deduced that $\phi_{2}$ has no definite sign either (here note that $\phi_2 \not \equiv 0$ by Lemma \ref{Lem:compact0}(i)). In fact, if otherwise, without loss of generality, $\phi_{2}$ is non-negative. Then, applying elliptic operator theory gives  $\phi_{2} > 0$. Thus, from
\eqref{3.23} and Proposition \ref{prop3.2}, $\sigma_2$ is the first eigenvalue of $L^+$ corresponding to the eigenfunction $\phi_2$.
Since $\sigma_{\nu_n,1}$ is the first eigenvalue of $L_{+, \nu}$ corresponding to eigenfunction $\phi_{\nu_n,1}$, by
perturbation theory, it is obtained that $\sigma_{\nu_n,1} \rightarrow \sigma_2$ and $\phi_{\nu_n,1} \rightarrow \phi_2$ in $L^2(\mathbb{R}^2)$. Therefore,
$1 = \| \phi_2\|^2_2 = \lim_{n\rightarrow \infty} (\phi_{\nu_n , 1} , \phi_{\nu_n , 2})_{L^2}
=  0$,  which is a contradiction.
Thus, from the spectral properties of $L_{+}$, one obtains that  $\sigma_{2}=0$,
i.e., $\sigma_{\nu_n,2}\rightarrow0$ as $\nu_n\rightarrow0^+$.
%Set
%$$
%w=\psi-\frac{(\psi,Q_{\omega,\nu})}{\|Q_{\omega,\nu}\|_2^2}Q_{\omega,\nu},
%$$
%then $(w,Q_{\omega,\nu})=0$.
Then, a straightforward calculation yields that
$$
\aligned
( & L_{+,\nu_n}w_{\nu_n} ,w_{\nu_n})_{L^2}\\
=&(L_{+,\nu_n}\phi_{\nu_n,2},\phi_{\nu_n,2})_{L^2}
+\left(\frac{(\phi_{\nu_n,2},Q_{\omega,\nu_n})_{L^2}}{\|Q_{\omega,\nu_n}\|_2^2}\right)^2(L_{+,\nu_n}
Q_{\omega,\nu_n},Q_{\omega,\nu_n})_{L^2}\\
&-
2\frac{(\phi_{\nu_n,2},Q_{\omega,\nu_n})_{L^2}}{\|Q_{\omega,\nu_n}\|_2^2}(L_{+,\nu_n}\phi_{\nu_n,2},Q_{\omega,\nu_n})_{L^2}\\
=&\sigma_{\nu_n,2}(\phi_{\nu_n,2},\phi_{\nu_n,2})_{L^2}-2\sigma_{\nu_n,2}\frac{|(\phi_{\nu_n,2},Q_{\omega,\nu_n})_{L^2}|^2}{\|Q_{\omega,\nu_n}\|_2^2}
+\left(\frac{(\phi_{\nu_n,2},Q_{\omega,\nu_n})_{L^2}}{\|Q_{\omega,\nu_n}\|_2^2}\right)^2(L_{+,\nu_n}
Q_{\omega,\nu_n},Q_{\omega,\nu_n})_{L^2}\\
<&-\sigma_{\nu_n,2}\frac{|(\phi_{\nu_n,2},Q_{\omega,\nu_n})_{L^2}|^2}{\|Q_{\omega,\nu_n}\|_2^2}
+\left(\frac{(\phi_{\nu_n,2},Q_{\omega,\nu_n})_{L^2}}{\|Q_{\omega,\nu_n}\|_2^2}\right)^2
\left[-2\omega\|Q_{\omega,\nu_n}\|_2^2+(3-p)\nu_n\|Q_{\omega,\nu_n}\|_{p+1}^{p+1}\right]\\
\leq&\left (-\sigma_{\nu_n,2}-2\omega+(3-p)\nu_n\|Q_{\omega,\nu_n}\|_{\infty}^{p-1}\right )\frac{|(\phi_{\nu_n,2},Q_{\omega,\nu_n})|^2}
{\|Q_{\omega,\nu_n}\|_2^2}<0
\endaligned
$$
for $n$ large enough, which contradicts to Lemma \ref{tuihua1} because of $(w_{\nu_n},Q_{\omega,\nu_n})=(w_{\nu_n},\partial_{x_i}Q_{\omega,\nu_n})=0,i=1,2$.
Therefore, there exists $\nu_2^*>0$ such that
the Morse index of $L_{+,\nu}$ is $1$ for $\nu\in(0,\nu_2^*)$.

Next, show that $\mbox{ker}\,  L_{+,\nu}=\{\zeta\cdot\nabla Q_{\omega,\nu}\}$ for $\nu>0$ small enough.
First, it is known
that (\ref{cov-Xguji3+}) holds.
Assume that there exists a sequence $\{\nu_n\}\subset\R$ satisfying $\nu_n\rightarrow0^+$
with another element $\phi_{\nu_n,*}\in \mbox{ker}\,  L_{+,\nu_n}\cap \mbox{span} \{ \partial _{x_1} Q_{\omega, \nu_n }, \partial _{x_2} Q_{\omega, \nu_n}\}^{\perp}$ satisfying $\| \phi_{\nu_n,*} \|_2 = 1$ as $\nu_n\rightarrow0^+$.
Then,
%using the variational characterization of eigenvalues yields
$
\langle L_{+,\nu_n}\phi_{\nu_n,*},\phi_{\nu_n,*}\rangle = 0 $ and  $\phi_{\nu_n,*}$ is bounded in $X$ uniformly for $n$ such that
there exists $\phi_{*}\in \tilde{X}$ satisfying $\phi_{\nu_n,*}\rightharpoonup \phi_{*}$ in $ \tilde{X}$ and
$\phi_{\nu_n,*}\rightarrow\phi_{*}$ in $L^q(\R^2)$ with $q\in[2,+\infty)$
as $\nu_n\rightarrow0^+$. Since $Q_{\omega,\nu_n}\rightarrow Q_\omega$ {in $X$} as $\nu_n\rightarrow0^+$,
a convergence argument yields that
\begin{equation}\label{cov-Xguji5}
\aligned
&\|\phi_{*}\|_H^2+B_1\left (Q_\omega^2,\phi_{*}^2\right )-B_2\left (Q_\omega^2,\phi_{*}^2\right )+2B_0(Q_\omega\phi_{*},Q_\omega\phi_{*})\\
\leq&\liminf\limits_{\nu_n\rightarrow0^+}\left(\|\phi_{\nu_n,*}\|_H^2+B_0(Q_{\omega,\nu_n}^2,\phi_{\nu_n,*}^2)
+2B_0(Q_{\omega,\nu_n}\phi_{\nu_n,*},Q_{\omega,\nu_n}\phi_{\nu_n,*})-p\nu_n\int_{\R^2}Q_{\omega,\nu_n}^{p-1}
\phi_{\nu_n,*}^2{\rm d}x\right)\\
=&\liminf\limits_{\nu_n\rightarrow0^+}\langle L_{+,\nu_n}\phi_{\nu_n,*},\phi_{\nu_n,*}\rangle=0.
\endaligned
\end{equation}
{Moreover,} it can be shown that $(\phi_{*},Q_\omega)_{L^2}=0$. In fact,
since $( L_{+,\nu_n}\phi_{\nu_n,*}, \varphi )_{L^2} =0$ for any $\varphi\in C_0^\infty(\R^2)$,
 a convergence argument gives that $(L_{+}\phi_{*},\varphi )_{L^2}=0$ for any $\varphi\in C_0^\infty(\R^2)$.
{Thus, $\phi_{*}\in \mbox{ker}\, L_+$ and there exist $c_1,c_2\in\R$ such that
$\phi_{*}=c_1\partial_{x_1}Q_\omega+c_2\partial_{x_2}Q_\omega$, which implies that $(\phi_{*},Q_\omega)_{L^2}=0$.
By
$\phi_{\nu_n,*}\in\mbox{span} \{ \partial _{x_1} Q_{\omega,\nu_n}, \partial _{x_2} Q_{\omega, \nu_n}\}^{\perp}$, (\ref{expo}) and $Q_{\omega,\nu_n}\rightarrow Q_\omega$
in $X$ and $\phi_{\nu_n,*}\rightharpoonup \phi_{*}$
in $ X$  as $\nu_n\rightarrow0^+$,
it is found  that $(\phi_{*},\partial_{x_i}Q_\omega)_{L^2}=0$ for $i=1,2$.  Hence,  (\ref{cov-Xguji5}) gives a contradiction since the nondegeneracy of
$L_+$ implies
$$
\langle L_+\phi_{*},\phi_{*}\rangle=\|\phi_{*}\|_H^2+B_1\left (Q_\omega^2,\phi_{*}^2\right )-B_2\left (Q_\omega^2,\phi_{*}^2\right )
+2B_0(Q_\omega\phi_{*},Q_\omega\phi_{*})>0.
$$
The proof is completed.}
\end{proof}

Finally, Propositions \ref{prop3.2} and \ref{prop3.6} imply
Theorem \ref{Th1}.

%%%%%%%%%%%%%%%%%%%%%%%%%%%%%%%%%%%%%%%%%%%%%%%%%%%%%%%%%%%%%%%%%%%%%%%%%%%

\section{Orbital stability}

In this section, the proof of
%%comprehensive theoretical
 orbital stability for the positive standing-wave solution $Q_{\omega,\nu}$ of
(\ref{eqn:2sn}) with $\nu\geq0$ small will be provided. The section will be divided into two subsections.

%%%%%%%%%%%%%%%%%%%%%%%%%%%%%%%%%%%%%%%%%%%%%%%%%%%%%%%%%%%%%%%%%%%%%%%%%%%%%%%%%%%%%%%%

\subsection{The linearized problem}

First, we give some properties of the positive solution $Q_{\omega,\nu}$ of \re{eqn:2sn}.
Let
$L_\nu$
be the linearization of equation \re{eqn:2sn} at $Q_{\omega,\nu}$  on $L^2(\R^2,\mathbb{C})$ with domain in
$\xi\in\tilde{X}(\R^2,\mathbb{C})$, where
$$
\tilde{X}(\R^2,\mathbb{C}):=\left \{\xi=\varphi+i\eta\, |\, \forall \varphi,\eta\in \tilde{X}\right \}.
$$
%with
%\begin{equation}\label{eqn:X+}
%\aligned
%\tilde{X}:=\bigg\{\vp\in X(\R^2,\R):\,\, &\text{there\, exists}\,f\in L^2(\R^2)\,\text{such\, that}\\
%& \text{for\, every}\, \psi\in C_c^\iy(\R^2),\, \int_{\R^2}\vp L_\nu\psi{\rm d}x=\int_{\R^2}f\psi{\rm d}x\bigg\}.
% \endaligned
%\end{equation}
More specifically, for any $\xi\in\tilde{X}(\R^2,\mathbb{C})$,
$$
L_\nu\xi=L_{+,\nu}\mathfrak{Re}(\xi)+iL_{-,\nu}\mathfrak{Im}(\xi),
$$
where $L_{+,\nu}$ was given in Theorem \ref{Th1}, and $L_{-,\nu}$ is defined by
\begin{eqnarray}
L_{-,\nu}\eta:=-\DD\eta+\omega\eta-\bigg(\ln\frac{1}{|x|}\ast Q_{\omega,\nu}^2\bigg)\eta-\nu Q_{\omega,\nu}^{p-1}\eta,\quad \eta\in\tilde{X},
\end{eqnarray}
respectively.
%where
%$$
%\aligned
%\tilde{X}:=\bigg\{\vp\in X(\R^3,\mathbb{R}):\,\, &\text{there\, exists}\,f\in L^2(\R^2)\,\text{such\, that}\\
%& \text{for\, every}\, \psi\in C_c^\iy(\R^2),\, \int_{\R^2}\vp L_+\psi{\rm d}x=\int_{\R^2}f\psi{\rm d}x\bigg\}.
% \endaligned
%$$
%We can write
%%
%\{\{
%{\color{red} [ZWN250611] Just wrote
%$$
%L_\nu\xi=L_{+,\nu}\mathfrak{Re}(\xi)+iL_{-,\nu}\mathfrak{Im}(\xi),
%$$
%but now you can rewrite a complex form in terms of a matrix?
%Does the reviewer agree?
%}
%\}\}
%%
%$$
%L_\nu=\begin{pmatrix} L_{+,\nu}& 0\\0&L_{-,\nu}\end{pmatrix}.
%$$
%%
%\footnote{[ZWN] The following is revised. Please check}
%%
\textcolor[rgb]{0.00,0.07,1.00}{By Theorem \ref{Th1} and the properties of $L_{+, \nu}$ studied in previous sections,
$L_{+,\nu}$ and $L_{-,\nu}$ are self-adjoint }and the solution $Q_{\omega,\nu}$ is non-degenerate
 in the sense that for small $\nu > 0$,
$$
\mbox{ker}\,  L_{+,\nu}={\rm span}\left\{\pl_{x_1}Q_{\omega,\nu},\pl_{x_2}Q_{\omega,\nu}\right\}.
$$
In addition, $Q_{\omega,\nu}\in C^\iy(\R^2)$, $Q_{\omega,\nu}(0)=\max_{x\in\R^2}Q_{\omega,\nu}(x)$.
By (\ref{expo}), it is obtained that $Q_{\omega,\nu}$ exponentially decays to zero at infinity \textcolor[rgb]{0.00,0.07,1.00}{uniformly with respect to}
%%
%\footnote{[ZWN] uniformly for or with respect to?}
%%
%for
$\nu > 0 $ small.

%\textbf{The spectrum of $L_{-,\nu}$.}
Since $Q_{\omega,\nu}$ is a solution of \re{eqn:2sn}, we have that  $L_{-,\nu}Q_{\omega,\nu}=0$, which implies that  $Q_{\omega,\nu}$ is an eigenfunction of the operator $L_{-,\nu}$
%we know that $Q_{\omega,\nu}$ is an eigenfunction of the operator
%$-\DD+\omega$
%in $\tilde{X} \cap L^2(\R^2,\big(\ln\frac{1}{|x|}\ast Q_{\omega,\nu}^2\big)\ud x)$
%associated with eigenvalue $1$.
associated with eigenvalue $0$.
Since $Q_{\omega,\nu}(x)>0,\,\, x\in\R^2$,
the eigenvalue $0$ is simple.
\begin{lemma}\label{L2}
For any fixed $\nu \geq 0$ small, there exists $\delta_0>0$ such that
$$
\lan L_{-,\nu} v,v\ran_{L^2}\geq \delta_0\|v\|_2^2
$$
for all $v\in \tilde{X}$ satisfying $(v, Q_{\omega,\nu})_{H^1}=0$.
\end{lemma}
\Proof
By \cite[Proposition 2.3]{Cingolani16} and (\ref{expo}), it is straightforward to see that
$$
\omega-\ln\frac{1}{|x|}\ast Q_{\omega,\nu}^2-\nu|Q_{\omega,\nu}|^{p-1}\rightarrow+\infty,\quad\quad |x|\rightarrow+\infty.
$$
Then,  \cite[Theorem 3.1]{Berezin91} gives that $(L_{-,\nu}, v , v)_{L^2}$ is bounded from below.
By a similar argument to
%%
%\footnote{[ZWN260604] to. Note that `similar to' but `the same as'. We need to check the whole paper for `to'}
%%
that in Lemma \ref{X-compact}, it is obtained that $L_{-,\nu}$ has only a discrete spectrum with eigenvalues of finite multiplicity approaching to $+\infty$.
Moreover, $L_{-,\nu}Q_{\omega,\nu}=0$ with $Q_{\omega,\nu}>0$, and $0$ is the smallest eigenvalue of $L_{-,\nu}$ (see \cite[Theorem 3.4]{Berezin91}) and  simple with $Q_{\omega,\nu}>0$ as the corresponding eigenfunction. Thus,
$$
\mbox{ker}\,  L_{-,\nu}={\rm span}\{Q_{\omega,\nu}\},
$$
and $\lan L_{-,\nu}v,v\ran\ge0$ for any $v\in \tilde{X}$. Now, we prove the lemma by contradiction. Assume that there exists a sequence $\{v_n\}\subset \tilde{X}$ with $\|v_n\|_2^2=1$ such that $(v_n, Q_{\omega,\nu})_{H^1}=0$ and
$\lan L_{-,\nu} v_n,v_n\ran_{L^2}\rightarrow0$ as $n\rightarrow\infty$, that is,
\begin{equation}\label{eqn:cov1}
\|v_n\|_{H^1}^2+{\int_{\R^2}\left(\ln|x|\ast Q_{\omega,\nu}^2\right)v_n^2{\rm d}x
-\nu\int_{\R^2}
Q_{\omega,\nu}^{p-1}v_n^{2}{\rm d}x}\rightarrow0,\quad \mbox{as} \quad n\rightarrow\infty.
\end{equation}
Since $\{Q_{\omega,\nu}\}_{\nu\in[0,1]}$ is bounded in $X$, it follows
from the Hardy-Littleword-Sobolev inequality and the Gagliardo-Nirenberg inequality that
\begin{equation}\label{eqn:cov2}
\aligned
\|v_n\|_{H^1}^2 & +\int_{\R^4}\ln(1+|x-y|)Q_{\omega,\nu}^2(y)v^2_n(x){\rm d}x{\rm d}y\\
=&\int_{\R^4}\ln\left(1+\frac{1}{|x-y|}\right)Q_{\omega,\nu}^2(y)v^2_n(x){\rm d}x{\rm d}y+{\nu\int_{\R^2}
Q_{\omega,\nu}^{p-1}v_n^{2}{\rm d}x}+o_n(1)\\
\leq& C\|Q_{\omega,\nu}\|_{\frac{8}{3}}^2\|v_n\|_{\frac{8}{3}}^2+{\nu\|Q_{\omega,\nu}\|_{\infty}^{p-1}\|v_n\|_2^{2}}+o_n(1)\\
\leq& C\|Q_{\omega,\nu}\|_{\frac{8}{3}}^2\|\nabla v_n\|^{\frac12}_2\|v_n\|_2^{\frac32}+C\|v_n\|_2^{2}+o_n(1),
\endaligned
\end{equation}
which implies that $\{v_n\}$ is bounded {in $H^1(\R^2)$,} and
$B_1\left (Q_{\omega,\nu}^2,v_n^2\right )<+\infty$ uniformly for $n$. Hence,
Lemma \ref{Lem:cw1} yields that $\{v_n\}$ is bounded in $\tilde{X}$, which gives that there exists $v_0\in\tilde{X}$ such that,
up to a subsequence, $v_n\rightharpoonup v_0$ in $\tilde{X}$ and $(v_0, Q_{\omega,\nu})_{H^1}=0$.
From (\ref{eqn:cov1})-(\ref{eqn:cov2}),
$\lan L_{-,\nu} v_n,v_n\ran\geq0$ and the compact embedding in $L^q$ for functions in $X$ with $q \geq 2$ (see Lemma \ref{Lem:compact0}), it is found that
$$
\aligned
0&\leq \lan L_{-,\nu} v_0,v_0\ran_{L^2}
\leq \liminf\limits_{n\rightarrow\infty}\left(\|v_n\|_{H^1}^2+B_1\left (Q_{\omega,\nu}^2,v_n^2\right )
-B_2\left (Q_{\omega,\nu}^2,v_n^2\right )-{\nu\int_{\R^2}
Q_{\omega,\nu}^{p-1}v_n^{2}{\rm d}x}\right)\\
&=\lim\limits_{n\rightarrow\infty}\lan L_{-,\nu} v_n,v_n\ran_{L^2}=0,
\endaligned
$$
which implies that $\lan L_{-,\nu} v_0,v_0\ran_{L^2}=0$. Thus, Fatou's lemma yields that
$$
\aligned
&\|v_0\|^2_{H^1}+B_1\left (Q_{\omega,\nu}^2,v_0^2\right )
\leq \liminf\limits_{n\rightarrow\infty}\left(\|v_n\|_H^2+B_1\left (Q_{\omega,\nu}^2,v_n^2\right )\right)\leq \limsup\limits_{n\rightarrow\infty}\left(\|v_n\|_H^2+B_1\left (Q_{\omega,\nu}^2,v_n^2\right )\right)\\
&=\lim\limits_{n\rightarrow\infty}\left(\lan L_{-,\nu} v_n,v_n\ran+B_2\left (Q_{\omega,\nu}^2,v_n^2\right )
+\nu{\int_{\R^2}
Q_{\omega,\nu}^{p-1}v_n^{2}{\rm d}x}\right)\\
&=\lan L_{-,\nu}v_0,v_0\ran+B_2\left (Q_{\omega,\nu}^2,v_0^2\right )
+\nu{\int_{\R^2}
Q_{\omega,\nu}^{p-1}v_n^{2}{\rm d}x}\\
&=\|v_0\|^2_{H^1}+B_1\left (Q_{\omega,\nu}^2,v_0^2\right )\quad \mbox{or} \quad \lim \limits_{n\rightarrow\infty}\left(\|v_n\|_{H^1}^2+B_1\left (Q_{\omega,\nu}^2,v_n^2\right)\right)  = \|v_0\|^2_{H^1}+B_1\left (Q_{\omega,\nu}^2,v_0^2\right)\, ,
\endaligned
$$
which, together with the weak convergence of {$v_n$ in $\tilde{X}$ to $v_0$}, gives strong convergence of $v_n $ to $v_0$
under the norm $\left (  \| \cdot \|^2_{H^1} + B_1 (Q^2_{\omega, \nu}, \cdot^2 )\right ) ^{1/2}$,
i.e., $\lim_{n\rightarrow \infty} \left ( \| v_n - v_0  \|^2_{H^1} + B_1 \left (Q^2_{\omega, \nu}, (v_n - v_0)^2 \right ) \right ) = 0 $.
Hence, Lemma \ref{Lem:cw1} yields that $v_n\rightarrow v_0$ in $\tilde{X}$ with
$\lan L_{-,\nu} v_0,v_0\ran_{L^2}=0$ and $\|v_0\|_2^2=1$. Note that
$\lan L_{-,\nu} u ,u\ran_{L^2}\geq 0 $ for {$u\in\tilde{X}$ }and $v_0$ attains the minimum
with $ \|v_0\|_2 = 1 $ and $(v_0, Q_{\omega, \nu} )_{H^1} = 0$. Thus, there exist
two Lagrange multipliers $\alpha_1,\alpha_2\in\R$ such that $v_0$ is a critical point of
$$
\lan L_{-,\nu} u,u\ran_{L^2}=\alpha_1\| u\|^2_2 +\alpha_2(Q_{\omega,\nu},u)_{H^1},\, {u\in\tilde{X}},
$$
or $v_0$ satisfies that
for any $\varphi\in \tilde{X}$,
$$
\lan L_{-,\nu} v_0,\varphi\ran_{L^2}=\alpha_1(v_0,\varphi)_{L^2}+ (\alpha_2/2) (Q_{\omega,\nu},\varphi)_{H^1}.
$$
If $\varphi=v_0$, it is found that $\alpha_1=0$. Then,   $\varphi=Q_{\omega,\nu}$ gives that, by $L_{-,\nu}Q_{\omega,\nu}=0$,
$\lan L_{-,\nu}Q_{\omega,\nu},v_0\ran_{L^2}=\lan L_{-,\nu} v_0,Q_{\omega,\nu}\ran_{L^2}=(\alpha_2/2)\|Q_{\omega,\nu}\|_2^2 =0$. $(L_{-,\nu} v_0, \varphi )_{L^2}=0$ or $v_0$ is a weak solution of $L_{-,\nu} v_0 = 0$. The properties of $v_0$ proved above give that $v_0\in \mbox{ker}\,  L_{-,\nu}$ and the uniqueness for the kernel of $L_{-,\nu}$ implies that for some  $h\in\R\setminus\{0\}$, after possible translation, $v_0=h Q_{\omega,\nu}$, which contradicts
with $(v_0, Q_{\omega,\nu})_{H^1}=0$. This completes the proof.
\qed

Next lemma provides a similar conclusion with $v$ under the norm in $X$ instead of $ L^2$.
\begin{lemma}\label{L2-1}
For any fixed $\nu \geq 0 $ small, there exists $\delta_1>0$ such that
$
\lan L_{-,\nu}v,v\ran_{L^2}\geq \delta_1\|v\|_X^2
$
for all $v\in \tilde{X}$ satisfying $(v,Q_{\omega,\nu})_{H^1}=0$.
\end{lemma}
\Proof
Note that $\lan L_{-,\nu} v,v\ran_{L^2} \ge0$ for any $v\in \tilde{X}$. Let
$$
\varsigma:=\inf\limits_{ (v,Q_{\omega,\nu})_{H}=0,\|v\|_X=1}\lan L_{-,\nu} v,v\ran_{L^2} \geq 0\, .
$$
Hence, we need to show  $\varsigma > 0$.
If otherwise, i.e., $\varsigma=0$, then there exists sequence $\{v_n\}\subset \tilde{X}$
%%%
%\footnote{[ZWN241022] where can we find a definition of te phrase ``minimizing sequence''???
%}
%%%
such that $\|v_n\|_X=1$ and $\lan L_{-,\nu}v_n,v_n\ran\rg0$ as $n\rg\iy$. So
\begin{equation}\lab{cov1-}
\aligned
\|\na v_n\|_2^2+\omega\|v_n\|_2^2+B_1\left  (Q_{\omega,\nu}^2,v_n^2\right)=B_2\left (Q_{\omega,\nu}^2,v_n^2\right )+\nu {\int_{\R^2}
Q_{\omega,\nu}^{p-1}v_n^{2}{\rm d}x}+o_n(1).
\endaligned
\end{equation}
Then, Lemma \ref{L2} yields $v_n\rightarrow 0$ in $L^2(\R^2)$,
and the fact that $\left(\ln\left(1+\frac{1}{|x|}\right)\ast Q_{\omega,\nu}^2\right)\in L^\infty(\R^2)$ implies
\begin{align}\lab{cov1+}
\left |B_2\left (Q_{\omega,\nu}^2,v_n^2\right )\right |\leq C\|v_n\|_2^2\rightarrow0\quad\text{as}\,\,n\rightarrow\infty\, .
\end{align}
{Moreover,   since $Q_{\omega,\nu}\in L^\infty(\R^2)$, we also have  $\int_{\R^2}
Q_{\omega,\nu}^{p-1}v_n^{2}{\rm d}x\rightarrow0$} as $n\rightarrow\infty$.
Hence, (\ref{cov1-}) and (\ref{cov1+}) imply that
$$
\|\na v_n\|_2^2+\omega\|v_n\|_2^2+B_1\left (Q_{\omega,\nu}^2,v_n^2\right )\rightarrow0\quad\text{as}\,\,n\rightarrow\infty.
$$
By Lemma \ref{Lem:cw1}, it is deduced that $v_n\rightarrow 0$ in $X$ as $n\rightarrow\infty$, which is a contradiction to $\|v_n\|_X=1$. The proof is completed.
\qed

%%%%%%%%%%%%%%%%%%%%%%%%%%%%%%%%%%%%%%%%%%%%%%%%%%%%%

Now, we study the operator $L_{+,\nu}$ and its spectrum. First, let us define some subspaces of $\tilde X$.
%%In what follows, we investigate the spectrum of $L_{+,\nu}$.
Let
$$
\mathcal{V}_\nu:=\left\{u\in \tilde{X}:(u,Q_{\omega,\nu})_{L^2}=0\right\}, \quad\mathcal{V}:=\left\{u\in \tilde{X}:(u,Q_{\omega})_{L^2}=0\right\}.
$$
%As a direct consequence of conclusion (ii) of Proposition \ref{ground} and Lemma 2.4 in \cite{Marco1}, we have immediately
%\bl\label{zhengxiajie}
%$
%\inf_{u\in\mathcal{V}}\lan L_+u,u\ran=0.
%$
%\el

By similar notations used in \cite{Marco1}, define four functions
$$
\aligned
\Theta_{j,\nu}(Q_{\omega,\nu}):=&\pa_{x_j}\left((\ln|x|\ast Q_{\omega,\nu}^2)Q_{\omega,\nu}\right)-\pa_{x_j}(\nu Q_{\omega,\nu}^p)\\
=&(\ln|x|\ast Q_{\omega,\nu}^2)\pa_{x_j}Q_{\omega,\nu}+Q_{\omega,\nu}\pa_{x_j}(\ln|x|\ast Q_{\omega,\nu}^2)-\nu\pa_{x_j}( Q_{\omega,\nu}^p)\\
\Theta_j(Q_\omega):=&\pa_{x_j}((\ln|x|\ast Q_\omega^2)Q_\omega)
=(\ln|x|\ast Q_\omega^2)\pa_{x_j}Q_\omega+Q_\omega\pa_{x_j}(\ln|x|\ast Q_\omega^2)
\endaligned
$$
for $j = 1,2 $
and two subspaces
$$
\aligned
&\mathcal{W}_\nu:=\left\{u\in \tilde{X}:(u,Q_{\omega,\nu})_{L^2}=(u,\Theta_j(Q_{\omega,\nu}))_{L^2}=0, \  j =1,2\right\}\\
&\mathcal{W}:=\left\{u\in \tilde{X}:(u,Q_\omega)_{L^2}=(u,\Theta_j(Q_\omega))_{L^2}=0\, , \  j = 1,2\right\},
\endaligned
$$
where $Q_\omega$ is the unique radially symmetric positive solution of equation \re{eqn:2sn} with $\nu=0$. Then, the following lemma holds.
\bl\label{X-jiajie}
$
\inf_{u\in\mathcal{W}}\frac{\lan L_+u,u\ran_{L^2}}{\|u\|_{X}^2}>0,
$
 where operator $L_+$ is defined in (\ref{eqn:L+}).
\el

\bp
We first give an inequality that will be proved later, i.e.,
\be\lab{positive}
\inf_{u\in\mathcal{W}}\frac{\lan L_+u,u\ran}{\|u\|_{L^2_w}^2}>0,
\ee
where
$$
\|u\|_{L^2_w}^2:=(u,u)_{L^2_w}:=\int_{\R^2}[1+\ln(1+|x|)]|u|^2{\rm d}x.
$$
If \re{positive} holds, then the lemma can be proved as follows.
If otherwise, then there exists a sequence $\{u_n\}\subset\mathcal{W}$ such that
$\|u_n\|_X=1$
and $\lan L_+u_n,u_n\ran_{L^2}\rg0$ as $n\rg\iy$.
Then, \re{positive} implies that $\|u_n\|_{L^2_w}\rightarrow0$ and
$u_n\rightharpoonup 0$ in $\tilde{X}$ as $n\rightarrow\infty$.
By Lemma \ref{Lem:compact0}(i), (\ref{eqn:es1}) and (\ref{eqn:es2}),
it is found that, up to a subsequence, as $n\rightarrow\infty$,
\begin{equation}\label{eqn:L+1}
\aligned
\bigg|\int_{\R^2}\bigg(\ln|x|\ast Q_\omega^2\bigg)u_n^2{\rm d}x\bigg|&\leq \left |B_1\left (Q_\omega^2,u_n^2\right )\right |+\left |B_2\left (Q_\omega^2,u_n^2\right )\right |\\
&{\leq\|Q_\omega\|^2_2\|u_n\|_w ^2+\|u_n\|^2_2\|Q_\omega\|_w ^2
+C\|Q_\omega\|^2_{\frac{8}{3}}\|u_n\|^2_{\frac{8}{3}}\rightarrow0,}\\
\bigg|\int_{\R^2}\bigg(\ln|x|\ast Q_\omega u_n\bigg)Q_\omega u_n{\rm d}x\bigg|&\leq |B_1(Q_\omega u_n,Q_\omega u_n)|+|B_2(Q_\omega u_n,Q_\omega u_n)|\\
&{\leq2\|Q_\omega\|_w \|Q_\omega\|_2\|u_n\|_w \|u_n\|_2+C\|Q_\omega\|^2_{\frac{8}{3}}
\|u_n\|^2_{\frac{8}{3}}\rightarrow 0,}
\endaligned
\end{equation}
where $Q_\omega\in \tilde{X}$ is used.
From the definition of operator $L_+$, it is obtained that as $n\rightarrow \infty$,
\begin{align*}
\lan L_+u_n,u_n\ran_{L^2}&=\|\na u_n\|_2^2+\omega\|u_n\|_2^2+\int_{\R^2}\bigg(\ln|x|\ast Q_\omega^2\bigg)u_n^2{\rm d}x
+2\int_{\R^2}\bigg(\ln|x|\ast Q_\omega u_n\bigg)Q_\omega u_n{\rm d}x\\
&=\|\na u_n\|_2^2+o_n(1) \rightarrow 0,
\end{align*}
which, together with the fact that $\|u_n\|_{L^2_w}\rightarrow0$, yields that $u_n\rightarrow0$ in $\tilde{X}$, which  contradicts the fact that $\|u_n\|_X=1$. Hence, the lemma is proved.

Next, we prove \re{positive} holds.   Note that Lemma \ref{X-compact} implies that
 $\inf_{\varphi\in\mathcal{W}}\lan L_+\varphi,\varphi\ran\ge0$ because $\mathcal{W}\subset\mathcal{V}$. Thus, if \re{positive} fails,
there exists a minimizing sequence $\{u_n\}\subset \mathcal{W}$ with $\|u_n\|_{L^2_w}=1$ such that $\lan L_+u_n,u_n\ran_{L^2}\rg0$ as $n\rg\iy$. Hence,
\begin{equation}\lab{cov1--}
\aligned
\|\na u_n\|_2^2+\omega\|u_n\|_2^2& +B_1\left (Q_\omega^2,u_n^2\right )+2B_1(Q_\omega u_n,Q_\omega u_n) \\
& =B_2\left (Q_\omega^2,u_n^2\right )+2B_2(Q_\omega u_n,Q_\omega u_n)+o_n(1).
\endaligned
\end{equation}
 (\ref{eqn:es1}) gives $|B_1(Q_\omega u_n,Q_\omega u_n)|\leq {2\|Q_\omega\|_w \|Q_\omega\|_2\|u_n\|_w \|u_n\|_2}$.
Moreover,
 (\ref{eqn:es2}) and the Gagliardo-Nirenberg inequality yield
\begin{equation}\lab{cov1}
\aligned
&|B_2(Q_\omega u_n,Q_\omega u_n)|\leq C\|Q_\omega u_n\|^2_{4/3}\leq C\|Q_\omega\|_{8/3}^2\|u_n\|_{8/3}^2
\leq C\|\nabla u_n\|_{2}^{\frac12}\|u_n\|_{2}^{\frac32},\\
&\left |B_2\left (Q_\omega^2,u_n^2\right )\right |\leq C\|Q_\omega^2\|_{4/3}\|u_n^2\|_{4/3}\leq C\|Q_\omega\|_{8/3}^2\|u_n\|_{8/3}^2\leq C\|\nabla u_n\|_{2}^{\frac12}\|u_n\|_{2}^{\frac32}.
\endaligned
\end{equation}
By \re{cov1--} with \re{cov1}, it is obtained that $\{u_n\}$ is bounded in $X$ and, up to a subsequence, there exists $u\in \tilde{X}$ such that $u_n\rightharpoonup u$ weakly in $\tilde{X}$ as $n\rg\iy$.
Similar to (\ref{cov-Xguji1}) and (\ref{cov-Xguji2}), it follows from (\ref{eqn:es1})-(\ref{eqn:es2}), Lemma \ref{Lem:cw2} and Lemma \ref{Lem:compact0}(i) that, up to a subsequence again,
\begin{equation}\label{cov-Xguji1-}
\aligned
  &\lim_{n\rightarrow\infty}B_2\left (Q^2_{\omega},u_n^2\right )=B_2\left (Q^2_{\omega},u^2\right)\, ,\\
  &\lim_{n\rightarrow\infty}B_2(Q_{\omega}u_n,Q_{\omega}u_n)=B_2(Q_{\omega}u,Q_{\omega}u)
\endaligned
\end{equation}
and
\begin{equation}\label{cov-Xguji2-}
\aligned
 &\lim_{n\rightarrow\infty}\left|B_1(Q_{\omega}u_n,Q_{\omega}u_n)
  -B_1(Q_{\omega}u,Q_{\omega}u)\right|\\
  &=\lim_{n\rightarrow\infty}\left|B_1(Q_{\omega}u_n-Q_{\omega}u,Q_{\omega}u_n)
  +B_1(Q_{\omega}u,Q_{\omega}u_n-Q_{\omega}u)\right|=0.
\endaligned
\end{equation}
Then, by (\ref{cov1--}), Lemma \ref{Lem:cw1} and a similar argument used in the proof of Lemma \ref{L2}, it is found that $u_n\rightarrow u$ in $\tilde{X}$.
Hence, $u\in\mathcal{W}$ is a minimizer of $\lan L_+ v,v \ran_{L^2}$ under the constraints $v\in {\mathcal W} $ and $\| v\|_{L^2_w} = 1$, and
there exist Lagrange multipliers $\la,\mu,\la_{1},\la_{2}$ such that
\be\lab{lagr}
\lan L_+u,\eta\ran_{L^2} =\la(u,\eta)_{L^2_w}+\mu(Q_\omega,\eta)_{L^2}+\sum_{i=1}^2\la_{i}\int_{\R^2}\Theta_i(Q_\omega)\eta {\rm d}x~~~
\forall \eta\in \tilde{X}.
\ee
{Taking $\eta=u$} in (\ref{lagr}) implies that $\la=0$ because of $u\in\mathcal{W}$.
For $j=1,2$,  the fact that $\mbox{ker}\, L_+ = {\rm span}\{\pl_{x_1}Q_\omega,\pl_{x_2}Q_\omega\}$ yields
$$
\lan L_+u,\pl_{x_j}Q_\omega\ran=\lan L_+\pl_{x_j}Q_\omega,u\ran=0.
$$
Hence,  by choosing $\eta=\pl_{x_j}Q_\omega$ in (\ref{lagr}) with $j=1,2$, it is obtained that
\begin{equation}\label{eqn:L+2}
\aligned
0&=\sum_{i=1}^2\la_{i}\int_{\R^2}\Theta_i(Q_\omega)\pl_{x_j}Q_\omega {\rm d}x
=\la_{j}\int_{\R^2}\Theta_j(Q_\omega)\pl_{x_j}Q_\omega {\rm d}x,
\endaligned
\end{equation}
where we have used the fact that
$(Q_\omega,\pl_{x_j}Q_\omega)=0,\,\,j=1,2,$
and for $i\not=j$,
\begin{equation}\label{eqn:L+3-}
\aligned
\int_{\R^2}\Theta_i(Q_\omega)\pl_{x_j}Q_\omega {\rm d}x&
=\int_{\R^2}\pl_{x_i}\left (\left (\ln|x|\ast Q_\omega^2\right )Q_\omega\right )\pl_{x_j}Q_\omega {\rm d}x\\
&=-\int_{\R^2}\pl_{x_i}(-\DD Q_\omega+\omega Q_\omega)\pl_{x_j}Q_\omega {\rm d}x\\
&=-\left ( \int_{\R^2}\na \pl_{x_i}Q_\omega\na \pl_{x_j}Q_\omega {\rm d}x+\int_{\R^2}\omega\pl_{x_i}Q_\omega\pl_{x_j}Q_\omega {\rm d}x\right ) =0\, ,
\endaligned
\end{equation}
where the fact that the partial derivative of a radially symmetric function in $\mathbb{R}^2$ is odd in that variable has been used.
%where we have the fact that $Q_\omega$ is symmetric radially.
 Moreover,
\begin{equation}\label{eqn:L+3}
\aligned
\int_{\R^2}\Theta_j(Q_\omega)\pl_{x_j}Q_\omega {\rm d}x&=\int_{\R^2}\pl_{x_j}\left (\left (\ln|x|\ast Q_\omega^2\right )Q_\omega\right )\pl_{x_j}Q_\omega {\rm d}x\\
&=-\int_{\R^2}\pl_{x_j}(-\DD Q_\omega+\omega Q_\omega)\pl_{x_j}Q_\omega {\rm d}x=-\|\pl_{x_j} Q_\omega\|^2_{H^1}.
\endaligned
\end{equation}
Thus, (\ref{eqn:L+2}) and (\ref{eqn:L+3}) imply that $\la_{j}=0$ for  $j=1,2$
and \re{lagr} is changed into
\begin{equation}\label{eqn:L+6}
\lan L_+u,\eta\ran_{L^2}=\mu(Q_\omega,\eta)_{L^2}~~~
\forall \eta\in \tilde{X}.
\end{equation}
If  $\mu=0$, then $u\in\mbox{ker}\, L_+$ and there exists a vector $\tilde{\beta}=(\beta_1,\beta_2)\in\R^2$ such that
  $u=\tilde{\beta}\cdot \na Q_\omega$.
Then, by  (\ref{eqn:L+3-}) and (\ref{eqn:L+3}), $u \in \mathcal{W} $ yields that
$$
0= \int_{\R^2}\Theta_j(Q_\omega)u {\rm d}x=\int_{\R^2}\Theta_j(Q_\omega)(\tilde{\beta}\cdot \na Q_\omega){\rm d}x=-\beta_j\|\pa_{x_j}Q_\omega\|_H^2,\,\,j=1,2.
$$
Or $\tilde{\beta}=0$ and $u=0$, which contradicts the fact that $\|u\|_{L^2_w} =1$. Thus, the proof of the lemma is complete if $\mu=0$ is proved.

To prove $\mu =0$, note that
$$
\DD(x\cdot\na Q_\omega)=2\DD Q_\omega+\sum_{j=1}^2x_j\pl_{x_j}\DD Q_\omega.
$$
The exponential decay of $Q_\omega$ at infinity gives
\begin{equation}\lab{zuhe1}
\aligned
&L_+(x\cdot\na Q_\omega)=-2\DD Q_\omega+\sum_{j=1}^2x_j\pl_{x_j}\left[-\DD Q_\omega+\omega Q_\omega+(\ln|x|\ast Q_\omega^2)Q_\omega\right]\\
&\qquad +\left[\ln\frac{1}{|x|}\ast Q_\omega^2\right]Q_\omega-\|Q_\omega\|_2^2Q_\omega
=-2\DD Q_\omega+\left[\ln\frac{1}{|x|}\ast Q_\omega^2\right]Q_\omega-\|Q_\omega\|_2^2Q_\omega.
\endaligned
\end{equation}
Moreover, since $Q_\omega$ is the unique positive solution of \re{eqn:2sn}, it is obtained that
\begin{align}\lab{zuhe2}
L_+Q_\omega=2(\ln|x|\ast Q_\omega^2)Q_\omega.
\end{align}
Thus,  \re{zuhe1}-\re{zuhe2} yield that
$L_+(aQ_\omega+bx\cdot\na Q_\omega))=\mu Q_\omega$, where
$$
a=\frac{-3\mu}{4\omega+2\|Q_\omega\|_2^2},\quad\quad b=\frac{-\mu}{2\omega+\|Q_\omega\|_2^2}.
$$
%Recall that there is a unique radial solution $\rho\in X$ (see \cite{Wein1})of
%$$
%L_+(\rho)=-|x|^2Q.
%$$
Since $L_+u=\mu Q_\omega$, it is deduced that
$$
u=\vartheta\cdot\na Q_\omega+aQ_\omega+bx\cdot\na Q_\omega
$$
for some $\vartheta\in\R^2$.
Since $u\in \mathcal{W}$, (\ref{eqn:L+3-}) and (\ref{eqn:L+3}) imply that
$$
\aligned
0=&\int_{\R^2}u\Theta_j(Q_\omega){\rm d}x=\frac{-3\mu}{4\omega+2\|Q_\omega\|_2^2}\int_{\R^2}\Theta_j(Q_\omega)Q_\omega {\rm d}x\\&
-\frac{\mu}{2\omega+\|Q_\omega\|_2^2}
\int_{\R^2}\Theta_j(Q_\omega)x\cdot\na Q_\omega {\rm d}x
-\vartheta_j\|\pl_{x_j}Q_\omega\|_{H^1}^2,\,\,j=1,2.
\endaligned
$$
It is straightforward to check that $\int_{\R^2}\Theta_j(Q_\omega)Q_\omega {\rm d}x=0$. Since
the function $(\ln|x|\ast Q_\omega^2)Q_\omega$ is radially symmetric
with respect to $x$, it is found that
$$
\aligned
&\int_{\R^2}\Theta_j(Q_\omega)x\cdot\na Q_\omega {\rm d}x=\int_{\R^2} \left ( \pl_{x_j}\left [\left (\ln|x|\ast Q_\omega^2\right )Q_\omega\right ] \right )  |x|Q'_\omega(|x|){\rm d}x=0\, ,
\endaligned
$$
where the fact that $x_j$-derivative of a radially symmetric function is odd in terms of $x_j$ has been used.
Hence,  $\vartheta=0$ is derived and $u=aQ_\omega+bx\cdot\na Q_\omega$.
Since $u\in \mathcal{W}$, it is deduced that
$$
0=(u,Q_\omega)_{L^2} =(aQ_\omega+bx\cdot\na Q_\omega, Q_\omega)_{L^2} =(a-b)\|Q_\omega\|_2^2,
$$
which gives $a= b$ and $\mu=0$.
\ep

\bl\label{X-jiajienu}
There exists $\nu_3^*>0$ such that $
\inf_{u\in\mathcal{W}_\nu}\frac{\lan L_{+,\nu}u,u\ran_{L^2}}{\|u\|_{X}^2}>0
$ for all $\nu\in[0,\nu_3^*)$.
\el

\bp
The lemma will be proved by a contradiction if otherwise.
Assume that
there exist a sequence $\{\nu_n\}$ with $\nu_n\rightarrow0^+$ and
a sequence $\{u_n\}\subset \mathcal{V}_{\nu_n}$ with the norm $\|u_n\|_{X}=1$
such that $\liminf_{n\rightarrow\infty}\lan L_{+,\nu_n}u_n,u_n\ran_{L^2}\leq0$.
Since $\left\|u_n\right\|_{X}^2=1$,  $\{u_n\}$ is bounded in $\tilde{X}$ and
there exists a $u_0\in \tilde{X}$ such that $u_n\rightharpoonup u_0$ in $\tilde{X}$ and
$u_n\rightarrow u_0, \, a.e.$
%%%
%\footnote{[ZWN] DO you use `a.e.' for a zero-measure set?
%}
%%%
with its derivatives.
It follows from (\ref{cov-X+}) that $Q_{\omega,\nu_n}\rightarrow Q_{\omega}$ in $\tilde{X}$, and then
$Q_{\omega,\nu_n}u_n\rightarrow Q_{\omega}u_0$ in $\tilde{X}$. By applying  almost the same
%%%
%\footnote{[ZWN] revised from
%`identical'
%}
%%
arguments as those for (\ref{cov-Xguji1}) and (\ref{cov-Xguji2}),
it is obtained that
\begin{equation}\label{eqn-spec}
\aligned
  0&\geq\liminf_{n\rightarrow\infty}\langle L_{+,\nu_n}u_n,u_n\rangle _{L^2}\\
  &=\liminf_{n\rightarrow\infty}\left(\left\|u_n\right\|_{H^1}^2+{B_1\left (Q_{\omega,\nu_n}^2,u_n^2\right )-
  B_2\left (Q^2_{\omega},u_0^2\right )+2B_1(Q_{\omega}\phi_0,Q_{\omega}u_0)
  -2B_2(Q_{\omega}u_0,Q_{\omega}u_0)}\right)\\
  &\geq\left\|u_0\right\|_{H^1}^2+B_0\left (Q^2_{\omega},u_0^2\right )+2B_0(Q_{\omega}u_0,Q_{\omega}u_0) =\langle L_{+}u_0,u_0\rangle_{L^2},
\endaligned
\end{equation}
which implies that $u_0\in \tilde{X}\setminus\{0\}$ by using a similar argument to \eqref{eqn:+++}. Moreover, $(u_n,Q_{\omega,\nu_n})_{L^2} =0$
gives $(u_0,Q_{\omega})_{L^2}=0$, and $(u_n,\Theta_{j,\nu_n}(Q_{\omega,\nu_n}))_{L^2}=0$
yields $(u_0,\Theta_{j}(Q_{\omega}))_{L^2}=0, j=1,2$. Therefore, (\ref{eqn-spec}) contradicts to the conclusion of Lemma \ref{X-jiajie} and
this completes the proof.
\ep

We are ready to prove the following proposition.
\bo\lab{zui}
If $w=u+iv\in \tilde{X}(\R^2,\mathbb{C})$ with $u,v\in \tilde{X}$ satisfying
$$
\|w+Q_{\omega,\nu}\|_{L^2(\R^2,\mathbb{C})}=\|Q_{\omega,\nu}\|_2  \mbox{ and }
(u,\Theta_{j,\nu}(Q_{\omega,\nu}))_{L^2}=0,\,j=1,2\, ,
$$
then, for any fixed $\nu\in[0,\nu_3^*)$ with $\nu_3^*$ given in Lemma \ref{X-jiajienu}, there exist $D,D_1,D_2>0$ independent of $\nu$ such that
$$
\lan L_{+, \nu} u,u\ran_{L^2}\ge D\|u\|_X^2-D_1\|w\|_{H^1(\R^2,\mathbb{C})}^3-D_2\|w\|_{H^1(\R^2,\mathbb{C})}^4.
$$
\eo

\bp
Since $\|w+Q_{\omega,\nu}\|_{L^2(\R^2,\mathbb{C})}=\|Q_{\omega,\nu}\|_2$, it is obtained that
$(u,Q_{\omega,\nu})_{L^2}=-\frac{1}{2}\|w\|_{L^2(\R^2,\mathbb{C})}^2$.
% Without loss of generality, assume that $\|r\|_2=1$.
Define
$$
u:=u_{\shortparallel}+u_{\perp},\,\,u_{\shortparallel}:=(u,Q_{\omega,\nu})_{L^2}\|Q_{\omega,\nu}\|_2^{-2}Q_{\omega,\nu}\, ,
$$
which gives $(u_{\perp},Q_{\omega,\nu})_{L^2}=0$.
It is found that
 $$
 \|\nabla u\|_2^2\leq 2(\|\nabla u_{\shortparallel}\|_2^2+\|\nabla u_{\perp}\|_2^2)
 =\frac{1}{2}\frac{\|w\|_{L^2(\R^2,\mathbb{C})}^4}{\|Q_{\omega,\nu}\|_2^{4}}\|\nabla Q_{\omega,\nu}\|_2^2+2\|\nabla u_{\perp}\|_2^2
 $$
 and
 $$
 \| u\|_w ^2\leq 2\left (\| u_{\shortparallel}\|_w ^2+\|u_{\perp}\|_w ^2\right )
 =\frac{1}{2}\frac{\|w\|_{L^2(\R^2,\mathbb{C})}^4}{\|Q_{\omega,\nu}\|_2^{4}}\| Q_{\omega,\nu}\|_w ^2+2\| u_{\perp}\|_w ^2.
 $$
Since $(Q_{\omega,\nu},\Theta_{j,\nu} (Q_{\omega,\nu}))_{L^2}=0$ for $j=1,2$, it is deduced from $(u,\Theta_{j,\nu} (Q_{\omega,\nu}))_{L^2}=0$ that  $(u_{\perp},\Theta_{j,\nu}(Q_{\omega,\nu}))_{L^2}=0,\,j=1,2$,
and $u_{\perp}\in\mathcal{W}_\nu$.
  Thus, above inequalities and Lemma \ref{X-jiajienu} imply that
\begin{equation}\lab{low1}
\aligned
 & \lan L_{+,\nu}u_{\perp},u_{\perp}\ran_{L^2}\ge C\|u_{\perp}\|_{X}^2
=C\left (\omega\|u\|_2^2-\omega\|u_{\shortparallel}\|_2^2+\|u_{\perp}\|_w ^2+\|\nabla u_{\perp}\|_2^2\right )\\
&\quad \geq C\left(\omega\|u\|_2^2-\omega \left |(u,Q_{\omega,\nu})_{L^2}\right |^2+\frac12\|\nabla u\|_2^2-\frac{1}{4}\frac{\|w\|_{L^2(\R^2,\mathbb{C})}^4}{\|Q_{\omega,\nu}\|_2^{4}}
\left (\|\nabla Q_{\omega,\nu}\|_2^2+\|Q_{\omega,\nu}\|_w ^2\right )+\frac12\|u\|_w ^2\right)\\
&\quad \geq  C_1\|u\|_X^2-C_2\|w\|_{L^2(\R^2,\mathbb{C})}^4
\endaligned
\end{equation}
for some $C, C_1,C_2>0$.
On the other hand, since $Q_{\omega,\nu}$ is a solution of \re{eqn:2sn}, it is found that
$$
\|\na Q_{\omega,\nu}\|_2^2+\omega\|Q_{\omega,\nu}\|_2^2
=\int_{\R^2}\left[\ln\frac{1}{|x|}\ast Q_{\omega,\nu}^2\right]Q_{\omega,\nu}^2{\rm d}x+\nu\|Q_{\omega,\nu}\|_{p+1}^{p+1}\, ,
$$
which gives
$$
\lan L_{+,\nu}Q_{\omega,\nu},Q_{\omega,\nu}\ran_{L^2}
=-2\int_{\R^2}\left[\ln\frac{1}{|x|}\ast Q_{\omega,\nu}^2\right]Q_{\omega,\nu}^2{\rm d}x
-\nu(p-1)\|Q_{\omega,\nu}\|_{p+1}^{p+1}\, .
$$
Also,  \re{eqn:es1} and \re{eqn:es3} imply that
\be\lab{low2}
\lan L_{+,\nu}u_{\shortparallel},u_{\shortparallel}\ran_{L^2}
=\frac{1}{4}\|w\|_{L^2(\R^2,\mathbb{C})}^4\|Q_{\omega,\nu}\|_2^{-4}\lan L_{+,\nu}Q_{\omega,\nu},Q_{\omega,\nu}\ran_{L^2}\ge
-C\|w\|_{L^2(\R^2,\mathbb{C})}^4.
\ee
Finally, since $Q_{\omega,\nu}$ satisfies \re{eqn:2sn}, by $(u_{\perp},Q_{\omega,\nu})_{L^2}=0$, one obtains that  for some $C>0$,
\begin{align*}
&\lan L_{+,\nu}u_{\perp},Q_{\omega,\nu}\ran_{L^2} =\lan L_{+,\nu}Q_{\omega,\nu},u_{\perp}\ran_{L^2}\\
&=
\int_{\R^2}\left (\na Q_{\omega,\nu}\cdot \na u_{\perp}+\omega Q_{\omega,\nu} u_{\perp}\right ){\rm d}x
-3\int_{\R^2}\left[\ln\frac{1}{|x|}\ast Q_{\omega,\nu}^2\right]Q_{\omega,\nu} u_{\perp}{\rm d}x
-\nu p\int_{\R^2}Q_{\omega,\nu}^{p}u_{\perp}{\rm d}x \\
&=-2\int_{\R^2}\left[\ln\frac{1}{|x|}\ast Q_{\omega,\nu}^2\right]Q_{\omega,\nu} u_{\perp}{\rm d}x
-\nu (p-1)\int_{\R^2}Q_{\omega,\nu}^{p}u_{\perp}{\rm d}x\geq
-C\|w\|_{H^1(\R^2,\mathbb{C})}.
\end{align*}
Thus,
\begin{equation}\lab{low3}
\lan L_{+,\nu}u_{\perp},u_{\parallel}\ran_{L^2}
=-\frac{1}{2}\|w\|_2^2\lan L_{+,\nu}u_{\perp},Q_{\omega,\nu}\ran_{L^2}\le C\|w\|_{H^1(\R^2,\mathbb{C})}^3.
\end{equation}
 The inequalities \re{low1}-\re{low3} yield the conclusion of this proposition.
\ep

%%%%%%%%%%%%%%%%%%%%%%%%%%%%%%%%%%%%%%%%%%%%%%%%%%%%%%

\subsection{Proof of orbital stability, i.e., Theorem \ref{Th3}}

This subsection provides the proof of orbital stability of $Q_{\omega, \nu}$ for \eqref{eqn:q1}.

%%Before proceeding, let us introduce a notation.
For any fixed real number $\e>0$, define
$$
\Omega_{Q_{\omega,\nu}}^{\e}:=\{z\in \tilde{X}(\R^2,\mathbb{C});\, \mbox{dist}(z,\Omega_{Q_{\omega,\nu}})<\e\}
$$
as the $\e$-neighborhood of $\Omega_{Q_{\omega,\nu}}$ in $X$, where $\Omega_{Q_{\omega,\nu}}:=\{e^{i\theta}Q_{\omega,\nu}(\cdot-y):\,\,\theta\in\R,y\in\R^2\}$ and dist$(\cdot, \cdot)$ is taken in $X$.

\begin{lemma}\label{lem-inf}
For any fixed $\nu\in[0,\nu_3^*)$ with $\nu_3^*$ given in Lemma \ref{X-jiajienu}, there exist $M,h,{d_1}>0$, depending only on
$\|Q_{\omega,\nu}\|_X$, such that for all $\e\in(0,h)$ and
$z\in \Omega_{Q_{\omega,\nu}}^{\e}$ with $\|z\|_{X(\R^2,\mathbb{C})}\leq M$,
there are $\tilde{y}\in\R^2$ with $|\tilde{y}|\leq {d_1}$ and $\tilde{\theta}\in\R$ satisfying
 $$
 \inf\limits_{y\in\R^2,\theta\in\R}\|e^{-i\theta}z(\cdot+y)-Q_{\omega,\nu}\|_{X(\R^2,\mathbb{C})}^2=
 \|e^{-i\tilde{\theta}}z(\cdot+\tilde{y})-Q_{\omega,\nu}\|_{X(\R^2,\mathbb{C})}^2.
 $$
\end{lemma}
\Proof
Given $M>0$, for any fixed $z\in X(\R^2,\mathbb{C})$ with $\|z\|_{X(\R^2,\mathbb{C})}\leq M$, it is obtained that
$$
\aligned
&\|e^{-i\theta}z(\cdot+y)-Q_{\omega,\nu}\|_{X(\R^2,\mathbb{C})}^{2}\\
& \quad =\|Q_{\omega,\nu}\|_{X}^{2}+
\|z(\cdot+y)\|_{X(\R^2,\mathbb{C})}^{2}-2\mathfrak{Re}\int_{\R^2}(\omega+\ln(1+|x|))e^{i\theta}\bar{z}(\cdot+y)Q_{\omega,\nu} {\rm d}x\\
&\qquad \qquad \qquad -2\mathfrak{Re}\int_{\R^2}e^{i\theta}\nabla\bar{z}(\cdot+y)\cdot \nabla Q_{\omega,\nu} {\rm d}x\\
& \quad =\|Q_{\omega,\nu}\|_{X}^{2}+
\|z(\cdot+y)\|_{X(\R^2,\mathbb{C})}^{2}-2\mathfrak{Re}\int_{\R^2}(\omega+\ln(1+|x|))e^{i\theta}\bar{z}(\cdot+y)Q_{\omega,\nu} {\rm d}x\\
&\qquad \qquad \qquad + 2\mathfrak{Re}\int_{\R^2}e^{i\theta}\bar{z}(\cdot+y)\Delta Q_{\omega,\nu} {\rm d}x\\
& \quad =:\|Q_{\omega,\nu}\|_{X}^{2}+
\|z(\cdot+y)\|_{X(\R^2,\mathbb{C})}^{2}-A_1+A_2,~~~
\forall y\in\R^2,\theta\in\R.
\endaligned
$$
Since $Q_{\omega,\nu}\in X$  exponentially decays to zero at infinity,
the families of functions $\{\ln(1+|\cdot-y|)Q_{\omega,\nu}(\cdot-y)\}_{y\in\R^2}$ and $\{\Delta Q_{\omega,\nu}(\cdot-y)\}_{y\in\R^2}$ converge
pointwise to zero as $|y|\rightarrow\infty$.
By the Lebesgue dominant convergence theorem, it is found that $A_1,A_2\rightarrow0$ as $|y|\rightarrow\infty$. Here, we note that the convergence may not be uniform with respect to $z$ with $\|z\|_{X(\R^2,\mathbb{C})}\leq M$ if Lebesgue dominant convergence theorem is used. To
obtain the uniform convergence to zero, we can use (\ref{expo}) and Lemma \ref{Lem:int0} in Appendix to estimate $A_1$ as follows,
\begin{align*}
2 & \left | \mathfrak{Re}\int_{\R^2}(\omega+\ln(1+|x|))e^{i\theta}\bar{z}(\cdot+y)Q_{\omega,\nu} {\rm d}x\right |\\
&\qquad \leq  C \int_{\R^2}(1 +\ln(1+|x-y|))   | \bar{z}(x)| \left | Q_{\omega,\nu} (x - y ) \right | {\rm d}x \\
&\qquad \leq  C\|z\|_{X(\R^2,\mathbb{C})} \left (
\int_{\R^2}\frac{ (1 +\ln(1+|x-y|)) ^2}{ ( 1 + \ln ( 1 + |x| ) ) } \left | Q_{\omega,\nu} (x - y ) \right |^2 {\rm d}x \right )^{1/2} \\
&\qquad \leq  C\|z\|_{X(\R^2,\mathbb{C})} \left (
\int_{\R^2}\frac{ e^{-(1/4) |x-y|}}{ ( 1 + \ln ( 1 + |x| ) ) } {\rm d}x \right )^{1/2}\\
 &\qquad\leq C\|z\|_{X(\R^2,\mathbb{C})}(1+\ln(1+ |y|))^{-1/2}\, ,
\end{align*}
which gives that $|A_1| \rightarrow 0$ uniformly as $ | y | \rightarrow \infty$ if $\|z\|_{X(\R^2,\mathbb{C})} \leq M$. $A_2$ can be estimated similarly.

Consider the function $g:\R^2\times\R\rightarrow\R$ defined by
$$
g(y,\theta)=\|e^{-i\theta}z(\cdot+y)-Q_{\omega,\nu}\|_{X(\R^2,\mathbb{C})}^{2},
$$
which is $2\pi$-periodic with respect to $\theta$. Then,
$\lim_{|y|\rightarrow+\infty}g(y,\theta)>\frac{3}{4}\|Q_{\omega,\nu}\|_{H^1}^2$
%%
%\footnote{[ZWN] ``uniformly for'' makes a confusion. `for ...' means `holds for each point'.
%Maybe we want to say `uniformly with respect to ...''? We need to check all `uniform'.
%}
%%
\textcolor[rgb]{0.00,0.07,1.00}{uniformly with respect to} $\theta\in[0,2\pi)$ and $z\in \Omega_{Q_{\omega,\nu}}^{\e}$ with $\|z\|_{X(\R^2,\mathbb{C})}\leq M$.
Thus, there exists $d_0 >0$, which is dependent on $Q_{\omega,\nu}$ but independent of
$z\in \Omega_{Q_{\omega,\nu}}^{\e}$ satisfying $\|z\|_{X(\R^2,\mathbb{C})}\leq M$, such that for any $d \geq d_0$,
\begin{equation}\lab{infH++0}
g(y,\theta)>\frac{1}{2}\|Q_{\omega,\nu}\|_{H^1}^2 \quad\text{for}\,\, |y|\geq d.
\end{equation}
Moreover, for any $|y|\leq d$ and $\theta\in[0,2\pi)$, one obtains that
\begin{equation}\lab{infH++}
\aligned
g(y,\theta)=&\|e^{-i\theta}z(\cdot+y)-Q_{\omega,\nu}\|_{X(\R^2,\mathbb{C})}^{2}\\
=&\|z(\cdot+y)-e^{i\theta}Q_{\omega,\nu}\|_{X(\R^2,\mathbb{C})}^{2}\\
=&\|z-e^{i\theta}Q_{\omega,\nu}(\cdot-y)\|_{H^1(\R^2,\C)}^{2}+\int_{\R^2}\ln(1+|x-y|)\left | z-e^{i\theta}Q_{\omega,\nu}(\cdot-y)\right | ^2dx\\
\leq&\|z-e^{i\theta}Q_{\omega,\nu}(\cdot-y)\|_{X(\R^2,\mathbb{C})}^{2}+\int_{\R^2}\ln(1+|y|)\left | z-e^{i\theta}Q_{\omega,\nu}(\cdot-y)\right | ^2dx\\
\leq&(1+\ln(1+d))\|z-e^{i\theta}Q_{\omega,\nu}(\cdot-y)\|_{X(\R^2,\mathbb{C})}^{2}.
\endaligned
\end{equation}

On the other hand, if the function $f:\R^2\times\R\rightarrow\R$ is defined by
$$
f(y,\theta)=\|z-e^{i\theta}Q_{\omega,\nu}(\cdot-y)\|_{X(\R^2,\mathbb{C})}^{2},
$$
which is $2\pi$-periodic with respect to $\theta$,  a direct computation yields
$$
\aligned
f(y,\theta)=&\|z-e^{i\theta}Q_{\omega,\nu}(\cdot-y)\|_{X(\R^2,\mathbb{C})}^{2}\\
\geq&\|z-e^{i\theta}Q_{\omega,\nu}(\cdot-y)\|_{H^1(\R^2,\mathbb{C})}^{2}\\
=&\|z\|_{H^1(\R^2,\mathbb{C})}^{2}+
\|Q_{\omega,\nu}\|_{H^1}^{2}-2\mathfrak{Re}\int_{\R^2}\omega e^{i\theta}\bar{z}Q_{\omega,\nu}(\cdot-y) {\rm d}x
-2\mathfrak{Re}\int_{\R^2}e^{i\theta}\nabla\bar{z}\nabla Q_{\omega,\nu}(\cdot-y) {\rm d}x\\
=&\|z\|_{H^1(\R^2,\mathbb{C})}^{2}+
\|Q_{\omega,\nu}\|_{H^1}^{2}-2\mathfrak{Re}\int_{\R^2}\omega e^{i\theta}\bar{z}Q_{\omega,\nu}(\cdot-y) {\rm d}x
+2\mathfrak{Re}\int_{\R^2}e^{i\theta}\bar{z}\Delta Q_{\omega,\nu}(\cdot-y) {\rm d}x\\
=:&\|z\|_{H^1(\R^2,\mathbb{C})}^{2}+
\|Q_{\omega,\nu}\|_{H^1}^{2}-\tilde{A}_1-\tilde{A}_2,~~~
\forall y\in\R^2,\theta\in\R.
\endaligned
$$
Similar to $A_1,A_2$, the Lebesgue dominant convergence theorem implies that
$\tilde{A}_1,\tilde{A}_2\rightarrow0$ as $|y|\rightarrow\infty$, while the uniform convergence can be obtained using the condition $\|z\|_{X(\R^2,\mathbb{C})}\leq M$. Then,
$\lim_{|y|\rightarrow+\infty}f(y,\theta)>\frac{3}{4}\|Q_{\omega,\nu}\|_{H^1}^2$ uniformly for $\theta\in[0,2\pi)$ and $z\in \Omega_{Q_\omega}^{\e}$ which
satisfies $\|z\|_{X(\R^2,\mathbb{C})}\leq M$.
Hence, there exists $\tilde{d}>0$ depending only on $Q_{\omega,\nu}$ such that for any $d \geq \tilde d$, $f(y,\theta)>\frac{3}{4}\|Q_{\omega,\nu}\|_{H^1} ^2$ for $|y|\geq {d}$.
%{\color{red}\{\{[ZWN250619] The following is revised. Please check.\}»ØžŽ£ºÄúÐÞžÄµÄÒâËŒºÜ×ŒÈ·¡£\}
%}

Let  $d_1>0$ large enough so that $d_1\geq \tilde{d}$ and $d_1 \geq d_0$ and choose
$$
h:=\frac{1}{2(1+\ln(1+d_1))}\|Q_{\omega,\nu}\|_{H^1}^2 < \frac12 \|Q_{\omega,\nu}\|_{H^1}^2 \, .
$$
Then, for $\e\in(0,h)$ and $z\in \Omega_{Q_{\omega,\nu}}^{\e}$ with $\|z\|_{X(\R^2,\mathbb{C})}\leq M$,
from the definition of $\Omega_{Q_{\omega,\nu}}^{\e}$
there exist $y_*\in\R^2$ and $\theta_*\in[0,2\pi)$ with $| y_*| \leq d_1$ such that
\begin{equation}\lab{infH++1}
f ( y_* , \theta_*) = \|z-e^{i\theta_*}Q_{\omega,\nu}(\cdot-y_*)\|_{X(\R^2,\mathbb{C})}^{2}\leq\frac{1}{2(1+\ln(1+d_1))}\|Q_{\omega,\nu}\|_{H^1}^2.
\end{equation}

Thus,  \eqref{infH++} and \eqref{infH++1} yield
$$
\aligned
g(y_*,\theta_*)=&\|e^{-i\theta_*}z(\cdot+y_*)-Q_{\omega,\nu}\|_{X(\R^2,\mathbb{C})}^{2}\\
&\leq(1+\ln(1+d_1))\|z-e^{i\theta_*}Q_{\omega,\nu}(\cdot-y_*)\|_{X(\R^2,\mathbb{C})}^{2} \leq \frac{1}{2}\|Q_{\omega,\nu}\|_{H^1}^2,
\endaligned
$$
which implies by \eqref{infH++0} and continuity of $g$ that
the infimum of $g$ is achieved at the compact set $\overline{B_{d_1}(0)}\times[0,2\pi]$, which completes the proof. \qed

Furthermore, a byproduct of the proof of Lemma \ref{lem-inf} yields  the following corollary.
Note that in the minimization problem stated in the corollary with respect to the $H^1(\R^2,\mathbb{C})$ norm,
it is unnecessary to impose the additional constraint $\|z\|_{X(\R^2,\mathbb{C})}\leq M$
for the condition of $z\in \Omega_{Q_{\omega,\nu}}^{\e}$, since translation invariance holds in this setting. However, the infimum of $g$ need not be attained on a compact set.
\begin{corollary}\label{cor-infH}
For any fixed $\nu\in[0,\nu^*)$ with $\nu^*$ given in Lemma \ref{X-jiajienu}, there exists $\tilde{R}$, depending  on $\|Q_{\omega,\nu}\|_{H^1}$ only,
such that for each $\e\in(0,\tilde{R})$ and $z\in \Omega_{Q_{\omega,\nu}}^{\e}$
there are $\theta^*\in\R$  and $y^*\in\R^2$ satisfying
 $$
 \inf\limits_{y\in\R^2,\theta\in\R}\|z-e^{i\theta}Q_{\omega,\nu}(\cdot-y)\|_{H^1(\R^2,\mathbb{C})}^2=
 \|z-e^{i\theta^*}Q_{\omega,\nu}(\cdot-y^*)\|_{H^1(\R^2,\mathbb{C})}^2.
 $$
\end{corollary}

\begin{lemma}\label{lem-xiajie}
There eixsts $0<\nu^{**}\leq \nu^*$ with $\nu^{*}$ given in Lemma \ref{X-jiajienu} such that for any fixed $\nu\in[0,\nu^{**})$, there exist $M>0$ and $\e>0$, depending only on $Q_{\omega,\nu}$, such that
$$
\mathcal{E}_\nu(z)-\mathcal{E}_\nu(Q_{\omega,\nu})\geq C \inf\limits_{y\in\R^2,\theta\in\R}\|e^{-i\theta}z(\cdot+y)
-Q_{\omega,\nu}\|_{X(\R^2,\mathbb{C})}^2
$$
for all $z\in \Omega_{Q_{\omega,\nu}}^{\e}$ which satisfies $\|z\|_{X(\R^2,\mathbb{C})}\leq M$
and $\|z\|_{2}=\|Q_{\omega,\nu}\|_2$, where $C$ is independent of $\nu$ and $\mathcal{E}_\nu (z) $ is defined in \eqref{1.9.1}.
\end{lemma}
\Proof
For each $\e>0$,
let $z\in \Omega_{Q_{\omega,\nu}}^{\epsilon}$  be a function such that $\|z\|_2=\|Q_{\omega,\nu}\|_2$.
By {Corollary \ref{cor-infH}}, there exist $y^*\in\R^2$ and $\theta^*\in[0,2\pi)$ such that
\begin{equation}\lab{infH}
 \inf\limits_{y\in\R^2,\theta\in\R}\|z-e^{i\theta}Q_{\omega,\nu}(\cdot-y)\|_{H^1(\R^2,\mathbb{C})}^2
 =\|z-e^{i\theta^*}Q_{\omega,\nu}(\cdot-y^*)\|_{H^1(\R^2,\mathbb{C})}^2.
\end{equation}
Let $w(x):=e^{-i\theta^*}z(x+y^*)-Q_{\omega,\nu}(x)$ and  $w=u+vi$,
where $u$ and $v$ denote the real and imaginary parts of $w$,  respectively.
Arguing similarly to Lemma \ref{lem-inf}, we obtain that
$$
\aligned
{\tilde{f}}&(y,\theta):=\|z-e^{i\theta}Q_{\omega,\nu}(\cdot-y)\|_{H^1(\R^2,\mathbb{C})}^{2}\\
&=\|Q_{\omega,\nu}\|_{H}^{2}+
\|z\|_{H^1(\R^2,\mathbb{C})}^{2}-2\, \mathfrak{Re}\int_{\R^2}e^{i\theta}\bar{z}(-\Delta+\omega)Q_{\omega,\nu}(x-y) {\rm d}x\\
&=\|Q_{\omega,\nu}\|_{H}^{2}+
\|z\|_{H^1(\R^2,\mathbb{C})}^{2}+2\, \mathfrak{Re}\int_{\R^2}e^{i\theta}\bar{z}\left[\ln(|\cdot|\ast |Q_{\omega,\nu}|^2)Q_{\omega,\nu}
 -\nu|Q_{\omega,\nu}|^{p-1}Q_{\omega,\nu}\right](x-y){\rm d}x.
\endaligned
$$
A direct computation yields for $y^*\in\R^2,\theta^*\in\R^2$,
\begin{equation}\lab{diff1}
{\frac{\partial\tilde{f}(y^*,\theta^*)}{\partial y_j}}
=2\, \mathfrak{Re}\int_{\R^2}e^{i\theta^*}\bar{z}(x+y^*)\Theta_{j,\nu}(Q_{\omega,\nu}){\rm d}x=0
\end{equation}
and
\begin{equation}\lab{diff2}
{\frac{\partial \tilde{f}(y^*,\theta^*)}{\partial \theta}}
=2\, \mathfrak{Im}\int_{\R^2}e^{i\theta^*}\bar{z}(x+y^*)(-\Delta+\omega)Q_{\omega,\nu} {\rm d}x=0.
\end{equation}
Since $e^{i\theta^*}\bar{z}(\cdot+y^*)=\bar{w}+Q_{\omega,\nu}$, equations (\ref{diff1}) and (\ref{diff2}) imply that
 $(u,\Theta_{j,\nu}(Q_{\omega,\nu}))_{L^2}=0$ and $(v,Q_{\omega,\nu})_{H^1}=0$.
By the functional
 $$
 \mathcal{S_{\omega,\nu}}(z)=\mathcal{E}_\nu(z)+\omega \mathcal{F}(z),\quad z\in X(\R^2,\mathbb{C}),
 $$
defined in \eqref{1.3.1}, it is straightforward to see that
 $\mathcal{S}'_{\omega,\nu}(Q_{\omega,\nu})=0$.
From the translation invariance of $S_{\omega,\nu}$ and the expression of $w$, the Taylor expansion of $\mathcal {S} _{\omega , \nu}$ gives
$$
\aligned
\mathcal{E}_\nu(z)-& \mathcal{E}_\nu(Q_{\omega,\nu})=\mathcal{E}_\nu(z)-\mathcal{E}_\nu(e^{i\theta^*}Q_{\omega,\nu}(\cdot-y^*))\\
&=\mathcal{S_{\omega,\nu}}(e^{i\theta^*}w(\cdot-y^*)+e^{i\theta^*}Q_{\omega,\nu}(\cdot-y^*))-\mathcal{S_{\omega,\nu}}(Q_{\omega,\nu})\\
&=\mathcal{S_{\omega,\nu}}(w+Q_{\omega,\nu})-\mathcal{S_{\omega,\nu}}(Q_{\omega,\nu})=\frac{1}{2}\left (\mathcal{S_{\omega,\nu}''}(Q_{\omega,\nu}+\theta w)w,w\right )_{L^2}
\endaligned
$$
for $\theta\in(0,1)$. Thus, for any $\phi,\psi\in \tilde{X}$, it is deduced that
\begin{align}
&\left (S_{\omega,\nu}''(\phi)\psi,\psi\right )_{L^2}\nonumber \\
&=\frac{1}{2}\int_{\R^2}\left ( \nabla \psi\cdot\overline{\nabla \psi}+\nabla \psi\cdot\overline{\nabla \psi}\right ) {\rm d}x
+\omega\int_{\R^2}|\psi|^2{\rm d}x
-\int_{\R^2}\left (\ln\frac{1}{|x|}\ast|\phi|^2\right )|\psi|^2{\rm d}x\nonumber \\
&\quad-\nu\int_{\R^2}|\phi|^{p-1}|\psi|^2{\rm d}x-\nu(p-1)\int_{\R^2}|\phi|^{p-3}|\mathfrak{Re} (\psi\overline{\phi})|^2{\rm d}x\nonumber \\
&\quad-\frac{1}{2}\int_{\R^4}\ln\left(\frac{1}{|x-y|}\right)\left [\phi(x)\overline{\psi(x)}+\overline{\phi(x)}\psi(x)\right ]
\left [\phi(y)\overline{\psi(y)}+\overline{\phi(y)}\psi(y)\right ]{\rm d}x{\rm d}y\nonumber \\
%&\quad-\frac{1}{2}\int_{\R^2}\frac{[\phi(x)\overline{v(x)}+\overline{\phi(x)}v(x)][\phi(y)\overline{v(y)}+\overline{\phi(y)}v(y)]}{|x - y|}\\
%&=\int_{\R^2}|\nabla v|^2 + 2\int_{\R^2}|v|^2 - 2\int_{\R^2}\frac{|\phi(x)|^2|v(y)|^2}{|x - y|} - \int_{\R^2}\frac{[\phi(x)\overline{v(x)}
%+ \overline{\phi(x)}v(x)]\cdot[\phi(y)\overline{v(y)} + \overline{\phi(y)}v(y)]}{|x - y|} \\
&= \|\psi\|_{H^1(\R^2,\mathbb{C})}^2-\int_{\R^2}\left (\ln\frac{1}{|x|}\ast|\phi|^2\right )|\psi|^2{\rm d}x
-\nu\int_{\R^2}|\phi|^{p-1}|\psi|^2{\rm d}x-\nu(p-1)\int_{\R^2}|\phi|^{p-3}|\mathfrak{Re} (\psi\overline{\phi})|^2{\rm d}x \nonumber \\
&\quad- 2\int_{\R^2}\ln\left(\frac{1}{|x-y|}\right)[\mathfrak{Re} \phi(x)\mathfrak{Re} \psi(x)+
\mathfrak{Im} \phi(x)\mathfrak{Im} \psi(x)][\mathfrak{Re} \phi(y)\mathfrak{Re} \psi(y)+\mathfrak{Im} \phi(y)\mathfrak{Im} \psi(y)] {\rm d}x{\rm d}y\nonumber \\
&=\|\psi\|_{H^1(\R^2,\mathbb{C})}^2-\int_{\R^2} \left (\ln\frac{1}{|x|}\ast|\phi|^2\right )|\psi|^2{\rm d}x
- 2\int_{\R^2}\left (\ln\frac{1}{|x|}\ast\mathfrak{Re} \phi\mathfrak{Re} \psi \right )\mathfrak{Re} \phi\mathfrak{Re} \psi{\rm d}x\nonumber \\
&\quad -
2\int_{\R^2}\left (\ln\frac{1}{|x|}\ast\mathfrak{Im} \phi\mathfrak{Im} \psi\right )\mathfrak{Im} \phi\mathfrak{Im} \psi{\rm d}x
-4\int_{\R^2}\left (\ln\frac{1}{|x|}\ast\mathfrak{Re} \phi\mathfrak{Re} \psi \right )\mathfrak{Im} \phi\mathfrak{Im} \psi{\rm d}x\nonumber \\
&\quad-\nu\int_{\R^2}|\phi|^{p-1}|\psi|^2{\rm d}x-\nu(p-1)\int_{\R^2}|\phi|^{p-3}
\left(\mathfrak{Re} \phi\mathfrak{Re} \psi+\mathfrak{Im} \phi\mathfrak{Im} \psi\right)^2{\rm d}x. \lab{jisuan}
\end{align}
which gives that
\begin{align}
2  \big(\mathcal{E}_\nu(z)-& \mathcal{E}_\nu(Q_{\omega,\nu}) \big  )\nonumber \\
=&\|w\|_{H^1(\R^2,\C)}^2+\int_{\R^2}\big (\ln|x|\ast|Q_{\omega,\nu}+\theta w|^2\big )|w|^2{\rm d}x \nonumber \\
& +2\int_{\R^2}
\big (\ln|x|\ast(Q_{\omega,\nu}+\theta u)u\big )(Q_{\omega,\nu}+\theta u)u{\rm d}x\nonumber \\
&+2\theta^2\int_{\R^2}\big (\ln|x|\ast v^2\big )v^2{\rm d}x+4\theta\int_{\R^2}\big (\ln|x|\ast(Q_{\omega,\nu}+\theta u)u\big )v^2{\rm d}x \nonumber \\
& -\nu\int_{\R^2}|Q_{\omega,\nu}+\theta w|^{p-1}|w|^2{\rm d}x
-{(p-1)}\nu\int_{\R^2}|Q_{\omega,\nu}+\theta w|^{p-3}\big (uQ_{\omega,\nu}+\theta u^2+\theta v^2\big )^2{\rm d}x\nonumber \\
=&(L_{+,\nu}u,u)_{L^2} +(L_{-,\nu}v,v)_{L^2}+2\theta\int_{\R^2}(\ln|x|\ast Q_{\omega,\nu} u)|w|^2{\rm d}x+\theta^2\int_{\R^2}\big (\ln|x|\ast |w|^2\big )|w|^2{\rm d}x\nonumber \\
&+4\theta\int_{\R^2}(\ln|x|\ast Q_{\omega,\nu} u)u^2{\rm d}x+2\theta^2\int_{\R^2}\big (\ln|x|\ast u^2\big )u^2{\rm d}x+2\theta^2\int_{\R^2}\big (\ln|x|\ast v^2\big )v^2{\rm d}x\nonumber \\
&+4\theta\int_{\R^2}(\ln|x|\ast Q_{\omega,\nu} u)v^2{\rm d}x+4\theta^2\int_{\R^2}\big (\ln|x|\ast u^2\big )v^2{\rm d}x\nonumber \\
&+\nu\int_{\R^2}\Big  (Q_{\omega,\nu}^{p-1}v^2+pQ_{\omega,\nu}^{p-1}u^2
-|Q_{\omega,\nu}+\theta w|^{p-1}|w|^2  \nonumber   \\
& \qquad\qquad \qquad -(p-1)|Q_{\omega,\nu}+\theta w|^{p-3}(uQ_{\omega,\nu}   +\theta u^2+\theta v^2)^2\Big  ){\rm d}x  \nonumber \\
=:&(L_{+,\nu}u,u)_{L^2}+(L_{-,\nu}v,v)_{L^2}+I_1+I_2+I_3+I_4+I_5+I_6+I_7+I_8,
%&\geq D\|w\|_{X(\R^2,\mathbb{C})}^{2}-D_1\|w\|_{X(\R^2,\mathbb{C})}^{4}-D_2\|w\|_{X(\R^2,\mathbb{C})}^{3}.
\lab{diff2+}
\end{align}
where $I_1:=2\theta\int_{\R^2}(\ln|x|\ast Q_{\omega,\nu} u)|w|^2{\rm d}x$,
$$
\aligned
&I_2:=\theta^2\int_{\R^2}(\ln|x|\ast w^2)w^2{\rm d}x,\quad
&I_3:=4\theta\int_{\R^2}(\ln|x|\ast Q_{\omega,\nu} u)u^2{\rm d}x,\\
&I_4:=2\theta^2\int_{\R^2}(\ln|x|\ast u^2)u^2{\rm d}x,\quad
&I_5:=2\theta^2\int_{\R^2}(\ln|x|\ast v^2)v^2{\rm d}x,\\
 &I_6:=4\theta\int_{\R^2}(\ln|x|\ast Q_{\omega,\nu} u)v^2{\rm d}x, \quad
&I_7:=4\theta^2\int_{\R^2}(\ln|x|\ast u^2)v^2{\rm d}x,
\endaligned
$$
and
\begin{align*}
I_8:=& \nu\int_{\R^2}\Big (Q_{\omega,\nu}^{p-1}v^2+pQ_{\omega,\nu}^{p-1}u^2
-|Q_{\omega,\nu}+\theta w|^{p-1}|w|^2\\
&\qquad\qquad -(p-1)|Q_{\omega,\nu}+\theta w|^{p-3}
(uQ_{\omega,\nu}+\theta u^2+\theta v^2)^2\Big ){\rm d}x.
\end{align*}

To estimate $I_i$, $i=1,...,8$, first observe that
$$
I_1:=2\theta\int_{\R^2}(\ln|x|\ast Q_{\omega,\nu} u)|w|^2{\rm d}x=2\theta\int_{\R^2}(\ln|x|\ast|w|^2)Q_{\omega,\nu} u{\rm d}x.
$$
The Sobolev embedding theorem yields $|w|^2\in L^p(\R^2)$ for any $p\geq 2$ because $|w|\in X$.
Define $\mathcal{U}:=\ln|x|\ast|w|^2$ as the Newtonian potential of $|w|^2$. Then,
applying elliptic regularity theory \cite[Theorem 9.9]{Gilbarg98},
it is deduced that $\mathcal{U}\in W_{loc}^{2,p}(\R^2)$,  $\Delta \mathcal{U}=|w|^2$ holds a.e. in $\R^2$,
%%%
%\footnote{[ZWN241022] What do you mean by a.e.? almost everywhere?
%See the LaTeX file for my Chinese: ????? ³ýÈ¥Áã²âŒ¯ÒÔÍâ£¿žÅÂÊÎÊÌâ£¿
%}
%%%
and
$$
\|\mathcal{U}\|_{W_{loc}^{2,p}(\R^2)}\leq C\||w|^2\|_p=C\|w\|^2_{L^{2p}(\R,\mathbb{C})},
$$
where $C$ depends on $p$ only. Hence, by Sobolev embedding theorem again, it is found that
$\mathcal{U}\in C_{loc}^{1,\beta}(\R^2)$ for some $0<\beta<1$ and
\begin{equation}\lab{diff3}
\|\mathcal{U}\|_ {C_{loc}^{1,\beta}(\R^2)}\leq C\|w\|^2_{L^{2p}(\R,\mathbb{C})}.
\end{equation}
From \cite[Proposition 2.3(ii)]{Cingolani16}, one obtains that there exists  $R_1>0$  large such that
$$
|\mathcal{U}(x)|\leq \||w|\|_2^2|\ln |x||+\|u\|_2^3,
~~~\mbox{if } |x|\geq R_1.
$$
It is then obtained from H\"{o}lder's and {Sobolev} inequalities and (\ref{diff3}) that
\begin{equation}\lab{diff4}
\aligned
|I_1|&\leq 2\theta\left|\int_{|x|\leq R_1}(\ln|x|\ast|w|^2)Q_{\omega,\nu} u{\rm d}x\right|
+2\theta\left|\int_{|x|\geq R_1}(\ln|x|\ast|w|^2)Q_{\omega,\nu} u{\rm d}x\right|\\
&\leq C\|w\|_{H^1(\R^2,\mathbb{C})}^3+C\|w\|_{H^1(\R^2,\mathbb{C})}^4.
\endaligned
\end{equation}
Moreover, Sobolev embedding theorem yields $u^2\in L^p(\R^2)$ and $v^2\in L^p(\R^2)$ for any $p\geq 2$
due to $u\in X$ and $v\in X$. Applying a similar argument to the estimate of $I_1$,
%%%
%\footnote{[ZWN241022] revised. Please note the uses of "a similar ??? to" and "the same ??? as"
%}
%%%
it is deduced that
\begin{equation}\lab{diff+}
|I_3|+|I_6|\leq C\|w\|_{H^1(\R^2,\mathbb{C})}^3+C\|w\|_{H^1(\R^2,\mathbb{C})}^4.
\end{equation}
In addition,  the Hardy-Littleword-Sobolev inequality implies that
\begin{equation}\lab{diff5}
\aligned
\theta^{-2} & (I_2+  I_4+I_5+I_7)\\
=& \int_{\R^2}(\ln|x|\ast |w|^2)|w|^2{\rm d}x+2\int_{\R^2}(\ln|x|\ast u^2)u^2{\rm d}x\\
&\qquad\qquad +2\int_{\R^2}(\ln|x|\ast v^2)v^2{\rm d}x+4\int_{\R^2}(\ln|x|\ast u^2)v^2{\rm d}x\\
=& \int_{\R^4}\ln(1+|x-y|)h(x,y){\rm d}x{\rm d}y-\int_{\R^4}\ln \left (1+\frac{1}{|x-y|}\right )h(x,y){\rm d}x{\rm d}y\\
\geq&-\int_{\R^4}\frac{1}{|x-y|}|h(x,y)|{\rm d}x{\rm d}y
\geq-C\|w\|_{H^1(\R^2,\mathbb{C})}^4,\\
\endaligned
\end{equation}
where  $|w|^2=u^2+v^2$ and
$$
h(x,y):=|w(x)|^2|w(y)|^2+2u^2(x)u^2(y)+2v^2(x)v^2(y)+4u^2(x)v^2(y)\geq0.
$$
Since $(uQ_{\omega,\nu}+\theta u^2+\theta v^2)^2\leq |Q_{\omega,\nu}+\theta w|^2|w|^2$ holds for all $x\in\R^2$,
 the {mean value theorem }and H\"{o}lder's inequality give that
{\begin{equation}\lab{diff6}
\aligned
I_8\geq
 &-\nu(p-1)\int_{\R^2}\left ( |Q_{\omega,\nu}+\theta w|^{p-1}|w|^2-Q_{\omega,\nu}^{p-1}u^2\right ) {\rm d}x\\
&-\nu\left|\int_{\R^2}\left (|Q_{\omega,\nu}+\theta w|^{p-1}-Q_{\omega,\nu}^{p-1}\right )|w|^2{\rm d}x\right|\\
\geq& -C\nu(p-1) \int_{\R^2}\left (|Q_{\omega,\nu}|^{p-1}+|w|^{p-1}\right )|w|^2{\rm d}x-C\nu \int_{\R^2}\left (|Q_{\omega,\nu}|^{p-2}+|w|^{p-2}\right )|w|^3{\rm d}x\\
\geq&  -C\nu(p-1)\|w\|_{X(\R^2,\mathbb{C})}^{2}
-C\|w\|_{H^1(\R^2,\mathbb{C})}^3 {-C}\|w\|_{H^1(\R^2,\mathbb{C})}^{p+1}.
\endaligned
\end{equation}
Therefore, by  Proposition \ref{zui},  Lemma \ref{L2-1}, and (\ref{diff2+})-(\ref{diff6})\textcolor[rgb]{0.00,0.00,0.00}{, one has}
$$
\aligned
&\mathcal{E}_\nu(z)-\mathcal{E}_\nu(Q_{\omega,\nu})\\
&\geq(L_+u,u)+(L_-v,v)+(I_2+I_4+I_5+I_7)-(|I_1|+|I_3|+|I_6|)+I_8\\
&\geq (D-C\nu(p-1))\|w\|_{X(\R^2,\mathbb{C})}^{2}-D_1\|w\|_{H^1(\R^2,\mathbb{C})}^{4}-D_2\|w\|_{H^1(\R^2,\mathbb{C})}^{3}
-D_3\|w\|_{H^1(\R^2,\mathbb{C})}^{p+1},
\endaligned
$$
where $D, D_1,D_2,D_3>0$ is independent of $\nu$.}
Let $\tilde{w}(x):=e^{-i\tilde{\theta}}z(x+\tilde{y})-Q_{\omega,\nu}(x)$, where $\tilde{\theta}\in[0,2\pi)$ and $\tilde{y}\in\R^2$
are given in Lemma \ref{lem-inf}.
Then, Lemma \ref{lem-inf} and  $w (x) = e^{-i\theta^*} z (x + y^* ) - Q _{\omega, \nu} (x) $ defined after \eqref{infH} imply that  $\|\tilde{w}\|_{X(\R^2,\mathbb{C})}^2\leq \|w\|_{X(\R^2,\mathbb{C})}^2$.
Moreover,  (\ref{infH}) gives
that $\|w\|_{H^1(\R^2,\mathbb{C})}\leq \|\tilde{w}\|_{H^1(\R^2,\mathbb{C})}$. Thus, for $\nu \geq 0$ small enough and $\epsilon > 0$ small enough,
$$
\aligned
&\mathcal{E}_\nu(z)-\mathcal{E}_\nu(Q_{\omega,\nu})
\geq C \| \tilde w\|_{X(\R^2,\mathbb{C})}^{2}\, ,
\endaligned
$$
which completes the proof of this lemma.
\qed

%%%%%%%%%%%%%%%%%%%%%%%%%%%%%%%%%%

Finally, we are ready to prove Theorem \ref{Th3}.

\textbf{Proof of orbital stability.}
For any fixed $\nu\in[0,\nu^{**})$ with $\nu^{**}$ given in Lemma \ref{lem-xiajie}, we prove the orbital stability of $Q_{\omega, \nu} $ by finding a contradiction if it is false.
 Thus, assume that there exists a sequence of initial data $\psi_n(0)\in X(\R^2,\mathbb{C})$ and
$\delta>0$ such that
\begin{equation}\label{Stab-eq2}
\|\psi_n(0)-
Q_{\omega,\nu}\|_{X(\R^2,\mathbb{C})}\rightarrow 0,
\end{equation}
and
\begin{equation*}
\sup_{t>0}\inf\limits_{y\in\R^2,\theta\in\R}\|\psi_n(t)-e^{i\theta}
Q_{\omega,\nu}(\cdot-y)\|_{X(\R^2,\mathbb{C})}\geq 2\delta,
\end{equation*}
where $\psi_n(t)$ is a solution of \eqref{eqn:q1} with the initial datum $\psi_n(0)$. By the continuity of $\psi _n (t)$ in $t$, we can choose the first time $t_n$ such that
\begin{equation}\label{Stab-eq3}
 \inf\limits_{y\in\R^2,\theta\in\R}\|\psi_n(t_n)-e^{i\theta}
Q_{\omega,\nu}(\cdot-y)\|_{X(\R^2,\mathbb{C})}=\delta.
\end{equation}
By the conservation of energy and mass,
it is deduced from \eqref{Stab-eq2} that
\begin{equation}\label{Stab-eq3+}
 \mathcal{E}_\nu(\psi_n(t_n))=\mathcal{E}_\nu(\psi_n(0))\rightarrow \mathcal{E}_\nu(Q_{\omega,\nu}),
  \quad \|\psi_n(t_n)\|_2=\|\psi_n(0)\|_2\rightarrow \|Q_{\omega,\nu}\|_2.
\end{equation}
Then, it can be shown that $\{\psi_n(t_n)\}$ is bounded in $X(\R^2,\mathbb{C})$. Indeed, by \eqref{eqn:es3}, \re{Stab-eq3+} and the Gagliardo-Nirenberg inequality, there exists $C>0$ such that for large $n$
$$
\aligned
\frac12\|\psi_n(t_n)\|^2_{H^1(\R^2,\mathbb{C})} & +\frac{1}{4}B_1\left (|\psi_n(t_n)|^2,|\psi_n(t_n)|^2\right )\\
&\leq \frac{1}{4}B_2\left (|\psi_n(t_n)|^2,|\psi_n(t_n)|^2\right )+C+\frac{\nu}{p+1}\|\psi_n(t_n)\|_{p+1}^{p+1}\\
&\leq \frac{1}{4}\|\psi_n(t_n)\|_{\frac{8}{3}}^4+C+C {\|\nabla\psi_n(t_n)\|_2^{p-1}}\|\psi_n(t_n)\|_2^{2}\\
&\leq C\|\nabla\psi_n(t_n)\|_{2}\|\psi_n(t_n)\|_{2}^3+C{\|\nabla\psi_n(t_n)\|_2^{p-1}}+C\\
&\leq C\|\psi_n(t_n)\|_{H^1(\R^2,\mathbb{C})}\|\psi_n(t_n)\|_{2}^3+C\|\psi_n(t_n)\|_{H^1(\R^2,\mathbb{C})}^{p-1}+C,
\endaligned
$$
which implies by $p<3$ that $\{\psi_n(t_n)\}$ is bounded in $H^1(\R^2,\mathbb{C})$.
 {Similarly to that in \cite[Proposition 3.1]{Cingolani16},
one uses Lions' Lemma to obtain the existence of $\bar{x}_n\in\R^2$ such that $\bar{\psi}_n(t_n):=\psi_n(t_n)(\cdot-\bar{x}_n)$ converges
weakly in $H^1(\R^2)$ to some function $\bar{\psi}_0\in H^1(\R^2)\setminus\{0\}$ as $n\rightarrow\infty$.
It further follows from Lemma \ref{Lem:cw1} that
$\||\bar{\psi}_n(t_n)|\|_w \leq C$}. Hence, there exists $M>0$ only depending on $\|Q_{\omega,\nu}\|_X$ such that
\begin{equation}\label{Stabb}
{\|\bar{\psi}_n(t_n)\|_{X(\R^2,\mathbb{C})}\leq M.}
\end{equation}
If
$$
u_n:={\frac{\|Q_{\omega,\nu}\|_2}{\|\bar{\psi}_n(t_n)\|_2}\bar{\psi}_n(t_n)}\, ,
$$
then
\begin{equation*}
\|u_n\|_2=\|Q_{\omega,\nu}\|_2,
\quad
\| {\bar{\psi}_n(t_n)}-  u_n\|_{X(\R^2,\mathbb{C})} \rightarrow 0,
%%  \mbox{~~and}~~
\quad
\mathcal{E}_\nu(u_n)\rightarrow \mathcal{E}_\nu(Q_{\omega,\nu}).
\end{equation*}
If $\delta > 0 $ is chosen sufficiently small such that $\delta<\varepsilon$, by  (\ref{Stabb}) and
applying Lemma \ref{lem-xiajie}, it is obtained that
\begin{equation}\label{Stab-2+}
  0\leftarrow \mathcal{E}_\nu(u_n)- \mathcal{E}_\nu(Q_{\omega,\nu})\geq C \inf\limits_{y\in\R^2,\theta\in\R}
  \|e^{-i\theta}u_n(\cdot+y)-Q_{\omega,\nu}\|_{X(\R^2,\mathbb{C})}^2.
\end{equation}
Moreover, Lemma \ref{lem-inf} implies that  there exist  $\tilde{y}_n\in\R^2,\tilde{\theta}_n\in\R$ such that for large
$n$,
\begin{equation}\label{Stab-3}
C\geq \inf\limits_{y\in\R^2,\theta\in\R}
  \|e^{-i\theta}u_n(\cdot+y)-Q_{\omega,\nu}\|_{X(\R^2,\mathbb{C})}^2
  = \|e^{-i\tilde{\theta}_n}u_n(\cdot+\tilde{y}_n)-Q_{\omega,\nu}\|_{X(\R^2,\mathbb{C})}^2.
\end{equation}
Since ${\{\bar{\psi}_n(t_n)\}}$ is bounded in $X(\R^2,\mathbb{C})$,
then $\{u_n\}$ is also bounded in $X(\R^2,\mathbb{C})$, that is, (\ref{Stabb}) also holds for $\{u_n\}$.
From the argument for proof of Lemma \ref{lem-inf}, it is deduced
that $\{\tilde{y}_n\}$ is a bounded sequence in $\R^2$, i.e., $|\tilde{y}_n|\leq C$ for some $C>0$
depending only $\|Q_{\omega,\nu}\|_X$.
By \re{Stab-3} and the azimuthal invariance of $\|\cdot\|_{X(\R^2,\mathbb{C})}$,  it is found that
\begin{equation}\label{Stab-4}
\aligned
\|G_n\|_{X(\R^2,\mathbb{C})}^2:=&\|u_n(\cdot+\tilde{y}_n)-e^{i\tilde{\theta}_n}Q_{\omega,\nu}\|_{X(\R^2,\mathbb{C})}^2\\
=&\|e^{-i\tilde{\theta}_n}u_n(\cdot+\tilde{y}_n)-Q_{\omega,\nu}\|_{X(\R^2,\mathbb{C})}^2\leq C.
\endaligned
\end{equation}
Then, the translation invariance of $\|\cdot\|_{H^1(\R^2,\mathbb{C})}$ yields
\begin{equation}\label{Stab-5}
\|G_n\|_{H^1(\R^2,\mathbb{C})}^2=\|G_n(\cdot-\tilde{y}_n)\|_{H^1(\R^2,\mathbb{C})}^2
=\|u_n-e^{i\tilde{\theta}_n}Q_{\omega,\nu}(\cdot-\tilde{y}_n)\|_{H^1(\R^2,\mathbb{C})}^2.
\end{equation}
Moreover, because of $|\tilde{y}_n|\leq C$, one obtains that
$$
\aligned
\|G_n(\cdot-\tilde{y}_n)\|_{w}^2&=\int_{\R^2}\ln(1+|x|)|G_n(\cdot-\tilde{y}_n)|^2{\rm d}x =\int_{\R^2}\ln(1+|x+\tilde{y}_n|)|G_n|^2{\rm d}x\\
&\leq\int_{\R^2}\ln(1+|x|)|G_n|^2{\rm d}x+\int_{\R^2}\ln(1+C)|G_n|^2{\rm d}x =\|G_n\|_{w}^2+C\|G_n\|_2^2\, ,\\
\endaligned
$$
which implies by (\ref{Stab-4}) and (\ref{Stab-5}) that
\begin{equation}\label{Stab-6}
\|G_n(\cdot-\tilde{y}_n)\|_{X(\R^2,\mathbb{C})}^2\leq C\|G_n\|_{X(\R^2,\mathbb{C})}^2.
\end{equation}
Hence, by (\ref{Stab-2+})-(\ref{Stab-6}), it is deduced that as $n\rightarrow \infty$,
\begin{equation}\label{Stab-7}
\aligned
  0\leftarrow &\mathcal{E}_\nu(u_n)- \mathcal{E}_\nu(Q_{\omega,\nu})\geq C \inf\limits_{y\in\R^2,\theta\in\R}
  \|e^{-i\theta}u_n(\cdot+y)-Q_{\omega,\nu}\|_{X(\R^2,\mathbb{C})}^2\\
  &= \|e^{-i\tilde{\theta}_n}u_n(\cdot+\tilde{y}_n)-Q_{\omega,\nu}\|_{X(\R^2,\mathbb{C})}^2 =\|u_n(\cdot+\tilde{y}_n)-e^{i\tilde{\theta}_n}Q_{\omega,\nu}\|_{X(\R^2,\mathbb{C})}^2\\
  &\geq C\|u_n-e^{i\tilde{\theta}_n}Q_{\omega,\nu}(\cdot-\tilde{y}_n)\|_{X(\R^2,\mathbb{C})}^2\\
  &\geq C{\|\bar{\psi}_n(t_n)-e^{i\tilde{\theta}_n}Q_{\omega,\nu}(\cdot-\tilde{y}_n)\|_{X(\R^2,\mathbb{C})}^2
  -C\|u_n-\bar{\psi}_n(t_n)\|_{X(\R^2,\mathbb{C})}}\\
  &\geq C\inf\limits_{y\in\R^2,\theta\in\R}\|\psi_n(t_n)-e^{i\theta}
Q_{\omega,\nu}(\cdot-y)\|_{X(\R^2,\mathbb{C})}+o_n(1) = C \delta+o_n(1)\, ,
\endaligned
\end{equation}
where the fact that $\|\psi_n(t_n)-u_n\|_{X(\R^2,\mathbb{C})} \rightarrow0$ and equation \eqref{Stab-eq3} have been used, which is a contradiction because $\delta > 0$ is a fixed constant. The stability proof is completed.
\qed

\s{Appendix}

Note that $Q_\omega$ is the { unique positive} solution of $-\Delta u + \omega u = \frac{1}{2 \pi}\left [  \ln{\frac{1}{|x|}}*|u|^2\right ]u$.
Write an operator $F(\nu,u)$ in the sense of polar coordinates as
%$$F(\nu,u):\mathbb{R}^{+}\times Y:=H^{1}_{rad}(\mathbb{R}^2)\cap L^{2}_{rad}(\mathbb{R}^{2},\ln{(1+|x|)}{\rm d}x)
%\to L^{2}_{rad}(\mathbb{R}^{2} )
%$$
%by
%$$
%F(\nu,u) =-\Delta u+\omega u-\left [\ln{\frac{1}{|x|}}*|u|^2\right ]u-\nu |u|^{p-1}u.
%$$
%In terms of radial symmetry, $F$ can be rewritten as
\begin{align*}
F(\nu,u):\mathbb{R}^{+} \times \overline{Y}:= H^{1}(\mathbb{R}^{+},r{\rm d}r)\cap L^{2}_{rad}(\mathbb{R}^{+},{r\ln{(1+r)}}{\rm d}r)\to L^{2}_{rad}(\mathbb{R}^{+},r{\rm d}r ),
\end{align*}
where
\begin{align*}
F(\nu,u)=-u_{rr}-\frac{1}{r}u_{r}+\omega u-\int_{0}^{\infty } u^{2}(s)s\ln{\frac{1}{\max\{r,s\}}{\rm d}s} \cdot u(r)  - \nu |u|^{p-1}u.
\end{align*}
The Fr\'{e}chet derivative of $F$ with respect to $u\in\overline{Y}$,
denoted by $D_{u}F(\nu,u):\overline{Y}\to L_{rad}^2(\R^+,r{\rm d}r)$,
is
\begin{align*}
    D_uF(\nu,u)v&=-v_{rr} - \frac{1}{r} v_{r} + \omega v - \int_{0}^{\infty } u^{2}(s)s\ln{\frac{1}{\max\{r,s\}}{\rm d}s} \cdot v(r) \\
    & - 2\int_{0}^{\infty } u(s)  {v(s)}s\ln{\frac{1}{\max\{r,s\}}{\rm d}s} \cdot u(r) - \nu p|u|^{p-2}uv,\quad\forall v\in\overline{Y}.
\end{align*}
Obviously, $F(0,Q_{\omega})=0$ and $F$ is continuously differentiable in the neighborhood of $(0,Q_w)$.
In particular, the operator $D_{w}F(0,Q_{\omega})$ has only a trivial kernel in $\overline{Y}$, that is, $\mbox{ker}\, [D_{w}F(0,Q_{\omega})]=\left\{0\right\}$.
Let $\mathcal{L}:\overline{Y} \to L^{2}_{rad}((0,+\infty),r{\rm d}r)$ be defined by
\begin{align*}
\mathcal{L}(v):&= - v_{rr} - \frac{1}{r} v_{r} + \omega r + \int_{0}^{\infty} Q^2_{\omega}{(s)}s\ln{(\max\{r,s\})}{\rm d}s \cdot v(r) \\
 &+ 2\int_{0}^{\infty } Q_{\omega}(s)v(s)s\ln{({\max\{r,s\})}{\rm d}s} \cdot Q_{\omega}(r) = D_{w}F(0,Q_{\omega}).
\end{align*}
Therefore, $\mathcal{L}$ is injective. It remains to show that $\mathcal{L}$ is surjective if the operator $\mathcal{L}$ is one-to-one and onto, i.e.,
 for any \(f \in L^{2}(\R^+, r\,{\rm d}r)\setminus\{0\}\), the following equation is solvable
\begin{equation}\label{eqn:or-f}
\begin{aligned}
   \mathcal{L}(v)&= -v_{rr} - \frac{1}{r}v_{r} + \omega v + W v \\
   &+ 2\int_{0}^{\infty} Q_{\omega}(s) v(s) \cdot s \ln{(\max\{r,s\})} \, {\rm d}s \cdot Q_{\omega}(r)=f,
\end{aligned}
\end{equation}
where
\[
{W(r) }:= \int_{0}^{\infty} Q_{\omega}^{2}(s) \cdot s \, \ln{(\max\{r,s\})} \, {\rm d}s,
\]
subject to the boundary conditions
\[
v_{r}(0)=0, \qquad \lim_{r\to +\infty} v(r)=0.
\]

\noindent
\begin{lemma}\label{Lem:compact00}
{$W\in L^\infty_{loc}(\R)$}, and for any $\varepsilon > 0$, there exists $R_{\varepsilon}>0$ such that for $r>R_\epsilon$,
\begin{equation}\label{Wr-0}
\left |{W(r)} - \int_{0}^{+\infty} Q_w^2(s) \cdot s \ln{r}{\rm d}s   \right | < \epsilon.
\end{equation}
Moreover, {there exists $R>0$ such that for $r>R$,}
\begin{align}\label{Wr-01}
\frac{1}{2} \ln{(1+r)}\left \| Q_w \right \|^2_{L^2(\R^+,r{\rm d}r)}<W(r)
<{ 2\ln{(1+r)} \left \| Q_w \right \|^2_{L^2(\R^+,r{\rm d}r)}}.
\end{align}
\end{lemma}
\Proof
Note that
\begin{equation}\label{Wr}
\aligned
 {W(r)}& = \int_0^{+\infty} Q_w^2(s)\cdot s \ln(\max\left \{r,s \right \}){\rm d}s \\
  & = \int_0^{r} Q_w^2(s) \cdot s \ln{r} {\rm d}s + \int_r^{+ \infty} Q_w^2(s) \cdot s \ln{s}{\rm d}s.
\endaligned
\end{equation}
{Since $Q_w$ decays exponentially at infinity and is radially symmetric, and decreases
monotonically along the radial direction,
it is straightforward to check that $W\in L^\infty_{loc}(\R)$.
Moreover, \eqref{Wr} can be rewritten as}
\begin{align}
 {W(r)}-\int_{0}^{+\infty} Q^2_{\omega}(s)\cdot s\ln{r}{\rm d}s
= -\int_{r}^{+\infty} Q^2_{\omega}(s)\cdot s\ln{r}{\rm d}s
+\int_{r}^{+\infty} Q^2_{\omega}(s)\cdot s\ln{s}{\rm d}s.
\end{align}
{Again, by the fact that $Q_w$ decays exponentially at infinity and is radially symmetric,
(\ref{Wr-0}) follows immediately.}
 The proof is completed.
\qed

\begin{lemma}\label{Lem:compact10}
For all $v, w \in \overline{Y}$,
\begin{equation}
\begin{aligned}
   \left | \int_{0}^{\infty} \int_{0}^{\infty} Q_w(s) \, v(s) \cdot s \, \ln(\max\{r,s\})Q_w(r)w(r)r{\rm d}r{\rm d}s  \right |
\le C \, \left \| v \right \| _{L^{2}(\R^+,r{\rm d}r)} \, \left \| w \right \| _{L^{2}(\R^+,r{\rm d}r)}\, .
\end{aligned}
\end{equation}
\end{lemma}

\Proof
Define
$$ {\overline{W} :=} \int_{0}^{\infty} \int_{0}^{\infty} Q_w(s)v(s)s\ln(\max\{r,s\})Q_w(r) \, w(r) \, r \, {\rm d}r \, {\rm d}s\, .
$$
By H{\"o}lder's inequality, it is obtained that
\begin{equation}
\begin{aligned}
|\overline{W}|\le & \int_{0}^{\infty} \int_{0}^{\infty} Q_w(s)v(s)s\ln[(1+r)(1+s)] \cdot Q_w(r) w(r) \, r \, {\rm d}s \, {\rm d}r \\
 =& \left[ \int_{0}^{\infty}\int_{0}^{\infty} Q_w(r)w(r)rQ_w(s)v(s) s \ln(1+s) {\rm d}s{\rm d}r \right]\\
 &+\left[ \int_{0}^{\infty}\int_{0}^{\infty} Q_w(s)v(s)sQ_w(r)w(r)r \ln(1+r) {{\rm d}s}{\rm d}r \right] \\
\le& \left( \int_{0}^{\infty} |Q_w(s)|^2 s [\ln(1+s)]^2{\rm d}s \right)^{\frac12}\|Q_w\|_{L^{2}(\R^+,r{\rm d}r)}\|v\|_{L^{2}(\R^+,r{\rm d}r)}\|w\| _{L^{2}(\R^+,r{\rm d}r)}\\
\le& C\cdot\|v\| _{L^{2}(\R^+,r{\rm d}r)}\|w\| _{L^{2}(\R^+,r{\rm d}r)}.
\end{aligned}
\end{equation}
The proof is completed. \qed

\begin{proposition}\label{Lem:manifo3}
$\mathcal{L}:\overline{Y} \to L^{2}_{rad}((0,+\infty),r{\rm d}r)$ is surjective.
\end{proposition}

\Proof
Consider a modified equation
\begin{align}
    \mathcal{L}_\lambda(v) := \mathcal{L}(v) + \lambda v = f,\quad {\lambda\in\R} \label{*}
\end{align}
which will be solved in the sense of distribution.
By Lemmas \ref{Lem:compact00}-\ref{Lem:compact10} and choosing $\lambda > 0 $ sufficiently large, for any $ v \in \overline{Y}$, there exists $R>0$ such that
 \begin{equation}\label{eqn:below}
\aligned
    & \left \langle \mathcal{L}_{\la}(v),v\right \rangle_{L^2(\R^+,r{\rm d}r)} \\
    & \quad = \left \langle -v_{rr} - \frac{1}{r}v_{r} + \omega v , v\right \rangle_{L^2(\R^+ ,r{\rm d}r)} \\
    & \qquad + \left \langle Wv + 2\int_{0}^{\infty} Q_{\omega}(s) v(s) \cdot s \ln{(\max\{r,s\})}
    {\rm d}s \cdot Q_{\omega}(r) + \lambda v,v\right \rangle_{L^2(\R^+,r{\rm d}r)} \\
    & \quad = \int_{0}^{\infty} |v_{r}|^{2} \, r{\rm d}r + \int_{0}^{\infty}(\lambda+\omega) v^{2}r{\rm d}r \\
    & \qquad + \left \langle Wv + 2\int_{0}^{\infty} Q_{\omega}(s) v(s)s \ln{(\max\{r,s\})}
    {\rm d}s \cdot Q_{\omega}(r),v\right \rangle_{L^2(\R^+,r{\rm d}r)} \\
 & \quad \ge \frac{1}{2}\| v\| ^{2}_{\overline{Y}} +(\lambda-C)\|v\|^2_{L^2(\R^+,r{\rm d}r)}
    -\int_{0}^{R}Wv^2r{\rm d}r\\
    & \quad \ge \frac{1}{2}\|v\| ^{2}_{\overline{Y}} + (\lambda-C)  \| v \| ^{2}_{L^2(\R^+,r{\rm d}r)}
    -\|W\|_{L^\infty[0,R]}\int_{0}^{R}v^2r{\rm d}r  \ge \frac{1}{2} \|v\| ^{2}_{\overline{Y}},
\endaligned
\end{equation}
which yields that the bilinear form $\langle L_{\lambda} \, \cdot \, , \,\cdot \rangle$ is coercive.
Moreover, by Lemmas \ref{Lem:compact00}-\ref{Lem:compact10}, it is deduced that for any $v,w \in \overline{Y}$,
\begin{equation}\label{eqn:upper}
\aligned
& \left \langle \mathcal{L}_{\la}(v),w \right \rangle_{L^2(\R^+,r{\rm d}r)} \\
& \quad = \left \langle -v_{rr} - \frac{1}{r}v_{r} + \omega v + Wv, w \right \rangle_{L^2(\R^+,r{\rm d}r)} \\
& \qquad + \left \langle 2\int_{0}^{\infty} Q_{\omega}(s) v(s) \cdot s \ln{(\max\{r,s\})} \, {\rm d}s \cdot Q_{\omega}(r)
+ \lambda v , w \right \rangle_{L^2(\R^+,r{\rm d}r)} \\
& \quad \le \left \langle v_{r} , w_{r} \right \rangle_{L^2(\R^+,{\rm d}r)}
+\omega \left \langle v , w \right \rangle  _{L^2(\R^+,r{\rm d}r)}
+ 2\left \| Q_{\omega} \right \|_{2}^{2} \, \left | \left \langle v \ln (1+r) , w \right \rangle _{L^2(\R^+,r{\rm d}r)} \right | \\
& \qquad + \lambda \left \langle v , w \right \rangle _{L^2(\R^+,r{\rm d}r)} + \left \langle 2\int_{0}^{\infty} Q_{\omega}(s) v(s)
\cdot s \ln{(\max\{r,s\})} \, {\rm d}s Q_{\omega}(r) , w \right \rangle_{L^2(\R^+,r{\rm d}r)} \\
& \quad \le (  {2\left \|Q_{\omega}\right \|_{2}^{2}}+1)\|v\|_{\overline{Y}}\|w\|_{\overline{Y}}
+(\lambda + C)\|v\|_{\overline{Y}}\|w\|_{\overline{Y}},
\endaligned
\end{equation}
which shows that  {$\left \langle\mathcal{L}_{\lambda} \, \cdot \, , \,\cdot \, \right \rangle_{L^2(\R^+,r{\rm d}r)}$} is bounded.
Consequently, by the Lax-Milgram theorem, the modified problem \eqref{*} has a unique solution $v^{*} \in \overline{Y}$.
Therefore, $v^{*}=\mathcal{L}_{\lambda}^{-1}f$,
where linear operator $\mathcal{L}_{\lambda}^{-1}: L^{2}_{rad}(\mathbb{R}^{+},r{\rm d}r)\rightarrow \overline{Y}$ is bounded.
Moreover, for any $\tilde{f}\in L^{2}_{rad}(\mathbb{R}^{+},r{\rm d}r)$, there exists a unique $\tilde{v}\in \overline{Y}$ such that
$\tilde{v}=\mathcal{L}_{\lambda}^{-1}\tilde{f}$.
From (\ref{eqn:below}) and (\ref{eqn:upper}), define new norm in $\overline{Y}$ by $\|\cdot\|_{\lambda}
=\{\left \langle \mathcal{L}_\lambda\cdot,\cdot\right \rangle_{L^2(\R^+,r{\rm d}r)}\}^{1/2}$ endowed with the corresponding inner product
\begin{equation}\label{eqn:inn-}
\left \langle\cdot,\cdot\right \rangle_{\lambda}:=\left \langle \mathcal{L}_\lambda\cdot,\cdot\right \rangle_{L^2(\R^+,r{\rm d}r)}.
\end{equation}
It is obtained from (\ref{eqn:below}) and (\ref{eqn:upper}) that $\|\cdot\|_{\lambda}$ is
equivalent to $\|\cdot\|_{\overline{Y}}$. The original equation can be rewritten as
$$
\left \langle v,w\right \rangle_{\lambda}=\left \langle \mathcal{L}_\lambda^{-1}f,w\right \rangle_{\lambda}
+\lambda\left \langle \mathcal{L}_\lambda^{-1}v,w\right \rangle_{\lambda}, \quad{\forall w\in \overline{Y}},
$$
which is equivalent to the following equation in the weak sense of $\overline{Y}$
\begin{equation}\label{eqn:oper-eq}
(I-\lambda \mathcal{L}_\lambda^{-1})v=\mathcal{L}_\lambda^{-1}f.
\end{equation}
By the Riesz-Fredholm theorem, equation (\ref{eqn:oper-eq}) is solvable if and only if
\begin{equation}\label{eqn:zhenjiao}
\mathcal{L}_\lambda^{-1}f\perp \mbox{ker}\, (I-\lambda \mathcal{L}_\lambda^{-1}).
\end{equation}
Thus, if $v\in \mbox{ker}\, (I-\lambda \mathcal{L}_\lambda^{-1})$, then $v=\lambda \mathcal{L}_\lambda^{-1}v$ in the weak sense.
That is, for any $w\in \overline{Y}$
$$
\left \langle v,w\right \rangle_{\lambda}=\lambda\left \langle \mathcal{L}_\lambda^{-1}v,w\right \rangle_{\lambda},
\quad{\forall w\in \overline{Y}},
$$
which implies by the definition of $\mathcal{L}_\lambda^{-1}$ that
$\left \langle v,w\right \rangle_{\lambda}=\lambda\left \langle v,w\right \rangle_{L^2(\R^+,r{\rm d}r)}$ for any $w\in\overline{Y}$. Hence,
 (\ref{eqn:inn-}) yiel{\rm d}s $\mathcal{L}(v)=0$.
Since operator $\mathcal{L}=D_{w}F(0,Q_{\omega})$ has only a trivial kernel in $\overline{Y}$,  $v=0$ and
(\ref{eqn:zhenjiao}) always holds. This yields that equation (\ref{eqn:oper-eq}) is solvable.
Therefore, the original equation (\ref{eqn:or-f}) has a unique nontrivial
weak solution $v_{*} \in \overline{Y}$. This completes the proof.
\qed

Consequently, operator $\mathcal{L}$ is one-to-one and onto, which implies by the open mapping theorem that the inverse
$\mathcal{L}^{-1}:L_{rad}^2(\R^2)\to \tilde{X}_{rad}$ is continuous.

{Next, an integral estimate is given in order to prove a uniform convergence
with respect to $z\in \Omega_{Q_{\omega,\nu}}^{\e}$ with $\|z\|\leq M$ in Lemma \ref{lem-inf} when $|y|\rightarrow\infty$.
\begin{lemma}\label{Lem:int0}
For $y\in\R^2$, there exists $C>0$ independent of $y$ such that
\begin{equation}\lab{int1}
\int_{\mathbb{R}^2} \frac{e^{-\frac{1}{4}|x-y|}}{1+\ln(1+|x|)} dx
    \leq \frac{C}{1+\ln(1+|y|)}.
\end{equation}
\end{lemma}
\Proof We use polar coordinates to verify the above integral estimate. Without loss of generality, let $|y|\geq 1$.
Define\textcolor[rgb]{0.00,0.07,1.00}{
\[
x=y+r\omega,\quad r>0,\;\omega = (\cos \theta , \sin \theta ) \in\mathbb{S}^1,\quad dx=r\,dr d\theta.\]
%%
%\footnote{[ZWN] revised from ``, then''}
%%
Then $|x-y|=r$ and$|x|=|y+r\omega|\geq ||y|-r|$}. The left side of (\ref{int1}) can be estimated by
\begin{align*}
    \Lambda(y):=& \int_0^\infty\int_0^{2\pi} \frac{e^{-\frac14 r}}{1+\ln(1+|y+r\omega|)}\,r\,d\theta  dr
    \leq 2\pi  \int_0^\infty \frac{e^{-\frac14 r}}{1+\ln(1+| |y| - r |)}\,r\, dr \, \\
 =  & 2\pi  \left ( \int_0^{|y|/2} +  \int_{|y|/2} ^\infty \right ) \frac{e^{-\frac14 r}}{1+\ln(1+| |y| - r |)}\,r\, dr \\
 \leq &  C  \left ( \int_0^{|y|/2}  \frac{e^{-\frac14 r}}{1+\ln(1+|y|/2)}\,r\, dr  +
 \int_{|y|/2} ^{\infty}  {e^{-\frac14 r}}\,r\, dr  \right ) \\
  \leq &  C  \left (  \frac{1}{1+\ln(1+|y|/2)} + e^{- |y|/16} \right ) \leq \frac{C}{1+\ln(1+|y|)}    \, ,\end{align*}
\textcolor[rgb]{0.00,0.07,1.00}{which is  the right hand side
of (\ref{int1}).  The proof is complete.}\qed}
%%%
%\footnote{[ZWN] revised from ``right side''. We need to check this for the whole paper.
%}
%%%
%%
%\footnote{[ZWN] Better not finish with the end mark of a proof. Can we add sentences for the end?
%}
%%

%%%%%%%%%%%%%%%%%%%%%%%%%%%%%%%%%%%%%%%%%%%%%%%%%%%%%%%%%%%%%%%%%%%%%%%%%%

\subsection*{Declaration of competing interest.}
The authors declare that they have no known competing financial interests or personal relationships
that could have appeared to influence the work reported in this paper.

%%%%%%%%%%%%%%%%%%%%%%%%%%%%%%%%%%%%%%%%%%%%%%%%%%%%%%%%%%%%%%%%%%%%%%%%%%%%
%\footnote{[ZWN241022] Better to list all references in alphabetic order of authors' family names.
%  }
%%

%%
%\footnote{[ZWN241022] In my address my zip code is 610065}
%%


\begin{thebibliography}{10}


\bibitem{Azzollini21} A. Azzollini, The planar Schr\"{o}dinger-Poisson system with a positive potential.
                                               {\it Nonlinearity}, {\bf 34} (2021),  5799--5820.


\bibitem{Albuquerque21} F. Albuquerque, J. Carvalho, G. Figueiredo, E. Medeiros, On a planar non-autonomous
                                 Schr\"{o}dinger-Poisson system involving exponential critical growth.
                                               {\it Calc. Var. Partial Differential Equations}, {\bf 60} (2021),
                                                       Paper No. 40, 30 pp.



\bibitem{Bellazzini11} J. Bellazzini, G. Siciliano, Stable standing waves for a class of nonlinear Schr\"{o}dinger-Poisson
                                       equations. {\it Z. Angew. Math. Phys.}, {\bf 62} (2011), 267--280.


\bibitem{Bellazzini13} J. Bellazzini, L. Jeanjean, T. Luo, Existence and instability of standing waves with prescribed norm
for a class of Schr\"{o}dinger-Poisson equations. {\it Proc. Lond. Math. Soc.}, {\bf 107} (2013), 303--339.




\bibitem{Benci98} V. Benci, D. Fortunato, An eigenvalue problem for the Schr\"{o}dinger-Maxwell equations,
                                       {\it Topol. Methods Nonlinear Anal.}, {\bf 11} (1998), 283--293.



\bibitem{Berestycki83} H. Berestycki, P.-L. Lions, Nonlinear scalar field equations. I. existence of a ground state,
{\it Arch. Ration. Mech. Anal.} {\bf 82} (1983), 313-345.


\bibitem{Berestycki81} H. Berestycki,  T. Cazenave, Instabilit\'{e} des \'{e}tats stationnaires dans les \'{e}quations
de Schr\"{o}dinger et de Klein-Gordon non lin\'{e}aires, {\it C. R. Acad. Sci. Paris.}, {\bf 293} (1981), 489--492.



\bibitem{Berezin91} F. Berezin, M. Shubin, {\it The Schr\"{o}dinger Equation},
Translated from the 1983 Russian edition by Yu. Rajabov, D.A. Leites and N. A. Sakharova and revised by Shubin,
with contributions by G.L. Litvinov and Leites,
Mathematics and its Applications (Soviet Series), vol. 66, Kluwer Academic Publishers Group, Dordrecht, (1991), 555pp.


\bibitem{Bonheure17} D. Bonheure, S. Cingolani, and J. Van Schaftingen, The logarithmic Choquard equation:
Sharp asymptotics and nondegeneracy of the groundstate, {\it J. Funct. Anal.}, {\bf 272} (2017), 5255--5281.



\bibitem{Bonheure18} D. Bonheure,  J. Casteras,  E. dos Santos, R. Nascimento, Orbitally stable standing
     waves of a mixed dispersion nonlinear Schr\"{o}dinger equation,
     {\it SIAM J. Math. Anal.}, {\bf 50} (2018), 5027--5071.


\bibitem{Bucur22} C. Bucur,  D. Cassani, C. Tarsi,
                 Quasilinear logarithmic Choquard equations with exponential growth in $\Bbb{R}^N$,
     {\it J. Differential Equations}, {\bf 328} (2022), 261--294.


\bibitem{Cabrera25} O. Cabrera, S. Cingolani, T. Weth, A new framework for Ljusternik-Schnirelmann
                                 theory and its application to planar Choquard equations, arxiv preprint arxiv:2502.18421(2025).


\bibitem{Cassani21} D. Cassani, C. Tarsi, Schr\"odinger-Newton equations in dimension two via a Pohozaev-Trudinger
                                           log-weighted inequality,
                                     {\it Calc. Var. Partial Differential Equations}, {\bf 60} (2021), 31 pp.


\bibitem{Cassani23} D. Cassani, Z. Liu, G. Romani, Nonlocal planar Schr\"odinger-Poisson systems in the fractional Sobolev
                                limiting case, {\it J. Differential Equations}, {\bf 383} (2023), 214--269.


\bibitem{Cassani25} D. Cassani, Z. Liu, G. Romani, Nonlocal planar Schr\"odinger-Poisson systems in $\R^N$:
                               the fractional Sobolev limiting case,
                               {\it Rend. Istit. Mat. Univ. Trieste.}, {\bf 57} (2025), 33 pp.


\bibitem{Cazenave03} T. Cazenave, {\it Semilinear Schr\"odinger Equations},
                               Courant Lect. Notes Math. vol. 10, New York
                                University Courant Institute of Mathematical Sciences, New York, (2003).


\bibitem{Chen20} S. Chen, X. Tang, On the planar Schr\"odinger-Poisson system with the axially symmetric potential,
                                              J. Differential Equations, {\bf 268} (2020), 945--976.


\bibitem{Lions} T. Cazenave, P. Lions, Orbital stability of standing waves for some nonlinear Schr\"odinger equations,
                                {\it Comm. Math. Phys.}, {\bf85} (1982), 549--561.

%\bibitem{Caz} T. Cazenave, {\it An Introduction to Nonlinear Schr\"odinger Equations}, Textos deM\'etodos Matem\'aticos, Vol. 26 (Universidade Federal do Rio de Janeiro, 1996).






\bibitem{Cingolani16} S. Cingolani, T. Weth, On the planar Schr\"{o}dinger-Poisson systems,
                                    {\it Ann. Inst. H. Poincar\'{e} Anal. Non Lin\'{e}aire}, {\bf 33} (2016), 169--197.


\bibitem{Cingolani19} S. Cingolani, T. Weth,
 Trudinger-Moser-type inequality with logarithmic convolution potentials, {\it J. Lond. Math. Soc.},
{\bf 105}(2022), 1897--1935.

\bibitem{Cingolani-Jeanjean} S. Cingolani, L. Jeanjean, Stationary waves with prescribed $L^2$-norm
         for the planar Schr\"{o}dinger-Poisson system, {\it SIAM J. Math. Anal.},
             {\bf 51} (2019), 3533--3568.




\bibitem{Choquard08} P. Choquard, J. Stubbe, M. Vuffray, Stationary solutions of the Schr\"{o}dinger-Newton
                                                  model-an ODE approach, {\it Differential Integral Equations}, {\bf 21} (2008),
                                                                                            665--679.


\bibitem{Colin19} M. Colin,  T. Watanabe, A refined stability result for standing waves of
                                                                the Schr\"{o}dinger-Maxwell system. {\it Nonlinearity},
                                                                 {\bf 32} (2019), no. 10, 3695--3714.



\bibitem{Marco1} P. d'Avenia, M. Squassina, Soliton dynamics for the Schr\"odinger-Newton system,
                   {\it Math. Models Methods Appl. Sci.}, {\bf 24} (2014), 553--572.



%
%
%\bibitem{Marco1} P. d'Avenia, M. Squassina, Soliton dynamics for the Schr\"{o}dinger-Newton system
%                                              {\it  Math. Models Methods Appl. Sci.}, {\bf 24}(2014), 553--572.


 \bibitem{Dolbeault21} J. Dolbeault, R. Frank, L. Jeanjean, Logarithmic estimates for
                               mean-field models in dimension two and the Schr\"{o}dinger-Poisson system.
                                 {\it C. R., Math., Acad. Sci. Paris}, {\bf 359} (2021), 1279--1293.



\bibitem{Du17} M. Du, T. Weth, Ground states and high energy solutions of the planar
                 Schr\"{o}dinger-Poisson system, {\it Nonlinearity}, {\bf 30} (2017), 3492--3515.




\bibitem{Fall14} M. Fall, E. Valdinoci, uniqueness and nondegeneracy of positive solutions
of $(-\triangle)^su+u=u^{p}$ in $\R^N$ when $s$ is close to $1$,
{\it Commun. Math. Phys.}, {\bf 329} (2014), 383--404.


\bibitem{Gilbarg98} D. Gilbarg, N. Trudinger, {\it Elliptic Partial Differential Equations of Second Order}, 2nd Edition,
               Springer, Berlin, (1998).


\bibitem{Guo23} Y. Guo, W. Liang, Y. Li, Existence and uniqueness of constraint minimizers for the planar
                                             Schr\"{o}dinger-Poisson
                                 system with logarithmic potentials, {\it J. Differential Equations},
                                           {\bf 369} (2023), 299--352.


\bibitem{Grillakis87} M. Grillakis, J. Shatah, W. Strauss, Stability theory of solitary waves
                       in the presence of symmetry. I, {\it J. Funct. Anal.}, {\bf 74} (1987), 160--197.





\bibitem{Harrison03} R. Harrison, T. Moroz, K. Tod, A numerical study of the Schr\"{o}dinger-Newton
equation, {\it Nonlinearity}, {\bf 16} (2003), 101--122.


\bibitem{Kikuchi07} H. Kikuchi, Existence and stability of standing waves for Schr\"{o}dinger-Poisson-Slater equation,
                                  {\it Adv. Nonlinear Stud.},   {\bf 7} (2007), 403--437.


\bibitem{Kwong89} M. Kwong, Uniqueness of positive solutions of $\DD u-u+u^p=0$ in $\RN$,
                {\it Arch. Rational Mech. Anal.}, {\bf 105} (1989), 243--266.


\bibitem{Lenzmann09} E. Lenzmann, Uniqueness of ground states for pseudorelativistic Hartree equations,
                {\it Anal. PDE.}, {\bf 2} (2009), 1--27.




%\bibitem{Lin22} Z. Lin, C. Zeng, Instability, index theorem, and exponential trichotomy for linear Hamiltonian PDEs,
%                                               {\it Mem. Amer. Math. Soc.}, {\bf 275} (2022), no. 1347, v+136 pp.



\bibitem{Liu22} Z. Liu, V. R\v{a}dulescu, C. Tang, J. Zhang, Another look at planar Schr\"odinger-Newton system,
                                  {\it J. Differential Equations}, {\bf 328} (2022), 65--104.


\bibitem{Liu23} Z. Liu, V. R\u{a}dulescu, J. Zhang, A planar Schr\"odinger-Newton system with Trudinger-Moser
               critical growth,  {\em Calc. Var. Partial Differential Equations}, {\bf 62} (2023), 31 pp.



\bibitem{Luca18} B. Luca, J. Van Schaftingen, Groundstates of the Choquard equations with
                                                   a sign-changing self-interaction potential
                                               {\it Z. Angew. Math. Phys.}, {\bf 69} (2018), no.3, Paper No. 86, 16 pp.




\bibitem{Masaki10} S. Masaki, Local existence and WKB approximation of solutions to Schr\"{o}dinger-Poisson
                                              system in the two-dimensional whole space,
                                {\em Comm. Partial Differential Equations}, {\bf 35} (2010), 2253--2278.



\bibitem{Masaki11} S. Masaki, Energy solution to a Schr\"{o}dinger-Poisson system in the two-dimensional whole space,
                                {\em SIAM J. Math. Anal.}, {\bf 43} (2011), 2719--2731.






\bibitem{Lieb01} E. Lieb, M. Loss, {\em Analysis},  Graduate Studies in Mathematics, vol. 14, American Mathematical Society, Providence, RI, 2001.


%\bibitem{Kera} S. Keraani, Semiclassical limit for nonlinear Schr\"odinger equation with potential II, {\it Asymptot. Anal.}, {\bf47}(2006), 171--186.

\bibitem{lieb} E. Lieb, Existence and uniqueness of the minimizing solution of Choquard's nonlinear equation,
{\it Stud. Appl. Math.}, {\bf 57} (1977), 93--105.


\bibitem{Lieb77} E. Lieb, B. Simon, The Thomas-Fermi theory of atoms, molecules and solids,
                                {\it Adv. Math.}, {\bf 23} (1977) 22--116.



\bibitem{Lions87} P. Lions,  Solutions of Hartree-Fock equations for Coulomb systems,
                                 {\it Commun. Math. Phys.}, {\bf 109} (1987) 33--97.

\bibitem{Pucci24} P. Pucci, L. Wang, B. Zhang, Bifurcation and regularity of entire solutions for the planar nonlinear
                                      Schr\"{o}dinger-Poisson system,
                                 {\it Math. Ann.}, {\bf 389} (2024) 4265--4300.


\bibitem{Shen24} L. Shen, M. Squassina, Planar Schr\"odinger-Poisson system with steep potential well:
                                      critical exponential case,
                                      {\it NoDEA Nonlinear Differential Equations Appl.}, {\bf 32} (2025) Paper No. 15, 31 pp.



\bibitem{Sun24} J. Sun, S. Yao, J. Zhang, Planar Schr\"odinger-Poisson system with exponential critical growth:
                                    local well-posedness and standing waves with prescribed mass,
                                            Stud. Appl. Math. {\bf 153} (2024), Paper No. e12760, 33 pp.



\bibitem{Stubbe08} J. Stubbe, Bound states of two-dimensional Schr\"{o}dinger-Newton equations, (2008), available at 0807.4059v1.

%\bibitem{Tao} T. Tao,  Why are solitions stable?, {\it Bull. Am. Math. Soc.}, {\bf 46}(2009), 1-33.

\bibitem{Wein0} M. Weinstein, Nonlinear Schr\"{o}dinger equations and sharp interpolation estimates.
                            {\it Comm. Math. Phys.}, {\bf 87} (1982/83), 567--576.


\bibitem{Wein1} M. Weinstein, Modulational stability of ground states of nonlinear Schr\"odinger equations,
{\it SIAM J. Math. Anal.}, {\bf 16} (1985), 472--491.

\bibitem{Wein2} M. Weinstein, Lyapunov stability of ground states of nonlinear dispersive evolution equations,
{\it Comm. Pure Appl. Math.}, {\bf 39} (1986), 51--67.


\bibitem{Xiang16} C. Xiang, Uniqueness and nondegeneracy of ground states
for Choquard equations in three dimensions, {\it Calc. Var. Partial Differential Equations}, (2016), 55:134.


\end{thebibliography}
\end{document}